\documentclass[a4paper,11pt,reqno]{amsart}
\pdfoutput=1

\usepackage{amsmath,amssymb,amsthm,mathrsfs}
\usepackage{mathtools}
\usepackage[all,cmtip,2cell]{xy}
\usepackage{bm}
\usepackage{afterpage}

\usepackage[margin=1in]{geometry}

\usepackage{lmodern}
\usepackage[T1]{fontenc}

\usepackage{comment}

\usepackage{bm}
\usepackage{microtype}

\usepackage[UKenglish]{isodate}

\UseAllTwocells

\usepackage[shortlabels]{enumitem}

\swapnumbers

\theoremstyle{plain}

\newtheorem{theorem}[subsection]{Theorem}
\newtheorem{proposition}[subsection]{Proposition}
\newtheorem{lemma}[subsection]{Lemma}
\newtheorem{corollary}[subsection]{Corollary}
\newtheorem*{theorema}{Theorem A}
\newtheorem*{theoremb}{Theorem B}
\newtheorem*{theoremc}{Theorem C}

\theoremstyle{definition} 
\newtheorem{definition}[subsection]{Definition} 
\newtheorem{example}[subsection]{Example}

\theoremstyle{remark}
\newtheorem{remark}[subsection]{Remark}
\newtheorem{recall}[subsection]{Recall}
\newtheorem{observation}[subsection]{Observation}
\newtheorem{notation}[subsection]{Notation}

\makeatletter
\let\c@subsection\c@equation
\makeatother
\numberwithin{equation}{section}

\usepackage[hidelinks]{hyperref}

\newcommand{\lra}{\longrightarrow}
\newcommand{\Cat}{\mathbf{Cat}}
\newcommand{\qcat}{\mathbf{qCat}}

\newcommand{\twocat}{\mathbf{2}\text{-}\mathbf{Cat}}
\newcommand{\twoqcat}{\mathbf{2}\text{-}\mathbf{qCat}}

\newcommand{\Set}{\mathbf{Set}}

\newcommand{\Hom}{\mathrm{Hom}}
\newcommand{\bHom}{\mathbf{Hom}}

\newcommand{\Ho}{\operatorname{Ho}}
\newcommand{\ho}{\operatorname{ho}}
\newcommand{\Bicat}{\operatorname{\mathbf{Bicat}}}
\newcommand{\st}{\mathrm{\mathbf{st}}}

\newcommand{\PCat}{\mathbf{PCat}}

\newcommand{\cd}[2][]{\vcenter{\hbox{\xymatrix#1{#2}}}}

\newcommand{\mydtwocell}[3][0.5]{\ar@{}[#2] \ar@{=>}?(#1)+/u  0.2cm/;?(#1)+/d 0.2cm/^{#3}}
\newcommand{\mydbldtwocell}[4][0.5]{\ar@{}[#2] \ar@{=>}?(#1)+/u  0.2cm/;?(#1)+/d 0.2cm/^{#3}_{#4}}
\newcommand{\myltwocell}[3][0.5]{\ar@{}[#2] \ar@{=>}?(#1)+/r 0.2cm/;?(#1)+/l 0.2cm/_{#3}}
\newcommand{\mydrtwocell}[3][0.5]{\ar@{}[#2] \ar@{=>}?(#1)+/ul  0.2cm/;?(#1)+/dr 0.2cm/^{#3}}

\newcommand{\fatpullbackcorner}[1][dr]{\save*!/#1-1.75pc/#1:(-1,1)@^{|-}\restore}
\newcommand{\fatpushoutcorner}[1][dr]{\save*!/#1+1.75pc/#1:(1,-1)@^{|-}\restore}

\makeatletter
\def\matrixobject@{%
  \edef \next@{={\DirectionfromtheDirection@ }}%
  \expandafter \toks@ \next@ \plainxy@
  \let\xy@@ix@=\xyq@@toksix@
  \xyFN@ \OBJECT@}
\let\xy@entry@@norm=\entry@@norm
\def\entry@@norm@patched{%
  \let\object@=\matrixobject@
  \xy@entry@@norm }
\AtBeginDocument{\let\entry@@norm\entry@@norm@patched}
\makeatother

\newcommand{\hdash}{\rotatebox[origin=c]{90}{$\vdash$}}

\newcommand{\veq}{\rotatebox[origin=c]{90}{$=$}}

\makeatletter
\@namedef{subjclassname@2020}{\textup{2020} Mathematics Subject Classification}
\makeatother

\title{A homotopy coherent cellular nerve for bicategories}
\author{Alexander Campbell}
\address{Centre of Australian Category Theory \\ Macquarie University \\ NSW 2109 \\ Australia}
\email{acampbell@msri.org}
\urladdr{http://web.science.mq.edu.au/~alexc/}


\date{12 April 2020}

\begin{document}

\begin{abstract}
The subject of this paper is a nerve construction for bicategories introduced by Leinster, which defines a fully faithful functor from the category of bicategories and normal pseudofunctors to the category of presheaves over Joyal's category $\Theta_2$. We prove that the nerve of a bicategory is a $2$-quasi-category (a model for $(\infty,2)$-categories due to Ara), and moreover that the nerve functor restricts to the right part of a Quillen equivalence between Lack's model structure for bicategories and a Bousfield localisation of Ara's model structure for $2$-quasi-categories. We deduce that Lack's model structure for bicategories is Quillen equivalent to Rezk's model structure for $(2,2)$-$\Theta$-spaces on the category of simplicial presheaves over $\Theta_2$. 

To this end, we construct the homotopy bicategory of a $2$-quasi-category, and prove that a morphism of $2$-quasi-categories is an equivalence if and only if it is essentially surjective on objects and fully faithful. We also prove a Quillen equivalence between Ara's model structure for $2$-quasi-categories and the Hirschowitz--Simpson--Pellissier model structure for quasi-category-enriched Segal categories, from which we deduce a few more results about $2$-quasi-categories, including a conjecture of Ara concerning weak equivalences of $2$-categories.
\end{abstract}

\maketitle

\tableofcontents

\section{Introduction} \label{secintro}
A fundamental fact of the theory of quasi-categories \cite{MR1935979,joyalbarcelona,MR2522659} is that Grothen\-dieck's nerve construction (introduced in \cite{grothendiecknerve}) defines a fully faithful functor 
\begin{equation} \label{grnerve}
N \colon \Cat \longrightarrow \qcat
\end{equation} from the category of categories to the category of quasi-categories. It is by means of this functor that quasi-categories (a model for $(\infty,1)$-categories) can be understood as a generalisation of categories, and moreover that quasi-category theory can be understood as a generalisation of category theory.

One dimension higher, analogous constructions sending $2$-categories (or more generally bicategories) to $(\infty,2)$-categories have been studied for several of the models for $(\infty,2)$-categories, by means of which these models can be understood as generalisations of $2$-categories.  
 For example:\ change of base along the nerve functor (\ref{grnerve}) defines a fully faithful functor $N_{\ast} \colon \twocat \lra \qcat\text{-}\Cat$ from the category of $2$-categories and $2$-functors to the category of quasi-category-enriched categories (which plays an important role in the program of Riehl and Verity, see for instance \cite{MR3415698}); the Roberts--Street--Duskin nerve with the equivalence marking defines a fully faithful functor from the category $\Bicat$ of bicategories and normal pseudofunctors to the category of (weak, saturated) $2$-complicial sets \cite{MR1421811,MR1897816,duskinsimp2,MR1973518,MR2498786,orcomplicialnerve}; the Lack--Paoli $2$-nerve (followed by change of base along the nerve functor (\ref{grnerve})) defines a fully faithful functor from $\Bicat$ to the category of (Reedy fibrant) quasi-category-enriched Segal categories \cite{MR2366560}.
 
Yet missing from this list of examples is the model for $(\infty,2)$-categories most similar to the quasi-category model for $(\infty,1)$-categories:\ Ara's $2$-quasi-categories \cite{MR3350089}, defined as the fibrant objects of a model structure on the category  $[\Theta_2^\mathrm{op},\Set]$  of presheaves over Joyal's category $\Theta_2$ (called \emph{$\Theta_2$-sets}, or \emph{$2$-cellular sets}). And at first blush, the results of \cite[\S7]{MR3350089} seem cause for concern; for there \emph{is} a canonical fully faithful nerve functor  $N_s \colon \twocat \lra [\Theta_2^\mathrm{op},\Set]$ (studied for instance in \cite{MR1916373,MR2331244,MR2925893}) -- which we call the \emph{strict \textup{(}$2$-cellular\textup{)} nerve} functor (see \S\ref{strnobs} below) --
 but Ara shows in \cite[Corollary 7.11]{MR3350089} that this functor does \emph{not} send every $2$-category to a $2$-quasi-category. (Indeed, this negative result prompted Ara to change the name for these fibrant objects from ``quasi-$2$-category'' to ``$2$-quasi-category'', since, as Ara says, ``strict $n$-categories should be quasi-$n$-categories''  \cite[Remark 5.21]{MR3350089}.)

The subject of this paper is an alternative $\Theta_2$-set-valued nerve construction for $2$-categories (indeed, for bicategories) introduced by Leinster \cite[Definition J]{MR1883478}, which defines a fully faithful functor $N \colon \Bicat \lra [\Theta_2^\mathrm{op},\Set]$ from the category of bicategories and normal pseudofunctors to the category of $\Theta_2$-sets, and which can be understood as a ``homotopy coherent''  variant of the strict $2$-cellular nerve construction. 
The first main goal of this paper is to show that this \emph{coherent nerve} construction succeeds where the strict nerve failed:
\begin{theorema} 
The coherent nerve of a bicategory is a $2$-quasi-category.
\end{theorema}
To prove Theorem A, we prove moreover that the coherent nerve functor restricts to the right part of a Quillen adjunction
\begin{equation} \label{introadj}
\xymatrix{
\Bicat_\mathrm{s} \ar@<-1.5ex>[rr]^-{\hdash}_-N && \ar@<-1.5ex>[ll]_-{\tau_b} [\Theta_2^\mathrm{op},\Set]}
\end{equation}
between Lack's model structure for bicategories (on the category of bicategories and strict morphisms) \cite{MR2138540} and Ara's model structure for $2$-quasi-categories. We also show that each component of the counit of this adjunction is a weak equivalence, and that Lack's model structure is right-induced from Ara's model structure along the right adjoint of this adjunction. 

The second main goal of this paper is to give an intrinsic characterisation of the $2$-quasi-categories that are equivalent (i.e.\ weakly equivalent in Ara's model structure for $2$-quasi-categories) to the coherent nerve of a bicategory. We say that a $2$-quasi-category $X$ is \emph{$2$-truncated} if for each pair of objects $x,y \in X$, the hom-quasi-category $\Hom_X(x,y)$ is equivalent to the nerve of a category. 
\begin{theoremb}
A $2$-quasi-category is equivalent to the coherent nerve of a bicategory if and only if it is $2$-truncated.
\end{theoremb}
To prove Theorem B, we prove that a morphism of $2$-quasi-categories is an equivalence (i.e.\ a weak equivalence in Ara's model structure for $2$-quasi-categories) if and only if it is essentially surjective on objects and an equivalence on hom-quasi-categories. Furthermore, we construct the \emph{homotopy bicategory} $\Ho(X)$ of a $2$-quasi-category $X$, which we show defines the left adjoint of an adjunction 
\begin{equation*} 
\xymatrix{
\Bicat \ar@<-1.5ex>[rr]^-{\hdash}_-N && \ar@<-1.5ex>[ll]_-{\Ho} \twoqcat
}
\end{equation*}
between the categories of bicategories and $2$-quasi-categories. We  prove that,  for each $2$-quasi-category $X$, the unit morphism $X \lra N(\Ho(X))$ is bijective on objects, and is an equivalence on hom-quasi-categories if and only if $X$ is $2$-truncated.

Using Theorem B, we prove that the adjunction \textup{(\ref{introadj})} is moreover a Quillen equivalence between Lack's model structure for bicategories and the model structure for $2$-truncated $2$-quasi-categories, which we construct as the Bousfield localisation of Ara's model structure for $2$-quasi-categories with respect to the boundary inclusion $\partial\Theta_2[1;3] \lra \Theta_2[1;3]$. Furthermore, we show that the composite of this adjunction with an adjunction due to Ara
\begin{equation} \label{introadjcomp}
\xymatrix{
\Bicat_\mathrm{s} \ar@<-1.5ex>[rr]^-{\hdash}_-N && \ar@<-1.5ex>[ll]_-{\tau_b} [\Theta_2^\mathrm{op},\Set] \ar@<-1.5ex>[rr]^-{\hdash}_-{t^!} && \ar@<-1.5ex>[ll]_-{t_!} [(\Theta_2\times\Delta)^\mathrm{op},\Set]
}
\end{equation}
is a Quillen equivalence between Lack's model structure for bicategories and Rezk's model structure for $(2,2)$-$\Theta$-spaces \cite{MR2578310}. We also prove that the coherent nerve functor defines a triequivalence between the  tricategories (in fact, strict $\Bicat$-enriched categories) of bicategories and $2$-truncated $2$-quasi-categories.

The third and final main goal of this paper is to prove that the coherent nerve $NA$ of a $2$-category $A$ is a fibrant replacement of its strict nerve $N_sA$ in the model structure for $2$-quasi-categories. 
\begin{theoremc}
For every $2$-category $A$, the canonical inclusion $N_sA \lra NA$ is a weak equivalence in the model structure for $2$-quasi-categories. 
\end{theoremc}
To prove Theorem C, we prove a Quillen equivalence
\begin{equation} \label{introadjd}
\xymatrix{
\PCat \ar@<-1.5ex>[rr]^-{\hdash}_-{d_*} && \ar@<-1.5ex>[ll]_-{d^*} [\Theta_2^\mathrm{op},\Set]
}
\end{equation}
between the Hirschowitz--Simpson--Pellissier model structure for quasi-category-enriched Segal categories \cite{MR2883823} and Ara's model structure for $2$-quasi-categories. 
Using this Quillen equivalence, we also prove Ara's conjecture \cite[\S7]{MR3350089} that a $2$-functor is a biequivalence if and only if it is sent by the strict nerve functor $N_s \colon \twocat \lra [\Theta_2^\mathrm{op},\Set]$ to a weak equivalence in the model structure for $2$-quasi-categories. Finally, we use this Quillen equivalence, together with a Quillen equivalence established by Joyal and Tierney \cite{MR2342834}, to give an explicit construction of the  ``locally Kan'' $2$-quasi-category associated to each quasi-category.

The structure of this paper is as follows. Following a brief recollection of the category $\Theta_2$ and of $\Theta_2$-sets in \S\ref{theta2sec}, in \S\ref{nervefunsec} we recall Leinster's nerve construction and prove (Theorem \ref{ffthm}) that it defines a fully faithful functor $N \colon \Bicat \lra [\Theta_2^\mathrm{op},\Set]$,
and moreover that the nerve of a bicategory is determined by its restriction to a tractable subcategory $\Theta_b$ of $\Theta_2$. In \S\ref{modstrsec}, we recall Lack's model category of bicategories and Ara's model structure for $2$-quasi-categories, and prove a recognition principle (Proposition \ref{welemma}) for left Quillen functors from the latter. We apply these results in \S\ref{mainthmsec} to prove (Theorem \ref{firstqadjthm}) that the adjunction (\ref{introadj}) is a Quillen adjunction 
between Lack's model structure for bicategories and Ara's model structure for $2$-quasi-categories. We then deduce Theorem A as Corollary \ref{fibcor}.

In \S\ref{hobicatsec}, we construct the homotopy bicategory of a $2$-quasi-category, and prove (Theorem \ref{honadjthm}) that this defines a left adjoint $\Ho \colon \twoqcat \lra \Bicat$ to the coherent nerve functor; we also prove (Theorem \ref{segalthm}) that the underlying bisimplicial set of a $2$-quasi-category is a quasi-category-enriched Segal category. In \S\ref{fundythmsec}, we prove (Theorem \ref{fundythm}) that a morphism of $2$-quasi-categories is an equivalence if and only if it is essentially surjective on objects and fully faithful, and thereby deduce Theorem B as Theorem \ref{2truncchar}. We apply these results in \S\ref{vssect}, where we prove (Theorems \ref{firstqequiv} and \ref{lackrezkequiv}) the Quillen equivalences (\ref{introadj}) and (\ref{introadjcomp}) between Lack's model structure for bicategories, the model structure for $2$-truncated $2$-quasi-categories, and Rezk's model structure for $(2,2)$-$\Theta$-spaces, and also in \S\ref{triequivsect}, where we prove (Theorem \ref{adjtriequivthm}) a triequivalence between bicategories and $2$-truncated $2$-quasi-categories.

In \S\ref{araconsec}, we prove (Theorem \ref{qe10}) the Quillen equivalence (\ref{introadjd}) between the Hirschowitz--Simpson--Pellissier model structure for quasi-category-enriched Segal categories and Ara's model structure for $2$-quasi-categories, from which we deduce Theorem C as Theorem \ref{fibreplthm}, and prove  Ara's conjecture (Theorem \ref{araconjthm}). In  \S\ref{sec1to2}, we construct (Theorem \ref{finalfinalthm}) an explicit fibrant replacement of the $\Theta_2$-set associated to each quasi-category in the model structure for $2$-quasi-categories.

\subsection*{Acknowledgements} 
This paper is based on talks given in 2018--2019  
 at the Australian Category Seminar, whose members the author thanks for being an intelligent and stimulating audience. Additional thanks are due to Richard Garner for his guidance. The author gratefully acknowledges the support of Australian Research Council Discovery Grant DP160101519 and Future Fellowship FT160100393.

\section{Joyal's category $\Theta_2$ and $\Theta_2$-sets} \label{theta2sec}

We begin with a recollection of Joyal's category $\Theta_2$, which can be understood as a two-dimensional analogue of the simplex category $\Delta$. This category admits many equivalent definitions (see for example \cite{joyaltheta,MR1787588,MR1825434,MR1916373,MR2331244,MR2299807,MR2770071,MR2925893}); for our purposes it will be most convenient to identify $\Theta_2$ with the full subcategory of the category $\twocat$ of (strict) $2$-categories and (strict) $2$-functors consisting of the $2$-categories freely generated by certain $2$-graphs, sometimes known as the \emph{$2$-dimensional globular pasting diagrams}.

\subsection{$2$-graphs} \label{2graphs}
A  \emph{$2$-graph} (or \emph{$2$-globular set}) is defined to be  a presheaf over the 2-dimensional globe category $\mathbb{G}_2$, which admits the following presentation:
\begin{equation*}
\mathbb{G}_2 =
\big\langle
\xymatrix{
0 \ar@<1ex>[r]^-{\sigma} \ar@<-1ex>[r]_-{\tau} & 1 \ar@<1ex>[r]^-{\sigma} \ar@<-1ex>[r]_-{\tau} & 2
}
\,  \big | \, \,
\sigma\sigma = \tau\sigma,  \sigma\tau = \tau\tau
\big\rangle.
\end{equation*}
Thus a $2$-graph $X$ consists of a set $X_0$ of objects, a set $X_1$ of arrows, and a set $X_2$ of $2$-cells, together with source and target functions
\begin{equation*}
\cd{
X_2 \ar@<1ex>[r]^-s \ar@<-1ex>[r]_-t & X_1 \ar@<1ex>[r]^-s \ar@<-1ex>[r]_-t & X_0
}
\end{equation*}
satisfying the globularity conditions $ss = st$ and $ts = tt$. Equivalently, a $2$-graph $X$ may be defined as a \emph{graph-enriched graph}, and as such is determined by its set $X_0$ of objects and, for each pair of objects $x,y \in X_0$, the \emph{hom-graph} $X(x,y)$ whose objects and arrows are the arrows and $2$-cells of $X$ with source object $x$ and target object $y$.

Every $2$-category $A$ has an underlying $2$-graph $UA$ made up of its objects, morphisms, and $2$-cells. This defines the right adjoint of a monadic adjunction  
\begin{equation} \label{fuadj}
\cd{
\twocat \ar@<-1.5ex>[rr]^-{\hdash}_-U && \ar@<-1.5ex>[ll]_-F [\mathbb{G}_2^\mathrm{op},\Set]
}
\end{equation}
between the category of $2$-categories  and the category of $2$-graphs,
whose left adjoint sends a $2$-graph $X$ to the $2$-category $FX$ freely generated by $X$. For an explicit construction of the $2$-category freely generated by a $2$-graph, see \cite[Appendix F]{MR2094071}.

\subsection{$2$-dimensional globular pasting diagrams}
To fix notation, for each integer $n \geq 0$, let $(n)$ denote the generating graph of the free category $[n] = \{0 < \cdots < n\}$. (Note that we also write $\mathbf{1}$ and $\mathbf{2}$ for the categories $[0]$ and $[1]$.)

For each integer $n \geq 0$ and each string of integers $m_1,\ldots,m_n \geq 0$, let $(n;\bm{m}) = (n;m_1,\ldots,m_n)$ denote the $2$-graph -- the \emph{$2$-dimensional globular pasting diagram} of ``length'' $n$ and ``widths'' $m_1,\ldots,m_n$ -- whose set of objects is $\{0,\ldots,n\}$, and whose nonempty hom-graphs are the graphs $(n;\bm{m})(i-1,i) = (m_i)$ for each $1 \leq i \leq n$. (See Figure \ref{globs} for an illustration of some of these $2$-graphs.) Note that the $2$-graphs $(0)$, $(1;0)$, and $(1;1)$ are precisely the representable $2$-graphs, i.e.\ those in the image of the Yoneda embedding $\mathbb{G}_2 \lra [\mathbb{G}_2^\mathrm{op},\Set]$.

Let $[n;\bm{m}] = [n;m_1,\ldots,m_n]$ denote the $2$-category freely generated by the $2$-graph \linebreak$(n;m_1,\ldots,m_n)$. By construction, this $2$-category 
has the set of objects $\{0,\ldots,n\}$, and its nonempty hom-categories are given by the cartesian products
\begin{equation*}
[n;\bm{m}](i,j) = [m_{i+1}] \times \cdots \times [m_j]
\end{equation*}
for each pair of integers $0 \leq i \leq j \leq n$. 

\begin{figure}
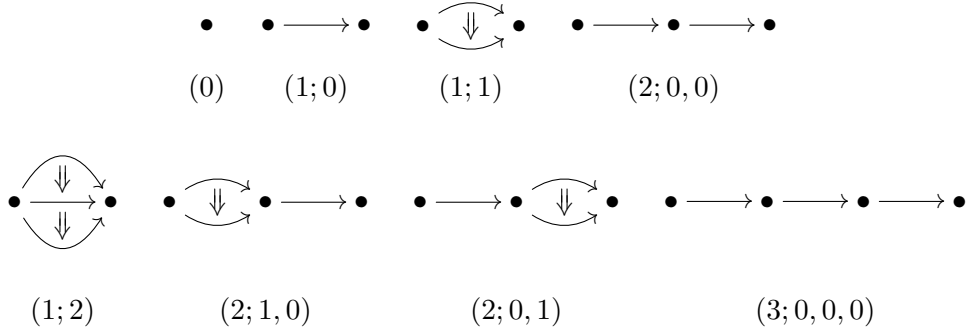
 
\begin{tabular}{cccc}
$\cd{
\bullet
}$ & $\cd{
\bullet \ar[r] & \bullet
}$
& $\cd{
\bullet \rtwocell  & \bullet 
}$ & 
$\cd{
\bullet \ar[r] & \bullet \ar[r] & \bullet
}$ \\
$(0)$ & $(1;0)$ & $(1;1)$ & $(2;0,0)$ \\
\end{tabular}
\begin{tabular}{cccc}
$\cd{
\bullet \ruppertwocell \rlowertwocell \ar[r] & \bullet \\
}$
&
$\cd{
\bullet \rtwocell  & \bullet  \ar[r] & \bullet
}$
&
$\cd{
\bullet \ar[r] & \bullet \rtwocell  & \bullet 
}$
&
$\cd{
\bullet \ar[r] & \bullet \ar[r] &\bullet \ar[r] &\bullet
}$ \\
$(1;2)$ & $(2;1,0)$ & $(2;0,1)$ & $(3;0,0,0)$ \\
\end{tabular}
\caption{$2$-dimensional globular pasting diagrams of degree $\leq 3$.} \label{globs}
\end{figure}

\subsection{Joyal's category $\Theta_2$}
We define the category $\Theta_2$ to be the full subcategory of $\twocat$ consisting of the $2$-categories $[n;\bm{m}] = [n;m_1,\ldots,m_n]$ for every $n \geq 0$ and $m_1,\ldots,m_n \geq 0$. 
By the adjunction (\ref{fuadj}), morphisms $(\phi;\bm{f}) \colon [n;\bm{m}] \lra [q;\bm{p}]$ in $\Theta_2$ are in natural bijection with morphisms of $2$-graphs $(n;\bm{m}) \lra U[q;\bm{p}]$ from the generating $2$-graph of $[n;\bm{m}]$ to the underlying $2$-graph of $[q;\bm{p}]$; such a morphism thus consists of a morphism $\phi \colon [n] \lra [q]$ in $\Delta$ and, for each integer $1 \leq i \leq n$, a functor 
\begin{equation*}
f_i \colon [m_i] \lra [p_{\phi(i-1) + 1}] \times \cdots \times [p_{\phi(i)}], 
\end{equation*}
i.e.\ a morphism $f_{ij} \colon [m_i] \lra [p_j]$ in $\Delta$ for each integer $\phi(i-1) < j \leq \phi(i)$ (cf.\ Berger's definition of $\Theta_2$ as the wreath product $\Delta \wr \Delta$ \cite[Definitions 3.1 and 3.3]{MR2331244}). (Note that we will denote the elementary face and degeneracy operators in $\Delta$ by the letters $\delta$ and $\sigma$ with superscripts.)

\subsection{$\Theta_2$ is an elegant Reedy category} \label{reedyrecall}
The \emph{degree} of an object $[n;m_1,\ldots,m_n] \in \Theta_2$ is defined to be the sum $n + m_1 + \cdots + m_n$. This is part of an EZ-Reedy category structure on $\Theta_2$ (see \cite{MR3054341}, cf.\ \cite[Lemma 2.4]{MR1916373}), 
 in which the ``degree-raising'' and ``degree-lowering'' morphisms are the monomorphisms and the split epimorphisms in $\Theta_2$ respectively. 
Thus $\Theta_2$ is a \emph{cat{\'e}gorie squelletique r{\'e}guli{\`e}re} in the sense of \cite[\S8.2]{MR2294028}, that is, an EZ-Reedy category whose degree-raising morphisms are the monomorphisms. (Note that, but for the convenience of citing results from \cite{MR2294028}, for our purposes it would suffice to know that $\Theta_2$ is an elegant Reedy category \cite[Corollary 4.4]{MR3054341}.)

\subsection{$\Theta_2$-sets} \label{onetwoobs}
Objects of the presheaf category $[\Theta_2^\mathrm{op},\Set]$ are called \emph{$\Theta_2$-sets} (or \emph{$2$-cellular sets}). We denote the value of a $\Theta_2$-set $X \colon\Theta_2^\mathrm{op} \lra \Set$ at an object $[n;\bm{m}] \in \Theta_2$ by $X_{n(\bm{m})}$, the elements of which we call the $[n;\bm{m}]$-elements of $X$; we  denote by $(\phi;\bm{f})^* \colon X_{q(\bm{p})} \lra X_{n(\bm{m})}$ the function induced by a morphism $(\phi;\bm{f}) \colon [n;\bm{m}] \lra [q;\bm{p}]$ in $\Theta_2$, which sends a $[q;\bm{p}]$-element $x$ of $X$ to the $[n;\bm{m}]$-element $x \cdot (\phi;\bm{f})$ of $X$. We say that an $[n;\bm{m}]$-element of a $\Theta_2$-set $X$ has \emph{dimension} $d$ when the object $[n;\bm{m}]$ has degree $d$ (see \S\ref{reedyrecall}). 

\begin{observation}[a full embedding of simplicial sets into $\Theta_2$-sets] \label{afullemb}
Restriction along the functor $\pi \colon \Theta_2 \lra \Delta$, defined by $\pi([n;\bm{m}]) = [n]$ and $\pi(\phi;\bm{f}) = \phi$, gives a fully faithful functor $\pi^* \colon [\Delta^\mathrm{op},\Set] \lra [\Theta_2^\mathrm{op},\Set]$ from the category of simplicial sets to the category of $\Theta_2$-sets. This functor has a right adjoint $\tau^* \colon [\Theta_2^\mathrm{op},\Set] \lra [\Delta^\mathrm{op},\Set]$, given by restriction along the full inclusion $\tau \colon \Delta \lra \Theta_2$, which sends $[n]$ to $[n;0,\ldots,0]$. We say that this right adjoint sends a $\Theta_2$-set $X$ to its \emph{underlying simplicial set} $\tau^*(X)$.
\end{observation}

We conclude this section with some further examples of (morphisms of) $\Theta_2$-sets, which we will need in our study of the model structure for $2$-quasi-categories in \S\S\ref{modstrsec}--\ref{mainthmsec}.

\subsection{Representable $\Theta_2$-sets and their boundaries} \label{bdyrecall}
For each object $[n;\bm{m}] \in \Theta_2$, let $\Theta_2[n;\bm{m}]$ denote the representable $\Theta_2$-set $\Theta_2(-,[n;\bm{m}])$; by definition, the elements of $\Theta_2[n;\bm{m}]$ are the morphisms of $\Theta_2$ with codomain $[n;\bm{m}]$. 
We define the \emph{boundary} $\partial\Theta_2[n;\bm{m}]$ to be the 
sub-$\Theta_2$-set of $\Theta_2[n;\bm{m}]$ 
whose elements are those morphisms in $\Theta_2$ with codomain $[n;\bm{m}]$ that factor through an object of $\Theta_2$ of degree strictly less than the degree of $[n;\bm{m}]$; 
we denote the \emph{boundary inclusion} by $\delta_{n(\bm{m})} \colon \partial\Theta_2[n;\bm{m}] \lra \Theta_2[n;\bm{m}]$. 
Every monomorphism in $[\Theta_2^\mathrm{op},\Set]$ can be expressed as a countable composite of pushouts of coproducts of boundary inclusions $\delta_{n(\bm{m})} \colon\partial\Theta_2[n;\bm{m}] \lra \Theta_2[n;\bm{m}]$ for $[n;\bm{m}] \in \Theta_2$ (cf.\ \cite[Proposition 8.1.37]{MR2294028}).

\subsection{Spines and spine inclusions} \label{spinerecall}
Let $g \colon \mathbb{G}_2 \lra \Theta_2$ denote the faithful functor whose image is the subcategory of $\Theta_2$ generated by the morphisms displayed below.
\begin{equation*}
\cd[@=3em]{
[0] \ar@<-1ex>[r]_{\delta^0} \ar@<1ex>[r]^{\delta^1} & [1;0] \ar@<-1ex>[r]_-{(\mathrm{id};\delta^0)} \ar@<1ex>[r]^-{(\mathrm{id};\delta^1)} & [1;1]
}
\end{equation*}
This functor induces an adjunction
\begin{equation} \label{jadj}
\xymatrix{
[\Theta_2^\mathrm{op},\Set] \ar@<-1.5ex>[rr]^-{\hdash}_-{g^{\ast}} && \ar@<-1.5ex>[ll]_-{g_!} [\mathbb{G}_2^\mathrm{op},\Set]
}
\end{equation}
between the categories of $\Theta_2$-sets and $2$-graphs
whose left and right adjoints are given by left Kan extension and restriction along $g^\mathrm{op}$ respectively. 

For each object $[n;\bm{m}] \in \Theta_2$, with generating $2$-graph $(n;\bm{m})$, we define the \emph{spine} $I[n;\bm{m}]$ to be the $\Theta_2$-set $g_!(n;\bm{m})$. The inclusion of generators $(n;\bm{m}) \lra U[n;\bm{m}] = g^*(\Theta_2[n;\bm{m}])$ corresponds under the adjunction $g_! \dashv g^*$ to a monomorphism $i_{n(\bm{m})}\colon I[n;\bm{m}] \lra \Theta_2[n;\bm{m}]$, which we call the \emph{spine inclusion}. The image of the spine inclusion $i_{n(\bm{m})}$ is the sub-$\Theta_2$-set of $\Theta_2[n;\bm{m}]$ generated by the inert morphisms in $\Theta_2$ with codomain $[n;\bm{m}]$ and domain either one of $[0]$, $[1;0]$, or $[1;1]$ (a morphism in $\Theta_2$ is said to be \emph{inert} if it is in the image of the ``free'' functor $F \colon [\mathbb{G}_2^\mathrm{op},\Set] \lra \twocat$).

\section{The coherent nerve functor} \label{nervefunsec}
In this section, we recall Leinster's $\Theta_2$-set-valued nerve construction for bicategories from \cite[Definition J]{MR1883478}, and prove that it defines a fully faithful functor from the category $\Bicat$ of bicategories and normal pseudofunctors to the category $[\Theta_2^\mathrm{op},\Set]$ of $\Theta_2$-sets, and moreover that the coherent nerve of a bicategory is determined by its restriction to a tractable subcategory $\Theta_b$ of $\Theta_2$. 
Our proof of these results follows the argument of \cite[\S3]{MR2366560}, where the corresponding results for the $2$-nerve for bicategories are proved. 

\begin{recall}[Kan's construction \cite{MR0131451}] \label{kanrecall}
 Let $K \colon \mathbf{A} \lra \mathcal{C}$ be a functor from a small category $\mathbf{A}$ to a locally small category $\mathcal{C}$. The \emph{singular functor} (or \emph{nerve functor}) induced by $K$ is the functor $\mathcal{C}(K,1) \colon \mathcal{C} \lra [\mathbf{A}^\mathrm{op},\Set]$ that sends an object $C\in\mathcal{C}$ to the presheaf $\mathcal{C}(K-,C) \colon \mathbf{A}^\mathrm{op} \lra \Set$, and sends a morphism $f \colon C \lra D$ in $\mathcal{C}$ to the natural transformation $\mathcal{C}(K-,f) \colon \mathcal{C}(K-,C) \lra  \mathcal{C}(K-,D)$. The functor $K$ is said to be \emph{dense} if the singular functor $\mathcal{C}(K,1) \colon \mathcal{C} \lra [\mathbf{A}^\mathrm{op},\Set]$ is fully faithful. If the category $\mathcal{C}$ is cocomplete, then the singular functor $\mathcal{C}(K,1)$ has a left adjoint
\begin{equation*}
\xymatrix{
\mathcal{C} \ar@<-1.5ex>[rr]^-{\hdash}_-{\mathcal{C}(K,1)} && \ar@<-1.5ex>[ll]_-{-\circledast K} [\mathbf{A}^\mathrm{op},\Set]
}
\end{equation*}
given by the left Kan extension of $K \colon \mathbf{A} \lra \mathcal{C}$ along the Yoneda embedding $\mathbf{A} \lra [\mathbf{A}^\mathrm{op},\Set]$, which sends a presheaf $F$ on $\mathbf{A}$ to the weighted colimit $F \circledast K \cong \int^A FA \times KA$ in $\mathcal{C}$. 
\end{recall}

\subsection{The strict nerve functor} \label{strnobs}
We define the \emph{strict \textup{(}$2$-cellular\textup{)} nerve functor} $$N_s \colon \twocat \lra [\Theta_2^\mathrm{op},\Set]$$ to be the singular functor induced by the 
 full inclusion $\Theta_2 \lra \twocat$. 
 This is the $2$-cellular nerve functor for $2$-categories studied in \cite{MR1916373}, where it is shown to be fully faithful (see also \cite{MR2331244,MR2925893}).
 Ara proved in \cite[Proposition 7.10]{MR3350089} that the strict nerve $N_sA$ of a $2$-category $A$ is a $2$-quasi-category (see \S\ref{secara}) if and only if the $2$-category $A$ is \emph{rigid}, i.e.\ has no non-identity invertible $2$-cells. 
 
 \begin{remark} \label{hcrmk}
 As observed in \cite[\S7]{MR3350089}, it is to be expected that the strict nerve functor does not send all $2$-categories to $2$-quasi-categories. For, as in quasi-category theory, one has the slogan that a morphism of $2$-quasi-categories $A \lra X$ is a ``weak functor from $A$ to $X$'', or a ``homotopy coherent diagram of shape $A$ in $X$''. But full fidelity of the strict nerve functor entails that morphisms of $\Theta_2$-sets $N_sC \lra N_sD$ correspond only to \emph{strict} $2$-functors $C \lra D$.
 
A special case of this slogan says that, for each $[n;\bm{m}] \in \Theta_2$, the $[n;\bm{m}]$-elements of a $2$-quasi-category $X$ correspond (via the Yoneda lemma) to ``weak functors from $\Theta_2[n;\bm{m}]$ to $X$''. This suggests how to define an alternative nerve construction, fit for the purpose of sending $2$-categories (and more generally bicategories) to $2$-quasi-categories:\ the $[n;\bm{m}]$-elements of the ``homotopy coherent'' nerve of a $2$-category (or bicategory) $B$ should be ``weak functors from $[n;\bm{m}]$ to $B$''. Taking ``weak functor'' to mean \emph{normal pseudofunctor}, this is precisely how Leinster's nerve construction is defined, as we now recall.
 \end{remark}

\subsection{Bicategories and normal pseudofunctors} \label{bicatrecall}
Let $\Bicat$ denote the category of bicategories and normal pseudofunctors (also called normal homomorphisms); for the definitions, see for instance \cite[\S9]{MR1421811} or the original \cite{MR0220789}. 
Recall that a normal pseudofunctor is a morphism of bicategories that preserves identities \emph{strictly} and preserves composition up to coherent isomorphism.  Thus the data of a normal pseudofunctor $F \colon A \lra B$ includes an invertible $2$-cell $\varphi_{g,f} \colon Fg.Ff \lra F(gf)$ in the bicategory $B$ for each composable pair of morphisms $f,g$ in the bicategory $A$; if each of these \emph{composition constraints} $\varphi_{g,f}$ is an identity $2$-cell in $B$, then the normal pseudofunctor $F$ is said to be a \emph{strict morphism} of bicategories. A strict morphism of bicategories between $2$-categories is precisely a $2$-functor.

Note that a normal pseudofunctor between $2$-categories whose codomain is rigid (see \S\ref{strnobs}) is necessarily a $2$-functor. Since the objects of $\Theta_2$ are rigid $2$-categories, this implies that the composite inclusion functor $\Theta_2 \lra \twocat\lra \Bicat$ is full.

\begin{definition}[the coherent nerve functor] \label{cohndef}
The \emph{coherent \textup{(}$2$-cellular\textup{)} nerve functor} $$N \colon \Bicat \lra [\Theta_2^\mathrm{op},\Set]$$ is defined to be the singular functor induced by the full inclusion $\Theta_2 \lra \Bicat$.
\end{definition}

The \emph{coherent nerve} $NB$ of a bicategory $B$ is precisely the nerve of $B$ defined by Leinster in \cite[Definition J]{MR1883478}. Note that the coherent nerve and the strict nerve of a rigid $2$-category coincide; in particular, for each $[n;\bm{m}] \in \Theta_2$, the representable $\Theta_2$-set $\Theta_2[n;\bm{m}]$ is both the strict nerve and the coherent nerve of the (rigid) $2$-category $[n;\bm{m}]$. 

\begin{observation}[simplicial and cellular nerves of categories] \label{nervecatrecall}
The fully faithful functor \linebreak$\pi^* \colon [\Delta^\mathrm{op},\Set] \allowbreak \lra [\Theta_2^\mathrm{op},\Set]$ (see \S\ref{onetwoobs}) sends the standard simplicial nerve of any category $C$ to the strict $2$-cellular nerve of $C$ seen as a locally discrete $2$-category, which is equal to its coherent $2$-cellular nerve (since any such $2$-category is rigid);  we will call this $\Theta_2$-set the \emph{$2$-cellular nerve} of $C$.
\end{observation}

The goal of this section is to prove that the coherent nerve functor of Definition \ref{cohndef} is fully faithful. First, let us describe the low-dimensional structure of the coherent nerve of a bicategory.

\subsection{The coherent nerve of a bicategory in low dimensions}
Let $A$ be a bicategory. For each object $[n;\bm{m}] \in \Theta_2$, the $[n;\bm{m}]$-elements of the coherent nerve $NA$ of $A$ (i.e.\ the elements of the set $(NA)_{n(\bm{m})}$) are by definition the normal pseudofunctors $[n;\bm{m}] \lra A$. With the aid of Figure \ref{pasts}, where such elements are illustrated, we see that the diagram of sets
\begin{equation*}
\cd[@=3.5em]{
(NA)_{1(1)} \ar@<-1.5ex>[r]_{(\mathrm{id};\delta^0)^*} \ar@<1.5ex>[r]^{(\mathrm{id};\delta^1)^*} & (NA)_{1(0)} \ar[l]|-{(\mathrm{id};\sigma^0)^*} \ar@<-1.5ex>[r]_-{(\delta^0)^*} \ar@<1.5ex>[r]^-{(\delta^1)^*} & (NA)_0 \ar[l]|-{(\sigma^0)^*} 
}
\end{equation*}
is precisely the underlying reflexive $2$-graph of $A$ (which consists of the sets of objects, morphisms, and $2$-cells of $A$, together with the functions assigning  sources and targets to the morphisms and $2$-cells of $A$, and the functions assigning identity morphisms and identity $2$-cells to the objects and morphisms of $A$).

Furthermore, we see that $(NA)_{2(0,0)}$ is the set of ``invertible $2$-simplices'' in $A$, that the three functions on the left below
\begin{equation*}
\cd[@=3em]{
(NA)_{2(0,0)}  \ar[r]|-{(\delta^1)^*} \ar@<1.5ex>[r]^-{(\delta^2)^*} \ar@<-1.5ex>[r]_-{(\delta^0)^*} & (NA)_{1(0)}
}
\qquad
\qquad
\cd[@=3em]{
(NA)_{2(0,0)} & (NA)_{1(0)} \ar@<1ex>[l]^-{(\sigma^0)^*} \ar@<-1ex>[l]_-{(\sigma^1)^*},
}
\end{equation*}
assign to an invertible $2$-simplex $\sigma$ as displayed in Figure \ref{pasts} its boundary morphisms $\sigma\cdot\delta^2 = f$, $\sigma\cdot\delta^1 = h$, and $\sigma\cdot \delta^0 = g$, and that the two functions on the right above send each morphism $f \colon a \lra b$ in $A$ to the invertible $2$-simplices in $A$ given by its left and right unit constraints in the bicategory $A$, as displayed below.
\begin{equation*}
f\cdot\sigma^1 = \cd{
& b \ar[dr]^-{1_b} \mydbldtwocell[0.6]{d}{\bm{l}}{\cong}\\
a \ar[ur]^-f \ar[rr]_-f &{} & b
}
\qquad
\qquad
f\cdot\sigma^0 = 
\cd{
& a \ar[dr]^-f \mydbldtwocell[0.6]{d}{\bm{r}}{\cong}\\
a \ar[ur]^-{1_a} \ar[rr]_-f &{} & b
}
\end{equation*}

As illustrated in Figure \ref{pasts}, the elements of $NA$ of dimension $3$ are determined by their boundary faces, and amount to certain pasting equations in the bicategory $A$.  Note that the description of the elements of $(NA)_{3(0,0,0)}$ involves the associativity constraints $\bm{a} \colon (hg)f \cong h(gf)$ of the bicategory $A$.

\begin{figure} 
\begin{tabular}{| c | c | c |}
\hline
$[n;\bm{m}]$ & \multicolumn{2}{c|}{Normal pseudofunctor $[n;\bm{m}] \lra A$} \\
\hline \hline
$[0]$ & $a$ &  \\ \hline
$[1;0]$ & $\cd{a \ar[r]^-f & b}$ & $f$ \\
\hline
$[1;1]$ & $\cd{a \rtwocell^f_g{\alpha} & b}$ & $\cd{f \ar[r]^-{\alpha} & g}$ \\ \hline
$[2;0,0]$ &  $\cd{
& b \ar[dr]^-g \mydbldtwocell[0.6]{d}{\sigma}{\cong}\\
a \ar[ur]^-f \ar[rr]_-h &{} & c
}$ & $\cd{
gf \ar[r]^-{\sigma}_-{\cong} & h
}$ \\ \hline
$[1;2]$ & $\cd[@=3em]{
a \ruppertwocell^f{\alpha} \ar[r]|-g \rlowertwocell_h{\beta} & b
} =
\cd[@=3em]{
a \rtwocell^f_h{\gamma} & b
}$ & $\cd{
& g \ar[dr]^-{\beta} \ar@{}[d]|(0.6){\veq} \\
f \ar[ur]^-{\alpha} \ar[rr]_-{\gamma} &{}& h
}$ \\ \hline
$[2;1,0]$ & $\cd[@=3em]{
& b \ar[dr]^-h \mydbldtwocell[0.6]{d}{\tau}{\cong}  \\
a \uruppertwocell^f{\alpha} \ar[ur]_g \ar[rr]_-m && c
}
=
\cd[@=3em]{
& b \ar[dr]^-h \mydbldtwocell[0.45]{d}{\sigma}{\cong} \\
a \ar[ur]^-f \rrtwocell^l_m{\gamma} & & c
}$ & $\cd[@=3em]{
hf \ar[r]^-{h\alpha} \ar[d]_{\sigma}^{\cong} \ar@{}[dr]|-{\veq} & hg \ar[d]^-{\tau}_-{\cong} \\
l \ar[r]_-{\gamma} & m
}$ \\ \hline
$[2;0,1]$ &$\cd[@=3em]{
& b \ar[dr]_-{k} \druppertwocell^h{\beta}  \mydbldtwocell[0.6]{d}{\tau}{\cong}  \\
a \ar[ur]^-f  \ar[rr]_-m && c
}
=
\cd[@=3em]{
& b \ar[dr]^-h \mydbldtwocell[0.45]{d}{\sigma}{\cong} \\
a \ar[ur]^-f \rrtwocell^l_m{\gamma} & & c
}$ & $\cd[@=3em]{
hf \ar[r]^-{\beta f} \ar[d]_{\sigma}^{\cong} \ar@{}[dr]|-{\veq} & kf \ar[d]^-{\tau}_-{\cong} \\
l \ar[r]_-{\gamma} & m
}$  \\ \hline
$[3;0,0,0]$ &  $\cd[@=2em]{
b  \ar[ddrr]|-{k} \ar[rr]^-g  && c \ar[dd]^-h \\
\\
a \ar[uu]^-f \ar[rr]_-m \mydbldtwocell[0.3]{uurr}{\beta}{\cong} \mydbldtwocell[0.7]{uurr}{\alpha}{\cong}&& d
}
=
\cd[@=2em]{
b \ar[rr]^-g \mydbldtwocell[0.3]{ddrr}{\gamma}{\cong} \mydbldtwocell[0.7]{ddrr}{\delta}{\cong}&& c \ar[dd]^-h \\
\\
a \ar[uu]^-f \ar[rr]_-m \ar[uurr]|-l && d
}$ & $\cd[@R=1.5em]{
(hg)f \ar[r]^-{\alpha f}_{\cong} \ar[dd]_-{\bm{a}}^-{\cong} & kf \ar[dr]^-{\beta}_-{\cong} \\
{} \ar@{}[rr]|-{\veq}&& m\\
h(gf) \ar[r]_-{h\gamma}^-{\cong} & hl \ar[ur]_-{\delta}^-{\cong}
}$ \\ \hline
\end{tabular}
\caption{Elements of dimension $\leq 3$ of the coherent nerve of a bicategory, displayed both as pasting diagrams and as commutative diagrams. \label{pasts}} 
\end{figure}

\subsection{The coherent nerve of a normal pseudofunctor} \label{npsobs}
Let $F \colon A \lra B$ be a normal pseudofunctor between bicategories. 
The induced morphism of $\Theta_2$-sets $NF \colon NA \lra NB$ is defined on the $[0]$-elements, $[1;0]$-elements, and $[1;1]$-elements of $NA$ by the action of the normal pseudofunctor $F$ on the objects, morphisms, and $2$-cells of $A$ respectively, and sends a $[2;0,0]$-element of $NA$ as on the left below
 \begin{equation} \label{2simpmap}
 \cd[@=3em]{
& b \ar[dr]^-g \mydbldtwocell[0.6]{d}{\sigma}{\cong}\\
a \ar[ur]^-f \ar[rr]_-h &{} & c
}
\qquad
\longmapsto
\qquad
\cd[@=3em]{
& Fb \ar[dr]^-{Fg} \mydbldtwocell[0.4]{d}{\varphi}{\cong} \\
Fa \ar[ur]^-{Ff} \ar@/^.8pc/[rr]^-{F(gf)} \ar@/_.8pc/[rr]_-{Fh} \mydbldtwocell{rr}{F\sigma}{\cong} & & Fc
}
\end{equation}
to the $[2;0,0]$-element of $NB$ determined by the pasting composite in $B$ displayed on the right above, which is the vertical composite $F\sigma \circ \varphi_{g,f} \colon Fg.Ff \lra Fh$ (where $\varphi_{g,f} \colon Fg.Ff \lra F(gf)$ denotes the composition constraint of the normal pseudofunctor $F$). Since the elements of dimension $3$ in $NA$ and $NB$ are determined by their boundaries, the components of $NF$ at the objects of $\Theta_2$ of degree $3$ involve no additional data, but merely assert the preservation  of certain pasting equations by the normal pseudofunctor $F$. 

\subsection{The proof of full fidelity} \label{thetab}
To prove that the coherent nerve functor $N \colon \Bicat \lra [\Theta_2^\mathrm{op},\Set]$ is fully faithful is to prove that the full inclusion $\Theta_2 \lra \Bicat$ is dense, for which it suffices to prove that the inclusion of some subcategory of $\Theta_2$ into $\Bicat$ is dense, by the following standard fact about dense functors (which is the $\Set$-enriched case of \cite[Proposition 1.1]{MR2366560}). 

\begin{lemma} \label{denselemma}
Let $F \colon \mathbf{A} \lra \mathbf{B}$ and $G\colon \mathbf{B} \lra \mathcal{C}$ be functors, where $\mathbf{A}$ and $\mathbf{B}$ are small categories and $\mathcal{C}$ is a locally small category. If the functor $G \colon \mathbf{B} \lra \mathcal{C}$ is fully faithful and the composite functor $GF \colon \mathbf{A} \lra \mathcal{C}$ is dense, then the functor $G \colon \mathbf{B} \lra \mathcal{C}$ is dense; furthermore, the singular functor $\mathcal{C}(G,1) \colon \mathcal{C} \lra [\mathbf{B}^\mathrm{op},\Set]$ is  isomorphic to the composite
\begin{equation*}
\cd[@C=3.5em]{
\mathcal{C} \ar[r]^-{\mathcal{C}(GF,1)} & [\mathbf{A}^\mathrm{op},\Set] \ar[r]^-{F_{\ast}} & [\mathbf{B}^\mathrm{op},\Set]
}
\end{equation*}
of the singular functor induced by $GF$ and the functor $F_{\ast}$ defined by right Kan extension along $F^\mathrm{op} \colon \mathbf{A}^\mathrm{op} \lra \mathbf{B}^\mathrm{op}$. \qed
\end{lemma}

\begin{remark}
At this point, we could invoke a result from Watson's thesis \cite{Wat13}, which, as we now explain, implies that the full inclusion $(\Theta_2)_{\leq 3} \lra \Bicat$ is dense, where $(\Theta_2)_{\leq 3}$ denotes the full subcategory of $\Theta_2$ consisting of the objects of degree $\leq 3$. In \cite[Chapter 6]{Wat13}, Watson defines a nerve functor $\mathbf{fBicat} \lra [\Theta_2^\mathrm{op},\Set]$ from the category of ``fancy bicategories'', which contains $\Bicat$ as a full subcategory,  to the category of $\Theta_2$-sets. Watson proves that this functor gives an equivalence of categories between $\mathbf{fBicat}$ and the full subcategory of $[\Theta_2^\mathrm{op},\Set]$ consisting of the ``$2$-reduced inner-Kan'' $\Theta_2$-sets; see \cite[Remark 6.7.2]{Wat13}.

By inspection, one can see that the restriction of Watson's nerve functor to the full subcategory $\Bicat$ of $\mathbf{fBicat}$ is (essentially) the composite
\begin{equation*} \label{watnerve}
\cd[@C=3.5em]{
\Bicat \ar[r]^-{N_3} & [(\Theta_2)_{\leq 3}^\mathrm{op},\Set] \ar[r]^-{(i_3)_*} & [\Theta_2^\mathrm{op},\Set]
}
\end{equation*}
of the singular functor $N_3$ induced by the inclusion $(\Theta_2)_{\leq 3} \lra \Bicat$, and the functor  $(i_3)_*$ defined by right Kan extension along (the opposite of) the inclusion $i_3 \colon (\Theta_2)_{\leq 3} \lra \Theta_2$. Hence the equivalence of categories proved by Watson implies that the functor $N_3$ is fully faithful, i.e.\ that the inclusion $(\Theta_2)_{\leq 3} \lra \Bicat$ is dense. Lemma \ref{denselemma}  then implies that the coherent nerve functor of Definition \ref{cohndef} is fully faithful, and moreover that it is isomorphic to the composite functor displayed above.

However, Watson's proof of this equivalence of categories is (of necessity) considerably longer and more complicated than a direct proof of the density of the inclusion $(\Theta_2)_{\leq 3} \lra \Bicat$ alone would be. (Watson proves this equivalence by showing that the nerve of a fancy bicategory is a $2$-reduced inner-Kan $\Theta_2$-set, and by constructing a pseudo-inverse to the nerve functor from the full subcategory of these $\Theta_2$-sets to $\mathbf{fBicat}$.) Therefore, to make clear that Theorem \ref{ffthm} below does not depend on the full extent of Watson's proof, we give in Proposition \ref{ffprop}
a simpler and more direct (and significantly shorter) proof of a slightly sharper density result.
 \end{remark}
 
 \subsection{A tractable subcategory of $\Theta_2$}
Let $\Theta_b$ denote the (non-full) subcategory of $\Theta_2$ generated by all morphisms between objects of degree $\leq 2$, and all monomorphisms from objects of degree $2$ to objects of degree $3$; the objects of $\Theta_b$ are therefore the objects of degree $\leq 3$ of $\Theta_2$. Let $N_b \colon \Bicat \lra [\Theta_b^\mathrm{op},\Set]$ denote the singular functor induced by the composite of the inclusions $i_b \colon \Theta_b \lra \Theta_2$ and $\Theta_2 \lra \Bicat$. This singular functor is equal to the composite of the coherent nerve functor $N \colon \Bicat \lra [\Theta_2^\mathrm{op},\Set]$ and the restriction functor $i_b^* \colon [\Theta_2^\mathrm{op},\Set] \lra [\Theta_b^\mathrm{op},\Set]$.

\begin{proposition} \label{ffprop}
The truncated coherent nerve functor $N_b \colon \Bicat \lra [\Theta_b^\mathrm{op},\Set]$ is fully faithful, i.e.\ the inclusion functor  $\Theta_b \lra \Bicat$ is dense. 
\end{proposition}
\begin{proof}
The functor $N_b$ is faithful since, by the observations in \S\ref{npsobs}, the data of a normal pseudofunctor $F \colon A \lra B$ can be recovered from the morphism of $\Theta_b$-sets $N_bF \colon N_bA \lra N_bB$. The action of $F$ on the objects, morphisms, and $2$-cells of $A$ is given precisely by the components of $N_bF$ at the objects $[0]$, $[1;0]$, and $[1;1]$ of $\Theta_b$. Moreover, for each composable pair of morphisms $f \colon a \lra b$, $g \colon b \lra c$ in $A$, the composition constraint $\varphi_{g,f} \colon Fg.Ff \cong F(gf)$ of the normal pseudofunctor $F$ can be recovered as the image of the commutative $2$-simplex displayed on the left below under the component of $N_bF$ at the object $[2;0,0] \in \Theta_b$. 
\begin{equation*}
\cd{
& b \ar[dr]^-g \ar@{}[d]|(0.6){\veq}\\
a \ar[ur]^-f \ar[rr]_-{gf} &{} & c
}
\qquad
\longmapsto
\qquad
\cd{
& Fb \ar[dr]^-{Fg} \mydbldtwocell[0.6]{d}{\varphi}{\cong}\\
Fa \ar[ur]^-{Ff} \ar[rr]_-{F(gf)} &{} & Fc
}
\end{equation*}

To prove that the functor $N_b$ is full, let $A$ and $B$ be bicategories, and let $F \colon N_bA \lra N_bB$ be a morphism in $[\Theta_b^\mathrm{op},\Set]$ (i.e.\ a natural transformation). As in the previous paragraph, we can determine from $F$ the data of a normal pseudofunctor from $A$ to $B$, namely functions from the sets of objects, morphisms, and $2$-cells of $A$ to those of $B$, and, for each composable pair of morphisms $f,g$ in $A$, an invertible $2$-cell $\varphi_{g,f} \colon Fg.Ff \lra F(gf)$ in $B$, whose source and target are as written by virtue of the naturality of $F$ with respect to the morphisms
\begin{equation} \label{2simpface}
\cd[@=3em]{
[1;0]  \ar[r]|-{\delta^1} \ar@<1.5ex>[r]^-{\delta^2} \ar@<-1.5ex>[r]_-{\delta^0} & [2;0,0]
}
\end{equation}
in $\Theta_b$. Furthermore, 
 naturality of $F$ with respect to the morphisms
\begin{equation*}
\cd[@=3em]{
[0] \ar@<-1.5ex>[r]_{\delta^0} \ar@<1.5ex>[r]^{\delta^1} & [1;0] \ar[l]|-{\sigma^0} \ar@<-1.5ex>[r]_-{(\mathrm{id};\delta^0)} \ar@<1.5ex>[r]^-{(\mathrm{id};\delta^1)} & [1;1] \ar[l]|-{(\mathrm{id};\sigma^0)} \ar[r]|-{(\mathrm{id};\delta^1)} \ar@<1.5ex>[r]^-{(\mathrm{id};\delta^2)} \ar@<-1.5ex>[r]_-{(\mathrm{id};\delta^0)} & [1;2]
}
\end{equation*}
implies that these functions respect the sources and targets of morphisms and $2$-cells, preserve identity morphisms and identity $2$-cells, and preserve vertical composition of $2$-cells, and thus define a morphism from the underlying reflexive $\Cat$-graph of $A$ to that of $B$. 

Now, let us show that the action of $F$ on general $[2;0,0]$-elements (``invertible $2$-simplices'') is given as in (\ref{2simpmap}) in terms of its action on $2$-cells, the above defined ``composition constraints'' $\varphi$, and vertical composition in the bicategory $B$. Let $\sigma \colon gf \lra h$ be an invertible $2$-simplex in $A$ as displayed on the left below, and let $\overline{\sigma} \colon Fg.Ff \lra Fh$ denote the invertible $2$-cell in $B$ corresponding to the $[2;0,0]$-element of $NB$ to which $\sigma$ is sent by $F$.
\begin{equation*}
\cd{
& b \ar[dr]^-g \mydbldtwocell[0.6]{d}{\sigma}{\cong}\\
a \ar[ur]^-f \ar[rr]_-h &{} & c
}
\qquad
\longmapsto
\qquad
\cd{
& Fb \ar[dr]^-{Fg} \mydbldtwocell[0.6]{d}{\overline{\sigma}}{\cong}\\
Fa \ar[ur]^-{Ff} \ar[rr]_-{Fh} &{} & Fc
}
\end{equation*}
Note that the boundary morphisms of $\overline{\sigma}$ are as displayed due to the naturality of $F$ with respect to the morphisms (\ref{2simpface}). The invertible $2$-cell $\sigma$ gives rise to the $[2;1,0]$-element of $NA$ displayed below, 
\begin{equation*}
\cd[@=3em]{
& b \ar[dr]^-g \mydbldtwocell[0.6]{d}{\sigma}{\cong}  \\
a \uruppertwocell^f{=} \ar[ur]_f \ar[rr]_-h && c
}
\qquad
=
\qquad
\cd[@=3em]{
& b \ar[dr]^-g \ar@{}[d]|(0.45){\veq} \\
a \ar[ur]^-f \rrtwocell^{gf}_h{\sigma} & & c
}
\end{equation*}
which is sent by $F$ to the $[2;1,0]$-element of $NB$ displayed below,
\begin{equation*}
\cd[@=3em]{
& Fb \ar[dr]^-{Fg} \mydbldtwocell[0.6]{d}{\overline{\sigma}}{\cong}  \\
Fa \uruppertwocell^{Ff\,}{=} \ar[ur]_{Ff} \ar[rr]_-{Fh} && Fc
}
\qquad
=
\qquad
\cd[@=3em]{
& Fb \ar[dr]^-{Fg} \mydbldtwocell[0.45]{d}{\varphi}{\cong} \\
Fa \ar[ur]^-{Ff} \rrtwocell^{F(gf)}_{Fh}{\ \ F\sigma} & & Fc
}
\end{equation*}
whose boundary faces are determined by the naturality of $F$ with respect to the morphisms
\begin{equation} \label{coneface}
\cd[@=3em]{
[1;1] \ar@<1ex>[r]^-{(\delta^2;\mathrm{id})} \ar@<-1ex>[r]_-{(\delta^1;\mathrm{id})} & [2;1,0] & [2;0,0]. \ar@<1ex>[l]^-{(\mathrm{id};\delta^0,\mathrm{id})} \ar@<-1ex>[l]_-{(\mathrm{id};\delta^1,\mathrm{id})} 
}
\end{equation}
This $[2;1,0]$-element of $NB$ asserts the desired equation in $B$:
\begin{equation} \label{keylem}
\overline{\sigma} = F\sigma \circ \varphi_{g,f}.
\end{equation}

It remains to verify the unit, naturality, and associativity axioms for the composition constraints $\varphi_{g,f}$ of the prospective normal pseudofunctor from $A$ to $B$.

By the identity (\ref{keylem}), the unit axioms are none other than the naturality of $F$ with respect to the morphisms
\begin{equation*}
\cd[@=3em]{
[1;0] & [2;0,0] \ar@<1ex>[l]^-{\sigma^0} \ar@<-1ex>[l]_-{\sigma^1},
}
\end{equation*}
which states that the following equations  hold for each morphism $f\colon a \lra b$ in $A$.
\begin{equation*}
\cd{
F1_b.Ff \ar@{}[dr]|-{\veq} \ar[r]^-{\varphi_{1_b,f}} \ar@{=}[d] & F(1_b.f) \ar[d]^-{F\bm{l}} \\
1_{Fb}.Ff \ar[r]_-{\bm{l}} & Ff
}
\qquad
\qquad
\cd{
Ff.F1_a \ar[r]^-{\varphi_{f,1_a}} \ar@{=}[d] \ar@{}[dr]|-{\veq} & F(f.1_a) \ar[d]^-{F\bm{r}} \\
Ff.1_{Fa} \ar[r]_-{\bm{r}} & Ff
}
\end{equation*}

It suffices to prove the naturality of the $2$-cells $\varphi_{g,f} \colon Fg.Ff \lra F(gf)$ in each variable separately. To prove naturality in the variable $f$, consider a diagram in $A$ as displayed below.
\begin{equation*}
\cd[@=3em]{
a \rtwocell^f_{f'}{\alpha}& b \ar[r]^-g & c
}
\end{equation*}
This diagram defines the $[2;1,0]$-element of $NA$ displayed below,
\begin{equation*}
\cd[@=3em]{
& b \ar[dr]^-g \ar@{}[d]|(0.6){\veq}  \\
a \uruppertwocell^f{\alpha} \ar[ur]_{f'} \ar[rr]_-{gf'} && c
}
\qquad
=
\qquad
\cd[@=3em]{
& b \ar[dr]^-g \ar@{}[d]|(0.45){\veq} \\
a \ar[ur]^-{f} \rrtwocell^{gf}_{gf'}{\ g\alpha} & & c
}
\end{equation*}
which is sent by $F$ to the $[2;1,0]$-element of $NB$ displayed below,
\begin{equation*}
\cd[@=3em]{
& Fb \ar[dr]^-{Fg} \mydbldtwocell[0.6]{d}{\varphi}{\cong}  \\
Fa \uruppertwocell^{Ff}{\ \ F\alpha} \ar[ur]_{Ff'} \ar[rr]_-{F(gf')} && Fc
}
\qquad
=
\qquad
\cd[@=3em]{
& Fb \ar[dr]^-{Fg} \mydbldtwocell[.45]{d}{\varphi}{\cong} \\
Fa \ar[ur]^-{Ff} \rrtwocell^{F(gf)}_{F(gf')}{\ \ \ \ F(g\alpha)} & & c
}
\end{equation*}
whose boundary faces are determined by the naturality of $F$ with respect to the morphisms (\ref{coneface}). This pasting equation in $B$ is precisely the statement of naturality of $\varphi_{g,f}$ in the variable $f$. The naturality of $\varphi_{g,f}$ in its other variable is proved similarly by the consideration of $[2;0,1]$-elements.

To verify the associativity axiom, let $f,g,h$ be a composable triple of morphisms in $A$. These define the $[3;0,0,0]$-element of $NA$ represented by the pasting equation displayed below,
\begin{equation*}
\cd[@=2em]{
b  \ar[ddrr]|-{hg} \ar[rr]^-g  && c \ar[dd]^-h \\
\\
a \ar[uu]^-f \ar[rr]_-{h(gf)} \mydbldtwocell[0.3]{uurr}{\bm{a}}{\cong} \ar@{}[uurr]|(.7){\veq} && d
}
\qquad
=
\qquad
\cd[@=2em]{
b \ar[rr]^-g \ar@{}[ddrr]|(.3){\veq} \ar@{}[ddrr]|(0.7){\veq}&& c \ar[dd]^-h \\
\\
a \ar[uu]^-f \ar[rr]_-{h(gf)} \ar[uurr]|-{gf} && d
}
\end{equation*}
in which $\bm{a} \colon (hg)f \lra h(gf)$ denotes the associativity constraint of the bicategory $A$, and which is to be read as the equation displayed below.
\begin{equation*}
\cd[@R=1.5em]{
(hg)f \ar@{=}[r] \ar[dd]_-{\bm{a}} & (hg)f \ar[dr]^-{\bm{a}} \\
{} \ar@{}[rr]|-{\veq}&& h(gf)\\
h(gf) \ar@{=}[r] & h(gf) \ar@{=}[ur]
}
\end{equation*}
This $[3;0,0,0]$-element of $NA$ is sent by $F$ to the $[3;0,0,0]$-element of $NB$ represented by the pasting equation displayed below,
 \begin{equation*}
\cd[@=2em]{
Fb  \ar[ddrr]|-{F(hg)} \ar[rr]^-{Fg}  && Fc \ar[dd]^-{Fh} \\
\\
Fa \ar[uu]^-{Ff} \ar[rr]_-{F(h(gf))} \mydbldtwocell[0.3]{uurr}{\overline{\bm{a}}}{\cong} \mydbldtwocell[0.7]{uurr}{\varphi}{\cong}&& d
}
\qquad
=
\qquad
\cd[@=2em]{
Fb \ar[rr]^-{Fg} \mydbldtwocell[0.3]{ddrr}{\varphi}{\cong} \mydbldtwocell[0.7]{ddrr}{\varphi}{\cong}&& Fc \ar[dd]^-{Fh} \\
\\
Fa \ar[uu]^-{Ff} \ar[rr]_-{F(h(gf))} \ar[uurr]|-{F(gf)} && Fd
}
\end{equation*}
whose faces are determined by the naturality of $F$ with respect to the morphisms
\begin{equation*}
\cd{
[2;0,0] \ar@<3ex>[r]^-{\delta^3} \ar@<1ex>[r]|-{\delta^2} \ar@<-1ex>[r]|-{\delta^1} \ar@<-3ex>[r]_-{\delta^0} & [3;0,0,0], 
}
\end{equation*}
and which is to be read  as the equation displayed below,
\begin{equation*}
\cd[@=3em]{
(Fh.Fg).Ff \ar[r]^-{\varphi.Ff} \ar[d]_-{\bm{a}} & F(hg).Ff \ar@{}[d]|-{\veq} \ar[r]^-{\varphi} & F((hg)f) \ar[d]^-{F\bm{a}} \\
Fh.(Fg.Ff) \ar[r]_-{Fh.\varphi} & Fh.F(gf) \ar[r]_-{\varphi} & F(h(gf))
}
\end{equation*}
where we have used the identity $\overline{\bm{a}} = F\bm{a}\circ \varphi$ (\ref{keylem}).
This equation is precisely the associativity axiom.

We have thus defined a  normal pseudofunctor $\widetilde{F}$ (say) from  $A$ to $B$. Since the $3$-dimensional elements of $NA$ and $NB$ are determined by their boundaries, to show that $F = N_b(\widetilde{F})$ it suffices to show that their components at the objects of degree $\leq 2$ of $\Theta_2$ are equal, which is immediate from the definition of $\widetilde{F}$ for the objects $[0$], $[1;0]$, and $[1;1]$, and follows from the definition of the composition constraints $\varphi$ and the identity (\ref{keylem}) for the object $[2;0,0]$. Hence the functor $N_b$ is full.
\end{proof}

We may now apply Lemma \ref{denselemma} to deduce the main result of this section.

\begin{theorem} \label{ffthm}
The coherent nerve functor $N \colon \Bicat \lra [\Theta_2^\mathrm{op},\Set]$ is fully faithful, and is isomorphic to the composite
\begin{equation*}
\cd[@C=3.5em]{
\Bicat \ar[r]^-{N_b} & [\Theta_b^\mathrm{op},\Set] \ar[r]^-{(i_b)_*} & [\Theta_2^\mathrm{op},\Set]
}
\end{equation*}
of the truncated coherent nerve functor $N_b$ of \textup{Proposition \ref{ffprop}} and the functor $(i_b)_*$ defined by right Kan extension along the opposite of the inclusion $i_b \colon \Theta_b \lra \Theta_2$.
\end{theorem}
\begin{proof}
Since the inclusion $\Theta_2 \lra \Bicat$ is full, and the inclusion $\Theta_b \lra \Bicat$ is dense by Proposition \ref{ffprop}, the result follows by Lemma \ref{denselemma}.
\end{proof}

\section{Two model structures} \label{modstrsec}
In this section, we recall the two model structures which will feature in the first main theorem of this paper (Theorem \ref{firstqadjthm}), namely Lack's model structure for bicategories on the category $\Bicat_\mathrm{s}$ of bicategories and strict morphisms (\S\ref{lackmodstrsubsect}) and Ara's model structure for $2$-quasi-categories on the category $[\Theta_2^\mathrm{op},\Set]$ of $\Theta_2$-sets (\S\ref{secara}), and prove a recognition principle (Proposition \ref{welemma}) for left Quillen functors from the latter.

We begin with Lack's model structure for bicategories, established in \cite{MR2138540}.

\begin{definition} \label{lackdefs}
A (normal) pseudofunctor  between bicategories $F \colon A \lra B$ is said to be:
\begin{enumerate}
\item a \emph{biequivalence} if:
\begin{enumerate}[(i)]
\item for each object $b \in B$, there exists an object $a \in A$ and an equivalence $Fa \simeq b$ in $B$, and
\item for each pair of objects $a,b \in A$, the functor $F \colon A(a,b) \lra B(Fa,Fb)$ is an equivalence of categories;
\end{enumerate}
\item an \emph{equifibration} if:
\begin{enumerate}[(i)]
\item for each object $a \in A$ and each equivalence $g \colon Fa \lra b$ in $B$, there exists an equivalence $f \colon a \lra a'$ in $A$ such that $Ff = g$, and
\item for each pair of objects $a,b \in A$, the functor $F \colon A(a,b) \lra B(Fa,Fb)$ is an isofibration of categories;
\end{enumerate}
\item a \emph{trivial fibration} if:
\begin{enumerate}[(i)]
\item for each object $b \in B$, there exists an object $a \in A$ such that $Fa = b$, and
\item for each pair of objects $a,b \in A$, the functor $F \colon A(a,b) \lra B(Fa,Fb)$ is a surjective-on-objects equivalence of categories.
\end{enumerate}
\end{enumerate}
It is easily shown that a (normal) pseudofunctor is a trivial fibration if and only if it is both a biequivalence and an equifibration.
\end{definition}

\begin{observation}[free-living isomorphisms and equifibrations] \label{equifibobs}
Let $\mathbb{I}$ denote the ``free-living isomorphism'', i.e.\ the contractible groupoid with two objects $0,1$. Let $\mathbb{I}_2$ denote the ``free-living invertible $2$-cell'', i.e.\ the 2-category with two objects $\bot,\top$ and hom-categories $\mathbb{I}_2(\bot,\bot) = \{\mathrm{id}_{\bot}\}$, $\mathbb{I}_2(\top,\top) = \{\mathrm{id}_{\top}\}$, $\mathbb{I}_2(\bot,\top) = \mathbb{I}$, and $\mathbb{I}_2(\top,\bot) = \emptyset$. 
For any bicategory $B$, a normal pseudofunctor $\mathbb{I} \lra B$ amounts precisely to an adjoint equivalence in $B$; a normal pseudofunctor $\mathbb{I}_2 \lra B$ is necessarily strict, and amounts to an invertible $2$-cell in $B$.

The proof of \cite[Proposition 6]{MR2138540} shows that a normal pseudofunctor is an equifibration if and only if it has the right lifting property in the category $\Bicat$ with respect to the functor  $\mathbf{1} \lra \mathbb{I}$ that picks out the object $0$ of $\mathbb{I}$ and the $2$-functor $\mathbf{2} \lra \mathbb{I}_2$ that picks out the morphism $0 \colon \bot \lra \top$ in $\mathbb{I}_2$.  
\end{observation}

\subsection{The category of bicategories and strict morphisms} \label{qrecall}
Lack's model structure for bicategories is defined not on the category $\Bicat$ of bicategories and normal pseudofunctors -- which is neither complete nor cocomplete (cf.\ \cite[Example 4.5]{MR1931220}) -- but on its subcategory $\Bicat_\mathrm{s}$ of bicategories and strict morphisms (see \S\ref{bicatrecall}). 
It can be proved (for instance using two-dimensional monad theory \cite{MR1007911}) that 
the category $\Bicat_\mathrm{s}$ is locally finitely presentable, that the inclusion functor $\Bicat_\mathrm{s} \lra \Bicat$ has a left adjoint
\begin{equation} \label{prestrictadj}
\xymatrix{
\Bicat_\mathrm{s} \ar@<-1.5ex>[rr]^-{\hdash}_-{} && \ar@<-1.5ex>[ll]_-{Q} \Bicat
}
\end{equation}
which sends a bicategory $A$ to its \emph{normal pseudofunctor classifier} $QA$, and that each component of the unit and counit of this adjunction is a (bijective-on-objects) biequivalence. It follows that the functor $Q \colon \Bicat \lra \Bicat_\mathrm{s}$ preserves and reflects biequivalences.

\subsection{The model structure for bicategories} \label{lackmodstrsubsect}
In \cite{MR2138540}, Lack constructs  a model structure on the category $\Bicat_\mathrm{s}$ in which a strict morphism of bicategories is a weak equivalence, fibration, or trivial fibration precisely when it a biequivalence, equifibration, or trivial fibration  (in the sense of Definition \ref{lackdefs}) respectively. 
Every bicategory is fibrant in this model structure, but not every bicategory is cofibrant; for each bicategory $A$, the counit morphism $QA \lra A$ of the adjunction (\ref{prestrictadj}) is a cofibrant replacement of $A$. 

\subsection{Biequivalences between cofibrant bicategories} \label{unitobs}
In \cite[\S1]{MR2138540}, Lack describes an adjunction
\begin{equation} \label{lackqequiv}
\xymatrix{
\twocat \ar@<-1.5ex>[rr]^-{\hdash}_-{} && \ar@<-1.5ex>[ll]_-L \Bicat_\mathrm{s}
}
\end{equation}
whose right adjoint is the full inclusion of the category $\twocat$ of $2$-categories and $2$-functors into  $\Bicat_\mathrm{s}$. It is proved in \cite[Lemma 10]{MR2138540} that, for every cofibrant bicategory $A$, the unit morphism $A \lra LA$ is a biequivalence. It follows that a strict morphism between cofibrant bicategories is a biequivalence if and only if it sent by the functor $L \colon \Bicat_\mathrm{s} \lra \twocat$ to a biequivalence.
Moreover, by \cite[Theorems 4 and 11]{MR2138540}, the category $\twocat$ admits a model structure (right-induced from the model structure for bicategories along the inclusion $\twocat \lra \Bicat_\mathrm{s}$), with respect to which the adjunction (\ref{lackqequiv}) is a Quillen equivalence. 

\subsection{Normal strictification} \label{nstrecall}
Let 
\begin{equation} \label{nstadj}
\xymatrix{
\twocat \ar@<-1.5ex>[rr]^-{\hdash}_-{} && \ar@<-1.5ex>[ll]_-{\st} \Bicat
}
\end{equation}
denote the composite of the adjunctions (\ref{lackqequiv}) and (\ref{prestrictadj}). The left adjoint of this adjunction sends a bicategory $A$ to its \emph{normal strictification} $\st A$. It follows from the previous paragraphs that each component of the unit and counit of this adjunction is a biequivalence.
\medskip

 We now turn to Ara's model structure for $2$-quasi-categories, which was constructed in \cite{MR3350089} using Cisinski's theory of localisers \cite{MR2294028}.
 
 \subsection{Localisers and Cisinski model structures} \label{cisinski}
 Let $\mathbf{A}$ be a small category. A morphism in the presheaf category $[\mathbf{A}^\mathrm{op},\Set]$ is said to be a \emph{trivial fibration} if it has the right lifting property with respect to all monomorphisms in $[\mathbf{A}^\mathrm{op},\Set]$. An \emph{$\mathbf{A}$-localiser} \cite[D\'efinition 1.4.1]{MR2294028} is a class $\mathsf{W}$ of morphisms in $[\mathbf{A}^\mathrm{op},\Set]$ such that:
 \begin{enumerate}
 \item $\mathsf{W}$ satisfies the two-out-of-three property,
 \item every trivial fibration belongs to $\mathsf{W}$,
 \item the class of monomorphisms belonging to $\mathsf{W}$ is stable under pushout and transfinite composition.
 \end{enumerate}
Any intersection of $\mathbf{A}$-localisers is an $\mathbf{A}$-localiser. An $\mathbf{A}$-localiser $\mathsf{W}$ is said to be \emph{accessible} if there exists a \emph{small} set of morphisms $S$ in $[\mathbf{A}^\mathrm{op},\Set]$ such that $\mathsf{W}$ is the smallest $\mathbf{A}$-localiser containing $S$ (i.e.\ the intersection of all $\mathbf{A}$-localisers containing $S$).
 
 Let $S$ be a small set of morphisms in $[\mathbf{A}^\mathrm{op},\Set]$. Cisinski proves \cite[Th\'eor\`eme 1.4.3]{MR2294028} that there exists a (necessarily unique) model structure on $[\mathbf{A}^\mathrm{op},\Set]$ whose cofibrations are the monomorphisms, and whose class of weak equivalences is the smallest $\mathbf{A}$-localiser containing $S$. We will call this model structure the \emph{Cisinski model structure generated by $S$}. (Conversely, Cisinski shows that every cofibrantly generated model structure on $[\mathbf{A}^\mathrm{op},\Set]$ whose cofibrations are the monomorphisms arises in this way.)

\begin{notation}[free-living isomorphisms] \label{freeliving}
Let $J$ denote the $2$-cellular nerve of the free-living isomorphism $\mathbb{I}$, and let $\delta_J \colon \partial J \lra J$ denote the nerve of the inclusion $\{0,1\} \lra \mathbb{I}$. Furthermore, let $J_2$ denote the strict $2$-cellular nerve of the free-living invertible $2$-cell $\mathbb{I}_2$ (which we note is not equal to the coherent $2$-cellular nerve of $\mathbb{I}_2$). Let $s_1 \colon \Theta_2[1;0] \lra J_2$ denote the strict nerve of the $2$-functor $\mathbf{2} \lra \mathbb{I}_2$ that picks out the morphism $0 \colon \bot \lra \top$ in $\mathbb{I}_2$, and let $j_2 \colon J_2 \lra \Theta_2[1;0]$ denote the strict nerve of the unique bijective-on-objects $2$-functor $\mathbb{I}_2 \lra \mathbf{2}$. 
\end{notation}

\subsection{The model structure for $2$-quasi-categories} \label{secara}
In \cite[\S5.17]{MR3350089}, Ara defines the \emph{model structure for $2$-quasi-categories}  to be the Cisinski model structure (see \S\ref{cisinski}) on the category $[\Theta_2^\mathrm{op},\Set]$ generated by the spine inclusions $i_{n(\bm{m})} \colon I[n;\bm{m}] \lra \Theta_2[n;\bm{m}]$ for all $[n;\bm{m}] \in \Theta_2$ (see \S\ref{spinerecall}) and the morphism $j_2 \colon J_2 \lra \Theta_2[1;0]$ (see \S\ref{freeliving}). The fibrant objects of this model structure are called \emph{$2$-quasi-categories}. 

By definition, the cofibrations in the model structure for $2$-quasi-categories are the monomorphisms in $[\Theta_2^\mathrm{op},\Set]$. Hence every $\Theta_2$-set is cofibrant, and (by \S\ref{bdyrecall}) a morphism of $\Theta_2$-sets is a trivial fibration if and only if it has the right lifting property with respect to the boundary inclusions $\partial\Theta_2[n;\bm{m}] \lra \Theta_2[n;\bm{m}]$ for every $[n;\bm{m}] \in\Theta_2$.

\medskip
We now deduce from Cisinski's theory of localisers the following recognition principle for left Quillen functors from the model category of $2$-quasi-categories, which we will use in \S\ref{mainthmsec} to prove our first main theorem.

\begin{proposition} \label{welemma}
Let $\mathcal{M}$ be a model category and let  $F \colon [\Theta_2^\mathrm{op},\Set] \lra \mathcal{M}$ be a cocontinuous functor that sends monomorphisms to cofibrations. Then $F$ sends the weak equivalences in the model structure for $2$-quasi-categories to weak equivalences in $\mathcal{M}$ if and only if it sends  the following morphisms to weak equivalences in $\mathcal{M}$:
\begin{enumerate}[font=\normalfont,  label=(\roman*)]
\item for each $[n;\bm{m}] \in \Theta_2$, the projection $\mathrm{pr}_2 \colon J \times \Theta_2[n;\bm{m}] \lra \Theta_2[n;\bm{m}]$,
\item for each $[n;\bm{m}] \in \Theta_2$, the spine inclusion $i_{n(\bm{m})} \colon I[n;\bm{m}] \lra \Theta_2[n;\bm{m}]$, and
\item the morphism  $j_2 \colon J_2 \lra \Theta_2[1;0]$.
\end{enumerate}
\end{proposition}
\begin{proof}
We first show that the morphisms (i)--(iii) listed above are weak equivalences in the model structure for $2$-quasi-categories. For the morphisms (ii)--(iii), this is immediate from the definition of  this model structure. For the morphisms (i), injectivity of $J$ (see \cite[Corollary 6.7]{MR3350089} or Proposition \ref{trivfibprop} below) implies that the projections $\mathrm{pr}_2 \colon J \times \Theta_2[n;\bm{m}] \lra \Theta_2[n;\bm{m}]$ are trivial fibrations, and hence weak equivalences. This proves necessity.

To prove sufficiency, let $\mathsf{W}$ denote the class of morphisms in $[\Theta_2^\mathrm{op},\Set]$ that are sent by the functor $F$ to weak equivalences in $\mathcal{M}$. Since $\mathsf{W}$ contains the morphisms (ii)-(iii) listed above, it suffices by the definition of the weak equivalences in the model structure for $2$-quasi-categories to show that $\mathsf{W}$ is a $\Theta_2$-localiser. It thus suffices to show that $\mathsf{W}$ satisfies the hypotheses of \cite[Proposition 8.2.15]{MR2294028}, which we do as follows:\ by \S\ref{reedyrecall}, $\Theta_2$ is a \emph{cat{\'e}gorie squelletique r{\'e}guli{\`e}re}; by \cite[Exemple 1.3.8]{MR2294028}, the object $J$ (together with its endpoints $\{0\},\{1\} \colon \Theta_2[0] \lra J$) define a \emph{donn\'ee homotopique \'el\'ementaire} on $\Theta_2$; since $F$ is cocontinuous and sends monomorphisms to cofibrations, $\mathsf{W}$ is a $\Theta_2$-\emph{pr\'elocalisateur} \cite[D\'efinition 8.2.10]{MR2294028}; by hypothesis, the class $\mathsf{W}$ contains the projections $J \times \Theta_2[n;\bm{m}] \lra \Theta_2[n;\bm{m}]$.
\end{proof}

We conclude this section with a couple of remarks which, though they will not be used in this paper, may be of interest to the reader.

\begin{remark}
It follows from the results of \cite{yukihorn} that, in the statement of Proposition \ref{welemma}, the set of morphisms (i)  can be replaced by the single morphism $J \lra \Theta_2[0]$.
\end{remark}

\begin{remark} \label{elcatex}
Cisinski's theory of localisers implies the following characterisation of $2$-quasi-categories in terms of a lifting property \cite[Theorem 2.14]{MR3350089}. A $\Theta_2$-set is a $2$-quasi-category if and only if it has the right lifting property with respect to the following morphisms  (where $\delta_{J^k} \colon \partial J^k \lra J^k$ denotes the $k$-fold Leibniz product (i.e.\ pushout-product) of the morphism $\delta_J \colon \partial J \lra J$ with itself):
\begin{enumerate}[(i)]
\item for each $[n;\bm{m}] \in \Theta_2$ and $\varepsilon \in \{0,1\}$, the Leibniz product 
\begin{equation*}
\{\varepsilon\} \widehat{\times} \delta_{n(\bm{m})} \colon (\Theta_2[0]\times \Theta_2[n;\bm{m}]) \cup (J \times \partial \Theta_2[n;\bm{m}]) \lra J \times \Theta_2[n;\bm{m}],
\end{equation*}
\item for each $k \geq 0$ and $[n;\bm{m}] \in \Theta_2$, the Leibniz product
\begin{equation*}
\delta_{J^k} \widehat{\times} i_{n(\bm{m})} \colon (\partial J^k \times \Theta_2[n;\bm{m}]) \cup (J^k \times I[n;\bm{m}]) \lra J^k \times \Theta_2[n;\bm{m}],
\end{equation*}
\item for each $k \geq 0$, the Leibniz product
\begin{equation*}
\delta_{J^k} \widehat{\times} s_1 \colon (\partial J^k \times J_2) \cup (J^k \times \Theta_2[1;0]) \lra J^k \times J_2.
\end{equation*}
\end{enumerate}
Alternative characterisations of $2$-quasi-categories in terms of lifting properties may be found in \cite{yukihorn}. 
\end{remark}

\section{From bicategories to $2$-quasi-categories} \label{mainthmsec}
In this section, we prove the first main theorem of this paper:\ that the coherent nerve functor studied in \S\ref{nervefunsec} restricts to the right part of a Quillen adjunction between Lack's model structure on $\Bicat_\mathrm{s}$ and Ara's model structure for $2$-quasi-categories on $[\Theta_2^\mathrm{op},\Set]$ (both recalled in \S\ref{modstrsec}),  
and hence that the coherent nerve of a bicategory is a $2$-quasi-category.

\subsection{A coherent nerve adjunction}
It follows from the adjunction (\ref{prestrictadj}) that the restriction of the coherent nerve functor $N \colon \Bicat \lra [\Theta_2^\mathrm{op},\Set]$ along the inclusion $\Bicat_\mathrm{s} \lra \Bicat$ is naturally isomorphic to the singular functor induced by the composite functor 
\begin{equation*}
\cd{
\Theta_2 \ar[r] & \Bicat \ar[r]^-{Q} & \Bicat_\mathrm{s}, 
}
\end{equation*}
which sends each object of $\Theta_2$ to its normal pseudofunctor classifier (\emph{qua} bicategory). Since the category $\Bicat_\mathrm{s}$ is cocomplete, it then follows by Kan's construction (Recollection \ref{kanrecall}) that this restricted coherent nerve functor has a left adjoint, as displayed below.
\begin{equation} \label{firstlookqadj}
\xymatrix{
\Bicat_\mathrm{s} \ar@<-1.5ex>[rr]^-{\hdash}_-{N} && \ar@<-1.5ex>[ll]_-{\tau_b} [\Theta_2^\mathrm{op},\Set]
}
\end{equation}

\begin{remark}
The coherent nerve functor $N \colon \Bicat \lra [\Theta_2^\mathrm{op},\Set]$ does not have a left adjoint. For the value of such a left adjoint at the spine $I[2;0,0]$ would represent the functor $\Bicat \lra \Set$ that sends a bicategory $A$ to the set of composable pairs of morphisms in $A$; but this functor is not representable (cf.\ \cite[Example 4.5]{MR1931220}). Nevertheless,  we will show in \S\ref{hobicatsec} that the coherent nerve functor does have a partial left adjoint, defined on the full subcategory of $2$-quasi-categories. 
\end{remark}

\begin{observation} \label{coobs} 
Theorem \ref{ffthm} (together with the Yoneda lemma) implies that the composite functor 
\begin{equation*}
\xymatrix{
\Bicat \ar[r]^-N & [\Theta_2^\mathrm{op},\Set] \ar[r]^-{\tau_b} & \Bicat_\mathrm{s}
}
\end{equation*}
is naturally isomorphic to the normal pseudofunctor classifier functor $Q \colon \Bicat \lra \Bicat_\mathrm{s}$ (see \S\ref{qrecall}). 
Furthermore, let 
\begin{equation} \label{2catqadj}
\xymatrix{
\twocat \ar@<-1.5ex>[rr]^-{\hdash}_-{N} && \ar@<-1.5ex>[ll]_-{\tau_2} [\Theta_2^\mathrm{op},\Set]
}
\end{equation}
denote the composite of the adjunctions (\ref{lackqequiv}) and (\ref{firstlookqadj}). Then, as above, the composite functor
\begin{equation*}
\xymatrix{
\Bicat \ar[r]^-N & [\Theta_2^\mathrm{op},\Set] \ar[r]^-{\tau_2} & \twocat
}
\end{equation*}
is naturally isomorphic to the normal strictification functor $\st \colon \Bicat \lra \twocat$ (see \S\ref{nstrecall}). Hence the components of the counits of the adjunctions (\ref{firstlookqadj}) and (\ref{2catqadj}) are biequivalences. 
\end{observation}

The goal of this section is to prove that the adjunction (\ref{firstlookqadj}) is a Quillen adjunction between Lack's model structure for bicategories and Ara's model structure for $2$-quasi-categories. To achieve this goal, it will suffice to prove that the left adjoint functor $\tau_b \colon [\Theta_2^\mathrm{op},\Set] \lra \Bicat_{\mathrm{s}}$ satisfies the hypotheses of Proposition \ref{welemma}.

\begin{proposition} \label{trivfibprop}
The coherent nerve functor $N \colon \Bicat \lra [\Theta_2^\mathrm{op},\Set]$ preserves and reflects trivial fibrations. Hence the functor $\tau_b \colon [\Theta_2^\mathrm{op},\Set] \lra \Bicat_\mathrm{s}$ sends monomorphisms to cofibrations.
\end{proposition}
\begin{proof}
Let $F\colon A \lra B$ be a normal pseudofunctor between bicategories. By \S\ref{secara}, the morphism $NF \colon NA\lra NB$ in $[\Theta_2^\mathrm{op},\Set]$ is a trivial fibration if and only if it has the right lifting property with respect to the boundary inclusion $\delta_{n(\bm{m})} \colon \partial\Theta_2[n;\bm{m}] \lra \Theta_2[n;\bm{m}]$ for every $[n;\bm{m}] \in \Theta_2$. By adjointness and Theorem \ref{ffthm}, this is so if and only if the morphism $N_bF \colon N_bA \lra N_bB$ in $[\Theta_b^\mathrm{op},\Set]$ has the right lifting property with respect to the morphism $i_b^*(\delta_{n(\bm{m})})$  for every $[n;\bm{m}] \in \Theta_2$ of degree $\leq 3$.

First, observe (with the aid of Figure \ref{pasts} in \S\ref{nervefunsec}) that the morphism $N_bF$ has the right lifting property with respect to the morphisms  $i_b^*(\delta_{n(\bm{m})})$ for $[n;\bm{m}] = [0]$, $[1;0]$, $[1;1]$, and $[1;2]$ precisely when the normal pseudofunctor $F$ is respectively surjective on objects, full on morphisms, full on $2$-cells, and faithful on $2$-cells, that is, precisely when $F$ is a trivial fibration in $\Bicat$. 

Therefore, to prove the first statement of the proposition, it remains to show that if $F$ is a trivial fibration in $\Bicat$, then $N_bF$ has the right lifting property in $[\Theta_b^\mathrm{op},\Set]$ with respect to the morphisms  $i_b^*(\delta_{n(\bm{m})})$ for $[n;\bm{m}] = [2;0,0]$, $[2;1,0]$, $[2;0,1]$, and $[3;0,0,0]$. For $[n;\bm{m}] = [2;0,0]$, the lifting property is satisfied because $F$ is full and conservative on $2$-cells. For the remaining objects of degree $3$, the lifting property is satisfied because $F$ is faithful on $2$-cells.

Hence the functor $N \colon \Bicat_\mathrm{s} \lra [\Theta_2^\mathrm{op},\Set]$ preserves trivial fibrations, and so by adjointness, its left adjoint $\tau_b \colon [\Theta_2^\mathrm{op},\Set] \lra \Bicat_\mathrm{s}$ sends monomorphisms to cofibrations.
\end{proof}

\begin{observation} \label{spineobs}
The free--forgetful adjunction (\ref{fuadj}) between the categories of $2$-categories and $2$-graphs is isomorphic to the composite
\begin{equation*}
\xymatrix{
\twocat  \ar@<-1.5ex>[rr]^-{\hdash}_-{N} && \ar@<-1.5ex>[ll]_-{\tau_2} [\Theta_2^\mathrm{op},\Set] \ar@<-1.5ex>[rr]^-{\hdash}_-{g^{\ast}} && \ar@<-1.5ex>[ll]_-{g_!} [\mathbb{G}_2^\mathrm{op},\Set]
}
\end{equation*}
of the adjunctions (\ref{2catqadj}) and (\ref{jadj}), as may be confirmed by comparing right adjoints.
A simple calculation (of universal properties) then shows that, for each $[n;\bm{m}] \in \Theta_2$, the spine inclusion $I[n;\bm{m}] \lra \Theta_2[n;\bm{m}]$ is (isomorphic to) the component of the unit of the adjunction $\tau_2 \dashv N$ at the $\Theta_2$-set $I[n;\bm{m}]$.
\end{observation}

\begin{lemma} \label{unitlem}
The functor $\tau_b \colon [\Theta_2^\mathrm{op},\Set] \lra \Bicat_\mathrm{s}$ sends each component of the unit of the adjunction $\tau_2 \dashv N \colon \twocat \lra [\Theta_2^\mathrm{op},\Set]$ to a biequivalence.
\end{lemma}
\begin{proof}
Proposition \ref{trivfibprop} implies that every bicategory in the image of the functor $\tau_b$ is cofibrant. Hence it suffices by \S\ref{unitobs} to show that the composite functor $\tau_2 = L\tau_b \colon [\Theta_2^\mathrm{op},\Set] \lra \twocat$ sends each component of the unit of the adjunction $\tau_2 \dashv N \colon \twocat \lra [\Theta_2^\mathrm{op},\Set]$ to a biequivalence. This follows from the two-out-of-three property, the triangle identity $\varepsilon_{\tau_2 X} \circ \tau_2(\eta_X) = 1_{\tau_2 X}$,
and the fact that the counit morphism $\varepsilon_A \colon \tau_2NA \lra A$ is a biequivalence for every $2$-category $A$  (see Observation \ref{coobs}).%
\end{proof}

\begin{proposition} \label{weprop}
The functor $\tau_b \colon [\Theta_2^\mathrm{op},\Set] \lra \Bicat_\mathrm{s}$ sends the weak equivalences in the model structure for $2$-quasi-categories to biequivalences. 
\end{proposition}
\begin{proof}
Since the left adjoint functor $\tau_b$ is cocontinuous and sends monomorphisms to cofibrations by Proposition \ref{trivfibprop}, it suffices to show that it sends the morphisms (i)--(iii) listed in Proposition \ref{welemma} to biequivalences.

(i) For each $[n;\bm{m}] \in \Theta_2$, the morphism $J \times \Theta_2[n;\bm{m}] \lra \Theta_2[n;\bm{m}]$ is (isomorphic to) the coherent nerve of the trivial fibration $\mathbb{I} \times [n;\bm{m}] \lra [n;\bm{m}]$, and is therefore (by Observation \ref{coobs})  sent by $\tau_b$ to  the morphism $Q(\mathbb{I} \times [n;\bm{m}] \lra [n;\bm{m}])$, which is a biequivalence by \S\ref{qrecall}. 

(ii) For each $[n;\bm{m}] \in \Theta_2$, the spine inclusion $I[n;\bm{m}] \lra \Theta_2[n;\bm{m}]$ is (isomorphic to) the component of the unit of the composite adjunction $L\tau_b \dashv NI$ at the object $I[n;\bm{m}]$ by Observation \ref{spineobs}, and hence, by Lemma \ref{unitlem}, is
sent by the functor $\tau_b$ to a biequivalence.

(iii) A direct calculation shows that the morphism $j_2 \colon J_2 \lra \Theta_2[1;0]$ is sent by the functor $\tau_b$ to the bijective-on-objects biequivalence $\mathbb{I}_2 \lra \mathbf{2}$.
\end{proof}

\begin{theorem} \label{firstqadjthm}
The adjunction
\begin{equation*} 
\xymatrix{
\Bicat_\mathrm{s} \ar@<-1.5ex>[rr]^-{\hdash}_-N && \ar@<-1.5ex>[ll]_-{\tau_b} [\Theta_2^\mathrm{op},\Set]}
\end{equation*}
 is a Quillen adjunction between Lack's model structure for bicategories and Ara's model structure for $2$-quasi-categories. The derived right adjoint of this Quillen adjunction is fully faithful.
\end{theorem}
\begin{proof}
By Propositions \ref{trivfibprop} and \ref{weprop}, the left adjoint of this adjunction preserves cofibrations and weak equivalences. Hence the adjunction is a Quillen adjunction. 

To prove the second statement, recall from Observation \ref{coobs} that each component of the counit of this adjunction is a biequivalence.  
Since every object in the model structure for bicategories is fibrant and every object in the model structure for $2$-quasi-categories is cofibrant, this implies that each component of the derived counit of this Quillen adjunction is an isomorphism, and hence that its derived right adjoint is fully faithful. 
\end{proof}

\begin{corollary} \label{fibcor}
The coherent nerve of a bicategory is a $2$-quasi-category. 
\end{corollary}
\begin{proof}
By Theorem \ref{firstqadjthm}, the functor $N \colon \Bicat_{\mathrm{s}} \lra [\Theta_2^\mathrm{op},\Set]$ is a right Quillen functor, and therefore preserves fibrant objects. 
\end{proof}

\subsection{The full embedding of bicategories into $2$-quasi-categories}
Let $\twoqcat$ denote the full subcategory of $[\Theta_2^\mathrm{op},\Set]$ consisting of the $2$-quasi-categories. The following result is a combination of (the first part of) Theorem \ref{ffthm} and Corollary \ref{fibcor}.
\begin{corollary} \label{lastcor}
The coherent nerve construction defines a fully faithful functor $$N \colon \Bicat \lra \twoqcat$$ from the category of bicategories and normal pseudofunctors to the category of $2$-quasi-categories.
\end{corollary}

\begin{remark}
Every bicategory is a retract of a $2$-category (e.g.\ its normal strictification)  in the category $\Bicat$. Hence, by application of the coherent nerve functor $N \colon \Bicat \lra [\Theta_2^\mathrm{op},\Set]$, the coherent nerve of any bicategory is a retract of the coherent nerve of a $2$-category. Therefore, to prove Corollary \ref{fibcor}, it would have sufficed to prove that the adjunction (\ref{2catqadj}) is a Quillen adjunction between Lack's model structure for $2$-categories and Ara's model structure for $2$-quasi-categories. 
\end{remark}

\begin{observation} \label{transferobs}
Since every bicategory is fibrant in Lack's model structure on $\Bicat_\mathrm{s}$, Theorem \ref{firstqadjthm} implies that the functor $N \colon \Bicat_\mathrm{s} \lra [\Theta_2^\mathrm{op},\Set]$ preserves fibrations and preserves and reflects weak equivalences. Since $N$ sends the morphisms $\mathbf{1} \lra \mathbb{I}$ and $\mathbf{2} \lra \mathbb{I}_2$ mentioned in Observation \ref{equifibobs} to the trivial cofibrations $\{0\} \lra J$ and $s_1 \colon \Theta_2[1;0] \lra J_2$ (see Notation \ref{freeliving}) in the model structure for $2$-quasi-categories, Theorem \ref{ffthm} and Observation \ref{equifibobs} imply that the functor $N \colon \Bicat_\mathrm{s} \lra [\Theta_2^\mathrm{op},\Set]$ also reflects fibrations. This shows that Lack's model structure for bicategories is right-induced from Ara's model structure for $2$-quasi-categories along the functor $N \colon \Bicat_\mathrm{s} \lra [\Theta_2^\mathrm{op},\Set]$.
\end{observation}

\begin{remark} \label{part1endremark}
Since the adjunction (\ref{lackqequiv}) is a Quillen equivalence between Lack's model structures for $2$-categories and bicategories, the conclusions of Theorem \ref{firstqadjthm} and Observation \ref{transferobs} hold also for the adjunction (\ref{2catqadj}) and Lack's model structure for $2$-categories in the places of the adjunction (\ref{firstlookqadj}) and Lack's model structure for bicategories. 
\end{remark}

\section{The homotopy bicategory of a $2$-quasi-category} \label{hobicatsec}
In this section, we construct the \emph{homotopy bicategory} of a $2$-quasi-category (\S\ref{hobicatsubsec}), which we prove (Theorem \ref{honadjthm}) defines a left adjoint $\Ho \colon \twoqcat \lra \Bicat$ to the coherent nerve functor $N \colon \Bicat \lra \twoqcat$ studied in the previous sections. We perform this construction in three steps: we first define the \emph{underlying bisimplicial set} of a $2$-quasi-category (\S\ref{undbisimp}), which we prove to be a \emph{quasi-category-enriched Segal category} (Theorem \ref{segalthm}); we then take the \emph{homotopy Tamsamani $2$-category} (\S\ref{changeofbase}) of this quasi-category-enriched Segal category (i.e.\ its change of base along the homotopy category functor $\ho \colon \qcat \lra \Cat$); finally, we take the \emph{bicategory reflection} (as constructed by Tamsamani \cite{MR1673923}) of this Tamsamani $2$-category (\S\ref{bicatreflsec}). We deduce the universal property of this homotopy bicategory construction from the universal property (as formulated and proved by Lack and Paoli \cite{MR2366560}) of the bicategory reflection construction. 

\begin{remark}
Alternatively, one can construct the homotopy bicategory $\Ho(X)$ of a $2$-quasi-category $X$ in a similar fashion to the standard construction of the homotopy category of a quasi-category (see \cite[\S4.2]{MR0420609} or \cite[\S1.2.3]{MR2522659}). One defines the objects, morphisms, and $2$-cells of $\Ho(X)$ to be the $[0]$-elements, $[1;0]$-elements, and ``homotopy classes'' of $[1;1]$-elements of $X$, and defines their sources, targets, and identities using the face and degeneracy operations in $X$. Next, by lifting $X$ against suitable trivial cofibrations in the model structure for $2$-quasi-categories, and by considering suitable degenerate elements, one can define all the remaining data and prove all the axioms required of a bicategory:\ the operations of composition of morphisms and of  vertical and horizontal composition of $2$-cells, the interchange law, the associativity and unit constraints, and the coherence and naturality of the latter. Furthermore, one can similarly construct a unit morphism $X \lra N(\Ho X)$, and prove its universal property.

However, actually carrying out these constructions \emph{by hand}, as one does for the homotopy category of a quasi-category, is a long and tedious process (though not an impossible one). 
This is due not only to the length of the definition of a bicategory, but also to the complicated combinatorics of $\Theta_2$-sets, and to the higher dimensional nature of these structures.

It is therefore fortunate that a more elegant proof is possible, which relies not on ad hoc constructions and calculations but on existing theory. In the proof we present in this section, we use Reedy category theory to produce our ``suitable'' trivial cofibrations as relative latching maps 
derived from certain ``horizontal spine inclusions'' (see \S\ref{horspsubsect}), which are easily proven to be trivial cofibrations. Using these trivial cofibrations, we prove that the underlying bisimplicial set of a $2$-quasi-category $X$ is a quasi-category-enriched Segal category (which result will have further applications in \S\S\ref{araconsec}--\ref{sec1to2}), and hence extract from $X$ a ``Tamsamani $2$-category''; we are then able to use the constructions and results of \cite{MR1673923} and \cite{MR2366560} to construct both  the homotopy bicategory $\Ho(X)$ and the unit morphism $X \lra N(\Ho X)$, and to prove their universal property.
\end{remark}

We begin with the definition of the hom-quasi-categories of a $2$-quasi-category $X$, whose homotopy categories will be the hom-categories of the homotopy bicategory of $X$.

\subsection{The suspension-hom adjunction} \label{susphomadjsubsect}
Let $\sigma \colon \Delta \lra \partial\Theta_2[1;0]/[\Theta_2^\mathrm{op},\Set]$ denote the functor that sends each $[m] \in \Delta$ to the bijective-on-$[0]$-elements inclusion $\partial\Theta_2[1;0] \lra \Theta_2[1;m]$.  By Kan's construction (Recollection \ref{kanrecall}), this functor induces an adjunction
\begin{equation}  \label{susphomadjlastmin}
\xymatrix{
\partial\Theta_2[1]/[\Theta_2^\mathrm{op},\Set] \ar@<-1.5ex>[rr]^-{\hdash}_-{\Hom} && \ar@<-1.5ex>[ll]_-{\Sigma} [\Delta^\mathrm{op},\Set]
}
\end{equation}
between the category of bi-pointed $\Theta_2$-sets (note that $\partial\Theta_2[1;0] \cong \Theta_2[0] + \Theta_2[0]$) and the category of simplicial sets. The left adjoint of this adjunction sends a simplicial set $S$ to the bi-pointed $\Theta_2$-set $(\bot,\top) \colon \partial\Theta_2[1;0] \lra \Sigma(S)$, which we call the \emph{\textup{(}$\Theta_2$-\textup{)}suspension} of $S$. The right adjoint sends a bi-pointed $\Theta_2$-set $(x,y) \colon \partial\Theta_2[1;0] \lra X$ to the simplicial set $\Hom_X(x,y)$, whose set of $m$-simplices is given by the following pullback (cf.\ \cite[\S4.11]{MR2578310}).
\begin{equation*}
\cd{
\Hom_X(x,y)_m \ar[r] \ar[d] \fatpullbackcorner & X_{1(m)} \ar[d] \\
1 \ar[r]_-{(x,y)} & X_0 \times X_0
}
\end{equation*}
If $X$ is a $2$-quasi-category, then Proposition \ref{susphomqadjthefirst} below  implies that the simplicial sets $\Hom_X(x,y)$ are quasi-categories; we will call them the \emph{hom-quasi-categories} of $X$.

\begin{example}[the hom-quasi-categories of the coherent nerve of a bicategory] \label{nbichoms}
Let $A$ be a bicategory. For each pair of objects $a,b \in A$, the hom-quasi-category $\Hom_{NA}(a,b)$ of the coherent nerve of $A$ is isomorphic to the nerve of the hom-category $A(a,b)$ of $A$.
\end{example}

\begin{proposition} \label{susphomqadjthefirst}
The suspension-hom adjunction $\Sigma \dashv \Hom$ \textup{(\ref{susphomadjlastmin})} 
 is a Quillen adjunction between the model structure on the category of bi-pointed $\Theta_2$-sets induced from Ara's model structure for $2$-quasi-categories, and Joyal's model structure for quasi-categories.
\end{proposition}
\begin{proof}
For each $m \geq 0$, the suspension functor sends the boundary inclusion $\partial\Delta[m] \lra \Delta[m]$ to the boundary inclusion $\partial\Theta_2[1;m] \lra \Theta_2[1;m]$. (This can be proved by expressing these objects as colimits of their non-degenerate elements; see for instance \cite[Lemme 8.2.22]{MR2294028}.) Hence the suspension functor preserves cofibrations. 

Furthermore, the suspension functor sends, for each $m \geq 2$, the spine inclusion $I[m] \lra \Delta[m]$ to the spine inclusion $I[1;m] \lra \Theta_2[1;m]$, and sends the simplicial nerve of the functor $\mathbb{I} \lra \mathbf{1}$ to the morphism of $\Theta_2$-sets $J_2 \lra \Theta_2[1;0]$.  Since these morphisms are weak equivalences in the model structure for $2$-quasi-categories, the analogue of Proposition \ref{welemma} for Joyal's model structure for quasi-categories (due in this case to Joyal and Tierney \cite{MR2342834}) implies that the suspension functor also preserves weak equivalences.
\end{proof}

Thus every $2$-quasi-category $X$ has an underlying quasi-category-enriched graph, consisting of the set $X_0$ of objects of $X$, and the hom-quasi-categories $\Hom_X(x,y)$. While this quasi-category-enriched graph does not extend to a quasi-category-enriched category in general, we will show in Theorem \ref{segalthm} that it does extend to a quasi-category-enriched \emph{Segal category}.

\begin{notation}[bisimplicial sets]
Let  $[(\Delta\times\Delta)^\mathrm{op},\Set]$ denote the category of bisimplicial sets. We will frequently identify a bisimplicial set $X$ with  the functor $\Delta^\mathrm{op} \lra [\Delta^\mathrm{op},\Set]$ that sends each $[n] \in \Delta$ to the \emph{$n$th column} $X_n$ of $X$ (i.e.\ the simplicial set $[m] \mapsto X_{n,m}$). Thus, for $S$ a simplicial set, we will write $S\pitchfork X$ for the 
end $\int_{[n] \in \Delta} S_n \pitchfork X_n$ in $[\Delta^\mathrm{op},\Set]$.
\end{notation}

\subsection{The underlying bisimplicial set of a $\Theta_2$-set} \label{undbisimp}
Let $d \colon \Delta\times\Delta \lra \Theta_2$ denote the functor given by $d([n],[m]) =  [n;m,\ldots,m]$. Restriction and right Kan extension along (the opposite of) this functor defines an adjunction
\begin{equation}  \label{dadj}
\xymatrix{
[(\Delta\times\Delta)^\mathrm{op},\Set] \ar@<-1.5ex>[rr]^-{\hdash}_-{d_*} && \ar@<-1.5ex>[ll]_-{d^*} [\Theta_2^\mathrm{op},\Set]
}
\end{equation}
between the category of bisimplicial sets and the category of $\Theta_2$-sets. We say that the left adjoint of this adjunction sends a $\Theta_2$-set $X$
 to its \emph{underlying bisimplicial set} $d^*(X)$. Note that the $0$th and $1$st columns of $d^*(X)$ are the discrete set $X_0$ and the coproduct $\sum_{x,y \in X_0} \Hom_X(x,y)$ respectively.

\begin{definition}[Joyal- and quasi-category-enriched Segal categories] \label{joysegdef}
A bisimplicial set $X$ is said to be a \emph{Joyal-enriched Segal category} if:
\begin{enumerate}[(i)]
\item the simplicial set $X_0$ is discrete, and
\item for each $n \geq 2$, the ``Segal map'' 
\begin{equation*}
X_n \cong \Delta[n] \pitchfork X \lra I[n]\pitchfork X \cong X_1 \times_{X_0} \cdots \times_{X_0} X_1
\end{equation*} is a weak categorical equivalence (i.e.\ a weak equivalence in Joyal's model structure for quasi-categories). 
\end{enumerate}
A Joyal-enriched Segal category $X$ is said to be a  \emph{quasi-category-enriched Segal category} if, for each $n \geq 1$, the simplicial set $X_n$ is a quasi-category.
\end{definition}

\begin{observation}[precategories]
Let $X$ be a bisimplicial set whose $0$th column $X_0$ is discrete, i.e.\ $X$ is a \emph{precategory}.
Then, for each $n \geq 1$, the simplicial set $X_n$ admits the coproduct decomposition $$X_n \cong \sum_{x_0,\ldots,x_n \in X_0} X(x_0,\ldots,x_n),$$ where $X(x_0,\ldots,x_n)$ denotes the fibre over $(x_0,\ldots,x_n)$ of the morphism $X_n \lra X_0^{n+1}$ induced by the inclusion $\{0,\ldots,n\} \lra \Delta[n]$. Hence  a precategory $X$ is a Joyal-enriched Segal category if and only if, 
for each $n \geq 2$ and $x_0,\ldots,x_n \in X_0$, the Segal map $$X(x_0,\ldots,x_n) \lra X(x_0,x_1) \times \cdots X(x_{n-1},x_n)$$ is a weak categorical equivalence; such an $X$ is moreover a quasi-category-enriched Segal category if and only if, for each $n \geq 1$ and $x_0,\ldots,x_n \in X_0$, the simplicial set $X(x_0,\ldots,x_n)$ is a quasi-category.
\end{observation}

Our proof that the underlying bisimplicial set $d^*(X)$ of a $2$-quasi-category $X$ is a quasi-category-enriched Segal category will depend on the following class of trivial cofibrations in the model structure for $2$-quasi-categories. 

\subsection{Horizontal spine inclusions} \label{horspsubsect}
For each object $[n;\bm{m}] = [n;m_1,\ldots,m_n] \in \Theta_2$, we define the \emph{horizontal spine inclusion} $Sp[n;\bm{m}] \lra \Theta_2[n;\bm{m}]$ to be the pullback (cf.\ \cite[\S4.10]{MR2578310})
\begin{equation*}
\cd{
Sp[n;\bm{m}] \ar[r] \ar[d] \fatpullbackcorner & I[n;0,\ldots,0] \ar[d] \\
\Theta_2[n;\bm{m}] \ar[r] & \Theta_2[n;0,\ldots,0]
}
\end{equation*}
 of the spine inclusion $I[n;0,\ldots,0] \lra \Theta_2[n;0,\ldots,0]$ along the unique bijective-on-objects morphism $\Theta_2[n;\bm{m}] \lra \Theta_2[n;0,\ldots,0]$. The horizontal spine $Sp[n;\bm{m}]$ is thus the colimit of the diagram (cf.\ \cite[\S5.1]{MR2578310})
\begin{equation*}
\cd{
 \Theta_2[1;m_1] & \Theta_2[0] \ar[l]_-{\delta^1} \ar[r]^-{\delta^0} & \cdots & \Theta_2[0] \ar[l]_-{\delta^1} \ar[r]^-{\delta^0} & \Theta_2[1;m_n]
}
\end{equation*}
in $[\Theta_2^\mathrm{op},\Set]$.

\begin{lemma} \label{horsplemma} For each $[n;\bm{m}] \in \Theta_2$, the horizontal spine inclusion $Sp[n;\bm{m}] \lra \Theta_2[n;\bm{m}]$ is a weak equivalence in the model structure for $2$-quasi-categories.
\end{lemma}
\begin{proof}
Since  the spine inclusion $I[n;\bm{m}] \lra \Theta_2[n;\bm{m}]$ is a weak equivalence, it suffices by the two-out-of-three property to prove that the  inclusion $I[n;\bm{m}] \lra Sp[n;\bm{m}]$ of the (ordinary) spine into the horizontal spine is a weak equivalence. But this latter inclusion is the colimit of the morphism of diagrams displayed below,
\begin{equation*}
\cd{
 I[1;m_1] \ar[d] & \Theta_2[0] \ar@{=}[d] \ar[l] \ar[r] & \cdots & \Theta_2[0] \ar@{=}[d] \ar[l] \ar[r] & I[1;m_n] \ar[d] \\
  \Theta_2[1;m_1] & \Theta_2[0] \ar[l]^-{\delta^1} \ar[r]_-{\delta^0} & \cdots & \Theta_2[0] \ar[l]^-{\delta^1} \ar[r]_-{\delta^0} & \Theta_2[1;m_n]
}
\end{equation*}
and is therefore a weak equivalence by the gluing lemma, since each vertical (resp.\ horizontal) morphism displayed above is a weak equivalence (resp.\ monomorphism). 
\end{proof}

\begin{remark}
A generalisation of Lemma \ref{horsplemma} is implicit in \cite[\S8.1]{MR3350089}, where it is stated that Ara's definition of the model structure for  Rezk $\Theta_n$-spaces agrees with Rezk's original definition.
\end{remark}

\begin{theorem} \label{segalthm}
The underlying bisimplicial set of a $2$-quasi-category is a quasi-category-enriched Segal category.
\end{theorem}
\begin{proof}
Let $X$ be a $2$-quasi-category. By construction, the $0$th column of its underlying bisimplicial set $d^*(X)$ is discrete, and its $1$st column is the coproduct $\sum_{x,y \in X_0} \Hom_X(x,y)$, which is a quasi-category  by Proposition \ref{susphomqadjthefirst}. Hence it suffices to show that, for each $n \geq 2$, the Segal map 
\begin{equation} \label{segalmap}
d^*(X)_n \lra d^*(X)_1 \times_{X_0} \cdots \times_{X_0} d^*(X)_1
\end{equation}
 is a trivial fibration in $[\Delta^\mathrm{op},\Set]$.

To this end, let
\begin{equation*}
\Hom : [\Delta,[\Theta_2^\mathrm{op},\Set]]^\mathrm{op} \times [\Theta_2^\mathrm{op},\Set] \lra [\Delta^\mathrm{op},\Set]
\end{equation*}
denote the bifunctor that sends a pair $(E,U)$, consisting of a cosimplicial $\Theta_2$-set $E$ and a $\Theta_2$-set $U$, to the simplicial set $\Hom(E,U)$ given by $\Hom(E,U)_m = [\Theta_2^\mathrm{op},\Set](E^m,U)$. Observe that, for each $n \geq 2$, the Segal map (\ref{segalmap}) can be expressed in terms of this hom bifunctor as 
\begin{equation} \label{cosimpprecomp}
\Hom(i,X) : \Hom(T,X) \lra \Hom(S,X),
\end{equation} 
where $i \colon S \lra T$ denotes the morphism of cosimplicial objects in $[\Theta_2^\mathrm{op},\Set]$ whose component at $[m] \in \Delta$ is the horizontal spine inclusion $Sp[n;m\ldots,m] \lra \Theta_2[n;m,\ldots,m]$. 

Therefore, to show that the Segal map (\ref{segalmap}) is a trivial fibration, it suffices by a standard result of Reedy category theory (being, for instance, a dual of  \cite[Proposition 9.1]{MR3217884}), to show that the morphism $i \colon S \lra T$ is a Reedy trivial cofibration of cosimplicial $\Theta_2$-sets, with respect to the model structure for $2$-quasi-categories (since $X$ is fibrant by assumption).  Since each horizontal spine inclusion is a weak equivalence by Lemma \ref{horsplemma}, it remains (by, for instance, \cite[Lemma 7.1]{MR3217884}) to prove that $i \colon S \lra T$ is a Reedy cofibration. 

To do so, it suffices (cf.\ \cite[Theorem 15.9.9]{MR1944041}) to show that each component of $i \colon S \lra T$ is a monomorphism, which is evident, 
and to show that the induced morphism between the maximal augmentations of $S$ and $T$, i.e.\ the induced morphism between the equalisers displayed below, is an isomorphism.
 \begin{equation*}
 \cd{
 \{0,\ldots,n\} \ar@{..>}[d] \ar[r] & Sp[n;0,\ldots,0] \ar@<+1ex>[r]^{\delta^1} \ar@<-1ex>[r]_-{\delta^0} \ar[d] & Sp[n;1,\ldots,1] \ar[d] \\
  \{0,\ldots,n\} \ar[r] & \Theta_2[n;0,\ldots,0]  \ar@<+1ex>[r]^{\delta^1} \ar@<-1ex>[r]_-{\delta^0}  & \Theta_2[n;1,\ldots,1]
}
\end{equation*}
But both of these equalisers are easily seen to be the discrete $\Theta_2$-set $\{0,\ldots,n\}$, and the induced morphism between them to be the identity. This completes the proof. \end{proof}

\begin{remark}
Beyond the results of the present section, Theorem \ref{segalthm} will also play a crucial role in our proof of the Quillen equivalence between $2$-quasi-categories and Joyal-enriched Segal categories in Theorem \ref{qe10}.
\end{remark}

We have thus completed the first step of our construction of the homotopy bicategory of a $2$-quasi-category. Next, we associate to each Joyal-enriched Segal category a  Tamsamani $2$-category (i.e.\ a $\Cat$-enriched Segal category).

\subsection{Tamsamani $2$-categories}
A simplicial category $C \colon \Delta^\mathrm{op} \lra \Cat$ is said to be a \emph{Tamsamani $2$-category} \cite{MR1673923,MR2366560} if:
\begin{enumerate}[(i)]
\item the category $C_0$ is discrete, and
\item for each $n \geq 2$, the Segal map
\begin{equation*}
C_n \cong \Delta[n] \pitchfork C \lra I[n]\pitchfork C \cong C_1 \times_{C_0} \cdots \times_{C_0} C_1
\end{equation*} is an equivalence of categories. 
\end{enumerate}
Let $\mathbf{Tam}$ denote the full subcategory of $[\Delta^\mathrm{op},\Cat]$ consisting of the Tamsamani $2$-categories.

\subsection{The homotopy Tamsamani $2$-category of a $2$-quasi-category} \label{changeofbase}
Recall that the simplicial nerve functor $N \colon \Cat \lra [\Delta^\mathrm{op},\Set]$ has a left adjoint $\ho \colon [\Delta^\mathrm{op},\Set] \lra \Cat$, which sends a simplicial set to its homotopy category. Post-composition by these functors defines an adjunction
\begin{equation}  \label{cobadj}
\xymatrix{
[\Delta^\mathrm{op},\Cat] \ar@<-1.5ex>[rr]^-{\hdash}_-{[1,N]} && \ar@<-1.5ex>[ll]_-{[1,\ho]} [(\Delta\times\Delta)^\mathrm{op},\Set]
}
\end{equation}
between the categories of simplicial categories and bisimplicial sets 
(where we have identified the latter with the functor category $[\Delta^\mathrm{op},[\Delta^\mathrm{op},\Set]]$). 
Since the functor $\ho \colon [\Delta^\mathrm{op},\Set] \lra \Cat$ preserves pullbacks over discrete objects, and sends weak categorical equivalences to equivalences of categories, the left adjoint of this adjunction sends every Joyal-enriched Segal category $X$ to a Tamsamani $2$-category
\begin{equation*}
\cd[@C=2.5em]{
\Delta^\mathrm{op} \ar[r]^-{X} & [\Delta^\mathrm{op},\Set] \ar[r]^-{\ho} & \Cat,
}
\end{equation*}
which we denote by $\ho(X)$ (and which could be called the \emph{change of base} of $X$ along the functor $\ho \colon [\Delta^\mathrm{op},\Set] \lra \Cat$). In particular, Theorem \ref{segalthm} has the following corollary.

\begin{corollary} \label{tamcor}
Let $X$ be a $2$-quasi-category. Then the simplicial category $\ho(d^*(X))$ is a Tamsamani $2$-category.
\end{corollary}
\begin{proof}
By Theorem \ref{segalthm}, the bisimplicial set $d^*(X)$ is a Joyal-enriched Segal category, and hence its change of base along $\ho \colon [\Delta^\mathrm{op},\Set] \lra \Cat$ is a Tamsamni $2$-category, as observed in \S\ref{changeofbase}. 
\end{proof}

For each $2$-quasi-category $X$, we call the simplicial category $\ho(d^*(X))$ the \emph{homotopy Tamsamani $2$-category} of $X$. By composing the adjunctions (\ref{cobadj}) and (\ref{dadj}), we obtain an adjunction 
\begin{equation}  \label{compadjho}
\xymatrix{
[\Delta^\mathrm{op},\Cat] \ar@<-1.5ex>[rr]^-{\hdash}_-{[1,N]} && \ar@<-1.5ex>[ll]_-{[1,\ho]} [(\Delta\times\Delta)^\mathrm{op},\Set]  \ar@<-1.5ex>[rr]^-{\hdash}_-{d_*} && \ar@<-1.5ex>[ll]_-{d^*} [\Theta_2^\mathrm{op},\Set]
}
\end{equation}
between the categories of simplicial categories and $\Theta_2$-sets. Corollary \ref{tamcor} implies that the left adjoint of this composite adjunction restricts to a functor
\begin{equation}\label{hotamfun}
\ho(d^*(-)) \colon \twoqcat \lra \mathbf{Tam}
\end{equation}
from the category of $2$-quasi-categories to the category of Tamsamani $2$-categories.

To complete our construction of the homotopy bicategory of a $2$-quasi-category, it remains to recall the construction of the ``bicategory reflection'' of a Tamsamani $2$-category.

\subsection{The bicategory reflection of a Tamsamani $2$-category} \label{bicatreflsec}
In \cite[Th\'eor\`eme 4.2]{MR1673923}, Tamsamani describes a construction which assigns to each Tamsamani $2$-category $C$ a bicategory $GC$, which we will call the \emph{bicategory reflection} of $C$. The objects of the bicategory $GC$ are the elements of the set $C_0$, and the hom-categories of $GC$ are the fibres of the functor $(d_1,d_0) \colon C_1 \lra C_0 \times C_0$.

This construction is reviewed by Lack and Paoli in \cite[\S7]{MR2366560}, where it is shown to define a functor $G \colon \mathbf{Tam} \lra \Bicat$. Moreover, Lack and Paoli prove a universal property of this construction, namely that the functor $G \colon \mathbf{Tam} \lra \Bicat$ is a partial left adjoint to the ``$2$-nerve'' functor $N_2 \colon \Bicat \lra [\Delta^\mathrm{op},\Cat]$. We refer the reader to \cite[\S3]{MR2366560} for the definition of the $2$-nerve functor (as the $\Cat$-enriched singular functor induced by the full inclusion of $\Delta$ into the $2$-category $\Bicat_2$ of bicategories, normal pseudofunctors, and icons). Suffice it to recall that the $2$-nerve functor is fully faithful \cite[Theorem 3.7]{MR2366560}, and that the $2$-nerve of a rigid $2$-category is (isomorphic to) its ``standard nerve'', in the sense of the following paragraph.

\begin{recall}[standard nerve for $2$-categories] \label{stnrecall}
Let $N_{st} \colon \twocat \lra [\Delta^\mathrm{op},\Cat]$ denote the (fully faithful) functor that sends a $2$-category $A$ to its \emph{standard nerve} $N_{st}A$, i.e.\ the simplicial category  whose category of $0$-simplices is the discrete set $A_0$ of objects of $A$, and whose category of $n$-simplices is the following coproduct of products of hom-categories.
\begin{equation*}
(N_{st}A)_n = \sum_{a_0,\ldots,a_n \in A_0} A(a_0,a_1) \times \cdots \times A(a_{n-1},a_n)
\end{equation*}
An easy calculation shows that the standard nerve functor $N_{st} \colon \twocat \lra [\Delta^\mathrm{op},\Cat]$ is naturally isomorphic to the composite functor
\begin{equation*}
\cd[@C=3em]{
\twocat \ar[r]^-{N_s} & [\Theta_2^\mathrm{op},\Set] \ar[r]^-{d^*} & [(\Delta\times\Delta)^\mathrm{op},\Set] \ar[r]^-{[1,\ho]} & [\Delta^\mathrm{op},\Cat],
}
\end{equation*}
where $N_s$ denotes the strict nerve functor (see \S\ref{strnobs}).
\end{recall}

We may now complete our construction of the homotopy bicategory of a $2$-quasi-category. 

\subsection{The homotopy bicategory of a $2$-quasi-category} \label{hobicatsubsec}
We define the \emph{homotopy bicategory} $\Ho(X)$ of a $2$-quasi-category $X$ to be the bicategory reflection (\S\ref{bicatreflsec}) of the homotopy Tamsamani $2$-category (Corollary \ref{tamcor}) of $X$, i.e.\ $\Ho(X) = G(\ho(d^*(X)))$. By construction, the objects of $\Ho(X)$ are the objects of $X$, and the hom-categories of $\Ho(X)$ are the homotopy categories of the hom-quasi-categories of $X$, i.e.\ $(\Ho X)(x,y) = \ho (\Hom_X(x,y))$. 
Furthermore, we define the \emph{homotopy bicategory functor} 
\begin{equation*}
\Ho \colon \twoqcat \lra \Bicat
\end{equation*}
 to be the composite 
 \begin{equation*}
 \cd{
 \twoqcat \ar[rr]^-{\ho(d^*(-))} && \mathbf{Tam} \ar[r]^-G & \Bicat
 }
 \end{equation*}
 of the homotopy Tamsamani $2$-category functor (\ref{hotamfun}) and the bicategory reflection functor (see \S\ref{bicatreflsec}).
 
 \medskip
 
We will use the following lemma in \S\ref{triequivsect}.
 
 \begin{lemma} \label{finprodlemma}
The homotopy bicategory functor $\Ho \colon \twoqcat \lra \Bicat$ preserves finite products.
\end{lemma}
\begin{proof}
First, observe that the homotopy bicategory of the terminal $\Theta_2$-set $\Theta_2[0]$ is the terminal bicategory $\mathbf{1}$.
Now, let $X$ and $Y$ be $2$-quasi-categories. By construction, the canonical normal pseudofunctor $$\Ho(X\times Y) \lra \Ho(X) \times \Ho(Y)$$ is bijective on objects, and is given on hom-categories by the canonical functor $$\ho(X(x,x')\times Y(y,y')) \lra \ho(X(x,x')) \times \ho(Y(y,y')),$$ which is an isomorphism since the homotopy category functor $\ho \colon [\Delta^\mathrm{op},\Set] \lra \Cat$ preserves finite products. Hence the result follows from the fact that a normal pseudofunctor is an isomorphism in the category $\Bicat$ if and only if it is bijective on objects and an isomorphism on hom-categories. 
\end{proof}
 
 We conclude this section by proving that the homotopy bicategory functor $\Ho \colon \twoqcat \lra \Bicat$ is left adjoint to the coherent nerve functor $N \colon \Bicat \lra \twoqcat$. We will deduce this from the universal property of the bicategory reflection construction (see \S\ref{bicatreflsec}), by way of the following factorisation of the coherent nerve functor through the $2$-nerve functor.

\begin{proposition} \label{2nerves}
The coherent nerve functor $N \colon \Bicat \lra [\Theta_2^\mathrm{op},\Set]$ is naturally isomorphic to the composite functor
\begin{equation*}
\cd{
\Bicat \ar[r]^-{N_2} & [\Delta^\mathrm{op},\Cat] \ar[r]^-{[1,N]} & [(\Delta\times\Delta)^\mathrm{op},\Set] \ar[r]^-{d_*} & [\Theta_2^\mathrm{op},\Set].
}
\end{equation*}
\end{proposition}
\begin{proof}
By the Yoneda lemma, the result follows from the following chain of natural isomorphisms, 
\begin{align*}
[\Theta_2^\mathrm{op},\Set](N[n;\bm{m}], d_*(N(N_2A))) 
&\cong [\Delta^\mathrm{op},\Cat](\ho(d^*(N[n;\bm{m}])),N_2A)  \\
&\cong [\Delta^\mathrm{op},\Cat](N_{st}[n;\bm{m}],N_2A)  \\
&\cong [\Delta^\mathrm{op},\Cat](N_2[n;\bm{m}],N_2A)  \\
&\cong \Bicat([n;\bm{m}],A)
\end{align*}
where we have used the composite adjunction (\ref{compadjho}), the factorisation of the standard nerve functor from Recollection \ref{stnrecall}, the fact that the $2$-categories $[n;\bm{m}]$ are rigid, and the full fidelity of the $2$-nerve functor \cite[Theorem 3.7]{MR2366560}.
\end{proof}

\begin{theorem} \label{honadjthm}
The homotopy bicategory  and coherent nerve functors are the left and right adjoints of an adjunction
\begin{equation*}
\xymatrix{
\Bicat \ar@<-1.5ex>[rr]^-{\hdash}_-N && \ar@<-1.5ex>[ll]_-{\Ho} \twoqcat
}
\end{equation*}
between the category of bicategories and normal pseudofunctors and the category of $2$-quasi-categories. 
\end{theorem}
\begin{proof}
Let $\eta \colon \mathrm{id} \lra \Ho \circ N \colon \twoqcat \lra \twoqcat$ denote the natural transformation whose component at a $2$-quasi-category $X$ is the composite
\begin{equation*}
\cd{
X \ar[r]^-{\eta_1} & d_*(N(\ho(d^*(X)))) \ar[rr]^-{d_*(N(\eta_2))} && d_*(N(N_2(\Ho X))) \cong N(\Ho X),
}
\end{equation*}
where $\eta_1$ denotes the unit of the composite adjunction (\ref{compadjho}) and $\eta_2$ denotes the unit of the partial adjunction $G\dashv N_2$, and where we have used  the natural isomorphism of Proposition \ref{2nerves}. It then follows from the universal properties of the units $\eta_1$ and $\eta_2$ that the natural transformation $\eta$ defines the unit of an adjunction $\Ho \dashv N$, as desired. 
\end{proof}

\section{Equivalences of $2$-quasi-categories} \label{fundythmsec}
The goal of this section is to prove (Theorem \ref{2truncchar}) that a $2$-quasi-category is equivalent to the coherent nerve of a bicategory if and only if it is $2$-truncated (Definition \ref{2truncdef}).  To this end, we prove (Theorem \ref{fundythm}) that a morphism of $2$-quasi-categories is an equivalence (see \S\ref{equivsubsect}) if and only if it is essentially surjective on objects (Definition \ref{esodef})  and fully faithful (Definition \ref{ffdef}). 

\subsection{Equivalences of quasi-categories} \label{1dimfundy}
Recall that a morphism in a quasi-category $X$ is said to be an \emph{isomorphism} if its homotopy class is an isomorphism in the homotopy category $\ho(X)$. A morphism of quasi-categories $f \colon X \lra Y$ is said to be \emph{essentially surjective on objects} if, for each object $y \in Y$, there exists an object $x \in X$ and an isomorphism $f(x) \cong y$ in $Y$.

One model for the hom-spaces of a quasi-category is provided by the Quillen adjunction (see \cite[Proposition 4.5]{MR2764043})
\begin{equation*} \label{susphomadj}
 \xymatrix{
\partial\Delta[1]/[\Delta^\mathrm{op},\Set] \ar@<-1.5ex>[rr]^-{\hdash}_-{\Hom} && \ar@<-1.5ex>[ll]_-{\Sigma_1} [\Delta^\mathrm{op},\Set]
}
\end{equation*}
between the category of bi-pointed simplicial sets endowed with the model structure induced from the model structure for quasi-categories, and the category of simplicial sets endowed with the model structure for Kan complexes. The left adjoint of this adjunction sends a simplicial set $S$ to the bi-pointed simplicial set $\Sigma_1(S)$ defined by the pushout on the left below,
\begin{equation*} \label{susppo}
 \cd{
\partial\Delta[1] \times S  \ar[r]^-{\mathrm{pr}_1} \ar[d] & \partial\Delta[1] \ar[d] \\ 
\Delta[1] \times S  \ar[r] & \fatpushoutcorner \Sigma_1(S)
 }
 \qquad
 \qquad
  \cd{
 \Hom_X(x,y) \ar[r] \ar[d] \fatpullbackcorner & X^{\Delta[1]} \ar[d] \\
 \Delta[0] \ar[r]_-{(x,y)} & X^{\partial\Delta[1]}
 }
 \end{equation*}
and the right adjoint sends a bi-pointed simplicial set $(x,y) \colon \partial\Delta[1] \lra X$ to the simplicial set $\Hom_X(x,y)$ defined by the pullback on the right above. A morphism of quasi-categories $f \colon X \lra Y$ is said to be \emph{fully faithful} if, for each pair of objects $x,y \in X$, the induced morphism $f \colon \Hom_X(x,y) \lra \Hom_Y(fx,fy)$ is an equivalence of Kan complexes.

A fundamental theorem of quasi-category theory states that a morphism of quasi-categories is an equivalence (i.e.\ a weak equivalence in the model structure for quasi-categories) if and only if it is essentially surjective on objects and fully faithful. The goal of this section is to prove the corresponding theorem for $2$-quasi-categories.

\subsection{Equivalences of $2$-quasi-categories} \label{equivsubsect}
A morphism of $2$-quasi-categories $f \colon X \lra Y$ is said to be an \emph{equivalence \textup{(}of $2$-quasi-categories\textup{)}} if it is a weak equivalence in Ara's model structure for $2$-quasi-categories. 

\begin{observation}
Since the object $J$ (see Notation \ref{freeliving}) is an interval object in this model structure, a morphism of $2$-quasi-categories $f \colon X \lra Y$ is an equivalence if and only if there exists a morphism $g \colon Y \lra X$ and morphisms $h \colon J \times X \lra X$ and $k \colon J \times Y \lra Y$ such that $h\circ(\{0\}\times X) = 1_X$, $h\circ(\{1\}\times X) = gf$, $k\circ(\{0\}\times Y) = 1_Y$, and $k\circ(\{1\}\times X) = fg$.
\end{observation}

\subsection{The underlying quasi-category of a $2$-quasi-category} 
We will call the underlying simplicial set  $\tau^*(X)$ of a $2$-quasi-category $X$ its \emph{underlying quasi-category}. This name is justified by the following proposition.

\begin{proposition} \label{undqadj}
The adjunction \textup{(}see \textup{Observation} \textup{\ref{afullemb}}\textup{)}
\begin{equation*}
 \xymatrix{
[\Theta_2^\mathrm{op},\Set] \ar@<-1.5ex>[rr]^-{\hdash}_-{\tau^*} && \ar@<-1.5ex>[ll]_-{\pi^*} [\Delta^\mathrm{op},\Set]
}
\end{equation*}
is a Quillen adjunction between Ara's model structure for $2$-quasi-categories and Joyal's model structure for quasi-categories. Hence the underlying simplicial set of a $2$-quasi-category is a quasi-category.
\end{proposition}
\begin{proof}
The left adjoint $\pi^*$ preserves monomorphisms, since it is a restriction functor. Moreover, for each $n \geq 2$, $\pi^*$ sends the spine inclusion $I[n] \lra \Delta[n]$ to the spine inclusion $I[n;0,\ldots,0] \lra \Theta_2[n;0,\ldots,0]$, which is a weak equivalence in the model structure for $2$-quasi-categories by the definition of the latter. Also,  $\pi^*$ sends the simplicial nerve of the functor $\mathbb{I} \lra \mathbf{1}$ to the morphism  of $\Theta_2$-sets $J \lra \Theta_2[0]$, which is a trivial fibration, and hence a weak equivalence. As in the proof of Proposition \ref{susphomqadjthefirst}, this proves the result.
\end{proof}

\begin{example}
The underlying quasi-category of the coherent nerve of a bicategory $A$ is the Roberts--Street--Duskin nerve \cite{MR1897816} of the locally groupoidal core of $A$ (which consists of the objects, morphisms, and invertible $2$-cells of $A$).
\end{example}

\subsection{Isomorphisms in $2$-quasi-categories}
A morphism (i.e.\ a $[1;0]$-element) in a $2$-quasi-category $X$ is said to be an \emph{isomorphism} if it satisfies the equivalent conditions of the following proposition.

\begin{proposition} \label{isoprop}
A morphism in a $2$-quasi-category $X$ is an isomorphism in the underlying quasi-category $\tau^*(X)$ if and only if it is an equivalence in the homotopy bicategory $\Ho(X)$. 
\end{proposition}
\begin{proof}
Unwinding the definitions, it suffices to show that a parallel pair of morphisms $f,g \colon x \lra y$ in a $2$-quasi-category $X$ are isomorphic as objects of the hom Kan complex $\Hom_{\tau^*(X)}(x,y)$ if and only if they are isomorphic as objects of the hom-quasi-category $\Hom_X(x,y)$. The former (resp.\ latter) is equivalent to the existence of an extension as in the diagram on the left (resp.\ right) below (note that $\partial\Theta_2[1;1] \cong \pi^*(\Sigma_1(\partial\Delta[1]))$).
\begin{equation*}
\cd{
\pi^*(\Sigma_1(\partial\Delta[1])) \ar[r]^-{(f,g)} \ar[d] & X \\
\pi^*(\Sigma_1(\Delta[1])) \ar@{..>}[ur]
}
\qquad
\qquad
\cd{
\partial\Theta_2[1;1] \ar[r]^-{(f,g)} \ar[d] & X \\
J_2 \ar@{..>}[ur]
}
\end{equation*}
Hence the result follows from the observation that both $\pi^*(\Sigma_1(\Delta[1]))$ and $J_2$ define relative cylinder objects for the boundary inclusion $\partial\Theta_2[1;0] \lra \Theta_2[1;0]$ in the model structure for $2$-quasi-categories, via the evident  factorisations
\begin{equation*}
\pi^*(\Sigma_1(\partial\Delta[1])) \lra \pi^*(\Sigma_1(\Delta[1])) \lra \pi^*(\Sigma_1(\Delta[0]))
\end{equation*}
\begin{equation*}
\partial\Theta_2[1;1] \lra J_2 \lra \Theta_2[1;0]
\end{equation*}
of the relative codiagonal of the boundary inclusion $\partial\Theta_2[1;0] \lra \Theta_2[1;0]$
into a cofibration followed by a weak equivalence.
\end{proof}

\begin{definition}[essentially surjective on objects] \label{esodef}
A morphism of $2$-quasi-categories $f \colon X \lra Y$ is said to be \emph{essentially surjective on objects} if, for each object $y \in Y$, there exists an object $x \in X$ and an isomorphism $f(x) \cong y$ in $Y$.
\end{definition}

\begin{observation} \label{esoobs}
Let $f \colon X \lra Y$ be a morphism of $2$-quasi-categories. It follows from Proposition \ref{isoprop} that the following are equivalent:
\begin{enumerate}[(i)]
\item the morphism of $2$-quasi-categories $f \colon X \lra Y$ is essentially surjective on objects;
\item the morphism of quasi-categories $\tau^*(f) \colon \tau^*(X) \lra \tau^*(Y)$ is essentially surjective on objects;
\item the normal pseudofunctor $\Ho(f) \colon \Ho(X) \lra \Ho(Y)$ is biessentially surjective on objects.
\end{enumerate}
\end{observation}

\begin{definition}[(locally) fully faithful] \label{ffdef}
A morphism of $2$-quasi-categories $f \colon X \lra Y$ is said to be \emph{fully faithful} (resp.\ \emph{locally fully faithful}) if, for each pair of objects $x,y \in X$, the induced morphism of quasi-categories $f \colon \Hom_X(x,y) \lra \Hom_Y(fx,fy)$ is an equivalence (resp.\ fully faithful).
\end{definition}

\begin{lemma} \label{hlemma}
Let $F_i \dashv G_i \colon \mathcal{M} \lra [\Delta^\mathrm{op},\Set]$, for $i=1,2$, be Quillen adjunctions between a model category $\mathcal{M}$ and the category of simplicial sets endowed with the model structure for Kan complexes. Suppose that the objects $F_1(\Delta[0])$ and $F_2(\Delta[0])$ are weakly equivalent in $\mathcal{M}$. 
If $f \colon X \lra Y$ is a morphism between fibrant objects in $\mathcal{M}$, then $G_1(f)$ is an equivalence of Kan complexes if and only if $G_2(f)$ is.
\end{lemma}
\begin{proof}
It follows from \cite[Lemma 9.7]{MR1870515} that the left derived functors of the left Quillen functors $F_1$ and $F_2$ are naturally isomorphic, and hence that the right derived functors of the right Quillen functors $G_1$ and $G_2$ are naturally isomorphic. The conclusion immediately follows.
\end{proof}

\begin{proposition} \label{undequiv}
Let $f \colon X \lra Y$ be a morphism of $2$-quasi-categories. If $f$ is essentially surjective on objects and fully faithful, then its underlying morphism of quasi-categories $\tau^*(f) \colon \tau^*(X) \lra \tau^*(Y)$ is an equivalence.
\end{proposition}
\begin{proof}
The morphism of quasi-categories $\tau^*(f)$ is essentially surjective on objects by Observation \ref{esoobs}. It remains to show that $\tau^*(f)$ is fully faithful, i.e.\ that for each pair of objects $x,y \in X$, the morphism of bi-pointed $2$-quasi-categories $f \colon (X,x,y) \lra (Y,fx,fy)$ is sent by the right adjoint of the composite Quillen adjunction 
\begin{equation*}
 \xymatrix{
\partial\Theta_2[1;0]/[\Theta_2^\mathrm{op},\Set] \ar@<-1.5ex>[rr]^-{\hdash}_-{\tau^*} && \ar@<-1.5ex>[ll]_-{\pi^*} \partial\Delta[1]/[\Delta^\mathrm{op},\Set] \ar@<-1.5ex>[rr]^-{\hdash}_-{\Hom} && \ar@<-1.5ex>[ll]_-{\Sigma_1} [\Delta^\mathrm{op},\Set]
}
\end{equation*}
to an equivalence of Kan complexes. But this morphism 
is sent by the right adjoint of the composite Quillen adjunction 
\begin{equation*}
 \xymatrix{
\partial\Theta_2[1;0]/[\Theta_2^\mathrm{op},\Set] \ar@<-1.5ex>[rr]^-{\hdash}_-{\Hom} && \ar@<-1.5ex>[ll]_-{\Sigma} [\Delta^\mathrm{op},\Set] \ar@<-1.5ex>[rr]^-{\hdash}_-{k^!} && \ar@<-1.5ex>[ll]_-{k_!} [\Delta^\mathrm{op},\Set]
}
\end{equation*}
to an equivalence of Kan complexes, since $f$ is fully faithful by assumption (where $k_! \dashv k^!$ denotes the Quillen adjunction of \cite[Theorem 6.22]{joyalbarcelona}). So the result follows by Lemma \ref{hlemma}, since both composite left adjoints send $\Delta[0]$ to the boundary inclusion $\partial\Theta_2[1;0] \lra \Theta_2[1;0]$.
\end{proof}

In the next few paragraphs, we will use a bifunctorial model for the derived hom-spaces in the model structure for $2$-quasi-categories, which is provided by Cisinski's theory of localisers.

\subsection{Derived hom complexes}
Let $\mathbf{A}$ be a small category, and let $\mathsf{W}$ be an $\mathbf{A}$-localiser. A cosimplicial object $D \colon \Delta \lra [\mathbf{A}^\mathrm{op},\Set]$ is said to be a \emph{cosimplicial $\mathsf{W}$-resolution} \cite[D\'efinition 2.3.12]{MR2294028} if:
\begin{enumerate}
\item the morphism $(d^1,d^0) \colon D^0 + D^0 \lra D^1$ is a monomorphism in $[\mathbf{A}^\mathrm{op},\Set]$, and
\item for every $X \in [\mathbf{A}^\mathrm{op},\Set]$ and every $n \geq 0$, the projection $X \times D^n \lra X$ belongs to $\mathsf{W}$.
\end{enumerate}
Given a cosimplicial $\mathsf{W}$-resolution $D$, we define the \emph{$D$-realisation functor} \cite[\S 2.3.8]{MR2294028} $$Real_D \colon [(\mathbf{A}\times\Delta)^\mathrm{op},\Set] \lra [\mathbf{A}^\mathrm{op},\Set]$$
to be the left Kan extension along the Yoneda embedding of the composite functor
\begin{equation*}
\cd[@C=3em]{
\mathbf{A} \times \Delta \ar[r]^-{Y \times D} & [\mathbf{A}^\mathrm{op},\Set] \times [\mathbf{A}^\mathrm{op},\Set] \ar[r]^-{- \times -} & [\mathbf{A}^\mathrm{op},\Set].
}
\end{equation*}
Furthermore, for each pair of objects $X, Y \in [\mathbf{A}^\mathrm{op},\Set]$,  we define the \emph{$D$-hom complex} $\bHom_D(X,Y)$ to be the simplicial set given by $\bHom_D(X,Y)_n = \Hom(X \times D^n,Y)$. This defines a bifunctor
\begin{equation*} \label{dhom}
\cd{
[\mathbf{A}^\mathrm{op},\Set]^\mathrm{op} \times [\mathbf{A}^\mathrm{op},\Set] \ar[rr]^-{\bHom_D(-,-)} && [\Delta^\mathrm{op},\Set],
}
\end{equation*}
which we call the \emph{$D$-hom complex bifunctor}.

\begin{proposition} \label{quillenbifun}
Let $\mathbf{A}$ be a small category, let  $\mathsf{W}$ be an accessible $\mathbf{A}$-localiser, and let $D$ be a cosimplicial $\mathsf{W}$-resolution.  Then the $D$-hom complex bifunctor $\bHom_D(-,-)$ 
is a right Quillen bifunctor with respect to the Cisinski model structure on $[\mathbf{A}^\mathrm{op},\Set]$ whose class of weak equivalences is $\mathsf{W}$, and the model structure for Kan complexes.
\end{proposition}
\begin{proof}
By adjointness, it is equivalent to show that the composite bifunctor
\begin{equation*}
\cd[@C=3em]{
[\mathbf{A}^\mathrm{op},\Set] \times [\Delta^\mathrm{op},\Set] \ar[r]^-{-\boxtimes -} & [(\mathbf{A}\times\Delta)^\mathrm{op},\Set] \ar[r]^-{Real_D} & [\mathbf{A}^\mathrm{op},\Set]
}
\end{equation*}
is a left Quillen bifunctor  (where $\boxtimes$ denotes the exterior product bifunctor). 

Let $u \colon A \lra B$ be a monomorphism in $[\mathbf{A}^\mathrm{op},\Set]$, and let $v \colon S \lra T$ be a monomorphism in $[\Delta^\mathrm{op},\Set]$. Let $Real_D(u \widehat{\boxtimes} v)$ denote the pushout-corner map of the commutative square
\begin{equation*}
\cd[@C=3em]{
Real_D(A \boxtimes S) \ar[rr]^-{Real_D(A \boxtimes v)} \ar[d]_-{Real_D(u \boxtimes S)} && Real_D(A \boxtimes T) \ar[d]^-{Real_D(u \boxtimes T)} \\
Real_D(B \boxtimes S) \ar[rr]_-{Real_D(B \boxtimes v)} && Real_D(B \boxtimes T)
}
\end{equation*}
in $[\mathbf{A}^\mathrm{op},\Set]$. Then $Real_D(u \widehat{\boxtimes} v)$ is a monomorphism by \cite[Lemme 2.3.2]{MR2294028} and \cite[Lemme 2.3.10]{MR2294028}. If $v$ is a weak homotopy equivalence, then $Real_D(u \widehat{\boxtimes} v)$ belongs to $\mathsf{W}$ by \cite[Proposition 2.3.15]{MR2294028} and the two-out-of-three property. On the other hand, if $u$ belongs to $\mathsf{W}$, then so does $Real_D(u \widehat{\boxtimes} v)$ by \cite[Lemme 2.3.25]{MR2294028} and the two-out-of-three property. This completes the proof.
\end{proof}

\begin{example} \label{kexample}
Let $k \colon \Delta \lra [\Delta^\mathrm{op},\Set]$ denote the functor that sends each $[n]$ to the nerve of the contractible groupoid with set of objects $\{0,\ldots,n\}$. By \cite[Theorem 6.22]{joyalbarcelona}, $k$ is a cosimplicial $\mathsf{W}_1$-resolution, where $\mathsf{W}_1$ denotes the $\Delta$-localiser consisting of the weak equivalences in Joyal's model structure for quasi-categories. 

By \cite[Proposition 8.7]{MR3350089}, the composite $K = \pi^*\circ k \colon \Delta \lra [\Theta_2^\mathrm{op},\Set]$ is a cosimplicial $\mathsf{W}_2$-resolution, where $\mathsf{W}_2$ denotes the $\Theta_2$-localiser consisting of the weak equivalences in Ara's model structure for $2$-quasi-categories. We will denote the $K$-hom complex bifunctor simply by $\bHom(-,-) \colon [\Theta_2^\mathrm{op},\Set]^\mathrm{op} \times [\Theta_2^\mathrm{op},\Set] \lra [\Delta^\mathrm{op},\Set]$. By Proposition \ref{quillenbifun}, this is a right Quillen bifunctor with respect to the model structures for $2$-quasi-categories and Kan complexes.

Note that there is a natural isomorphism 
\begin{equation} \label{enradj}
\bHom(\pi^*(S),X) \cong \bHom_k(S,\tau^*(X)) \end{equation}
for $S \in [\Delta^\mathrm{op},\Set]$ and $X \in [\Theta_2^\mathrm{op},\Set]$.
\end{example}

\subsection{A double suspension-hom adjunction}
For each simplicial set $S$, let $\Sigma^2(S)$ denote the $\Theta_2$-set $\Sigma^2(S)$ defined by the pushout square on the left below (where $K_!$ denotes the left Kan extension along the Yoneda embedding of the functor $K$ from Example \ref{kexample}).
\begin{equation*} \label{doublesushom}
 \cd{
\partial\Theta_2[1;1] \times K_!(S) \ar[r]^-{\mathrm{pr}_1} \ar[d] & \partial\Theta_2[1;1] \ar[d] \\
\Theta_2[1;1] \times K_!(S) \ar[r] & \fatpushoutcorner \Sigma^2(S)
 }
 \qquad
 \qquad
  \cd{
\Hom_X^2(u,v) \ar[r] \ar[d] \fatpullbackcorner & \bHom(\Theta_2[1;1],X) \ar[d] \\
 \Theta_2[0] \ar[r]_-{(u,v)} & \bHom(\partial\Theta_2[1;1],X)
 }
 \end{equation*}
For each parallel pair of $[1;0]$-elements $u,v \colon x\lra y$ in a $\Theta_2$-set $X$, let $\Hom_X^2(u,v)$ denote the simplicial set defined by the pullback square on the right above. These constructions define the left and right adjoints of an adjunction, as displayed below.
\begin{equation} \label{dblsusphomadj}
\xymatrix{
\partial\Theta_2[1;1]/[\Theta_2^\mathrm{op},\Set] \ar@<-1.5ex>[rr]^-{\hdash}_-{\Hom^2} && \ar@<-1.5ex>[ll]_-{\Sigma^2} [\Delta^\mathrm{op},\Set] 
}
\end{equation}

\begin{proposition}
The adjunction $\Sigma^2 \dashv \Hom^2$ \textup{(\ref{dblsusphomadj})} is a Quillen adjunction between the model structure on $\partial\Theta_2[1;1]/[\Theta_2^\mathrm{op},\Set]$ induced from Ara's model structure for $2$-quasi-categories, and the model structure for Kan complexes on $[\Delta^\mathrm{op},\Set]$.
\end{proposition}
\begin{proof}
For any morphism 
\begin{equation*}
\cd[@C=1em]{
& \partial\Theta_2[1;1] \ar[dl]_-{(u,v)} \ar[dr]^-{(fu,fv)} \\
X \ar[rr]_-{f} && Y
}
\end{equation*} in the category $\partial\Theta_2[1;1]/[\Theta_2^\mathrm{op},\Set]$, the induced morphism of simplicial sets $$\Hom^2(f) \colon \Hom_X^2(u,v) \lra \Hom_Y^2(fu,fv)$$ is a pullback of the pullback corner map of the commutative square
\begin{equation*}
\cd{
\bHom(\Theta_2[1;1],X) \ar[rr]^-{\bHom(1,f)} \ar[d] && \bHom(\Theta_2[1;1],Y) \ar[d] \\
\bHom(\partial\Theta_2[1;1],X) \ar[rr]_-{\bHom(1,f)} && \bHom(\partial\Theta_2[1;1],Y)
}
\end{equation*}
in $[\Delta^\mathrm{op},\Set]$.  Hence the result follows from Proposition \ref{quillenbifun}.
\end{proof}

\begin{proposition} \label{hcartprop}
Let $f \colon X \lra Y$ be a morphism of $2$-quasi-categories. Then $f$ is locally fully faithful if and only if the commutative square of Kan complexes
\begin{equation} \label{hcart}
\cd{
\bHom(\Theta_2[1;1],X) \ar[rr]^-{\bHom(1,f)} \ar[d] && \bHom(\Theta_2[1;1],Y) \ar[d] \\
\bHom(\partial\Theta_2[1;1],X) \ar[rr]_-{\bHom(1,f)} && \bHom(\partial\Theta_2[1;1],Y)
}
\end{equation}
is a homotopy pullback square.
\end{proposition}
\begin{proof}
Since the vertical morphisms are Kan fibrations by Proposition \ref{quillenbifun}, a standard result implies that the square is a homotopy pullback square if and only if, for each parallel pair of morphisms $u,v \colon x \lra y$ in $X$, the induced morphism between fibres, i.e. the morphism 
\begin{equation} \label{dblhommap}
\Hom^2(f) \colon \Hom_X^2(u,v) \lra \Hom_Y^2(fu,fv),
\end{equation}
is an equivalence of Kan complexes. But, by Lemma \ref{hlemma} applied the Quillen adjunction $\Sigma^2 \dashv \Hom^2$ (\ref{dblsusphomadj}) and the composite Quillen adjunction
\begin{equation*}
 \xymatrix{
\partial\Theta_2[1;1]/[\Theta_2^\mathrm{op},\Set] \ar@<-1.5ex>[rr]^-{\hdash}_-{\Hom} && \ar@<-1.5ex>[ll]_-{\Sigma} \partial\Delta[1]/[\Delta^\mathrm{op},\Set] \ar@<-1.5ex>[rr]^-{\hdash}_-{\Hom} && \ar@<-1.5ex>[ll]_-{\Sigma_1} [\Delta^\mathrm{op},\Set]
}
\end{equation*}
(note that both left adjoints send $\Delta[0]$ to the boundary inclusion $\partial\Theta_2[1;1] \lra \Theta_2[1;1]$),
the morphism (\ref{dblhommap}) is an equivalence for each parallel pair of morphisms $u,v \colon x \lra y$ in $X$ if and only if the morphism 
$$f \colon \Hom_{\Hom_X(x,y)}(u,v) \lra \Hom_{\Hom_Y(fx,fy)}(fu,fv)$$
is an equivalence for each parallel pair of morphisms $u,v \colon x \lra y$ in $X$, i.e.\ if and only if $f \colon X \lra Y$ is locally fully faithful. 
\end{proof}

\begin{proposition} \label{genthm}
Let $\mathsf{D}$ be a class of $\Theta_2$-sets with the following properties:
\begin{enumerate}[font=\normalfont, label=(\alph*)]
\item $\mathsf{D}$ is saturated by monomorphisms in $[\Theta_2^\mathrm{op},\Set]$ \textup{(}in the sense of \cite[D\'efinition 1.1.12]{MR2294028}\textup{)},
\item any $\Theta_2$-set weakly equivalent in the model structure for $2$-quasi-categories to an object of $\mathsf{D}$ belongs to $\mathsf{D}$, and
\item $\Theta_2[1;1] \in \mathsf{D}$.
\end{enumerate}
Then every $\Theta_2$-set belongs to $\mathsf{D}$.
\end{proposition}
\begin{proof}
By (a) and (c), the $\Theta_2$-sets $\Theta_2[0]$ and $\Theta_2[1;0]$ belong to $\mathsf{D}$, since they are retracts of $\Theta_2[1;1]$.

 Let $\mathsf{C}$ denote the class of $2$-graphs $X$ for which $g_!(X) \in \mathsf{D}$ (see \S\ref{spinerecall}). By (a) and \cite[Remarque 1.1.13]{MR2294028}, this class $\mathsf{C}$ is saturated by monomorphisms in $[\mathbb{G}_2^\mathrm{op},\Set]$. The class $\mathsf{C}$ contains the representable $2$-graphs by the previous paragraph, and hence  contains every $2$-graph by \cite[Proposition 8.2.8]{MR2294028}, since $\mathbb{G}_2$ is a direct Reedy category (in which every morphism is a monomorphism). Hence, for every object $[n;\bm{m}] \in \Theta_2$, the spine $I[n;\bm{m}] = g_!(n;\bm{m})$ belongs to the class $\mathsf{D}$. 

Property (b) now implies that, for every $[n;\bm{m}] \in \Theta_2$, the representable $\Theta_2$-set $\Theta_2[n;\bm{m}]$ belongs to $\mathsf{D}$, since it is weakly equivalent in the model structure for $2$-quasi-categories to the object $I[n;\bm{m}] \in \mathsf{D}$ via the spine inclusion $I[n;\bm{m}] \lra \Theta_2[n;\bm{m}]$. 

Finally, since $\Theta_2$ is a \emph{cat{\'e}gorie squelletique r{\'e}guli{\`e}re} (\S \ref{reedyrecall}), we may conclude by another application of \cite[Proposition 8.2.8]{MR2294028} that every $\Theta_2$-set belongs to the class $\mathsf{D}$. 
\end{proof}

\begin{theorem} \label{fundythm}
Let $f \colon X \lra Y$ be a morphism of $2$-quasi-categories. Then the following are equivalent. 
\begin{enumerate}[font=\normalfont, label=(\roman*)]
\item $f \colon X \lra Y$ is an equivalence of $2$-quasi-categories.
\item $f \colon X \lra Y$ is essentially surjective on objects and fully faithful.
\item The morphism 
\begin{equation} \label{gen}
\bHom(1,f) : \bHom(\Theta_2[1;1],X) \lra \bHom(\Theta_2[1;1],Y)
\end{equation}
is an equivalence of Kan complexes.
\end{enumerate}
\end{theorem}
\begin{proof}
Suppose $f \colon X \lra Y$ is an equivalence of $2$-quasi-categories. Then, by the Quillen adjunction $\pi^* \dashv \tau^*$ of Proposition \ref{undqadj}, the morphism of underlying quasi-categories $\tau^*(f) \colon \tau^*(X) \lra \tau^*(Y)$ is an equivalence, and therefore is essentially surjective on objects. Hence $f$ is essentially surjective on objects by Observation \ref{esoobs}. Furthermore, by the Quillen adjunction $\Sigma \dashv \Hom$ of Proposition \ref{susphomqadjthefirst}, $f$ is fully faithful. This proves the implication (i) $\implies$ (ii).

Now, suppose $f \colon X \lra Y$ is essentially surjective on objects and fully faithful. Then, by Proposition \ref{undequiv}, the morphism of underlying quasi-categories $\tau^*(f)$ is an equivalence. Hence, by Proposition \ref{quillenbifun} and the natural isomorphism (\ref{enradj}),  the morphism of Kan complexes
$$\bHom(1,f) : \bHom(\pi^*(S),X) \lra \bHom(\pi^*(S),Y)$$
 is an equivalence for every  simplicial set $S$. In particular, since $\partial\Theta_2[1;1] \cong \pi^*(\Sigma_1(\partial\Delta[1]))$, the bottom morphism of the commutative square (\ref{hcart}) is an equivalence of Kan complexes. But $f$ is locally fully faithful by assumption, and so this square is a homotopy pullback square by Proposition \ref{hcartprop}, whence the morphism (\ref{gen})
is an equivalence of Kan complexes. This proves the implication (ii) $\implies$ (iii).

Finally, suppose the morphism (\ref{gen}) is an equivalence of Kan complexes. Let $\mathsf{D}$ denote the class of $\Theta_2$-sets $A$ for which the morphism
$$\bHom(1,f) : \bHom(A,X) \lra \bHom(A,Y)$$
is an equivalence of Kan complexes. The class $\mathsf{D}$ is saturated by monomorphisms by Proposition \ref{quillenbifun}, \cite[Remarque 1.4.16]{MR2294028}, and \cite[Lemme 1.1.15]{MR2294028}. Furthermore, by Proposition \ref{quillenbifun} and the two-out-of-three property, any $\Theta_2$-set weakly equivalent to an object of $\mathsf{D}$ belongs to $\mathsf{D}$. Since the class $\mathsf{D}$ contains the object $\Theta_2[1;1]$ by assumption, Proposition \ref{genthm} implies that every $\Theta_2$-set belongs to $\mathsf{D}$. The Yoneda lemma now implies that $f$ is an equivalence of $2$-quasi-categories. This proves the implication (iii) $\implies (i)$, which completes the proof of the theorem. 
\end{proof}

To conclude this section, we apply Theorem \ref{fundythm} and the homotopy bicategory construction of \S\ref{hobicatsec} to prove an intrinsic characterisation of those $2$-quasi-categories which are equivalent to the coherent nerve of a bicategory. Recall (see \cite[\S26]{joyalnotes} or \cite[\S3]{camplan}) that a quasi-category is said to be \emph{$1$-truncated} if it is equivalent to the nerve of a category (or, equivalently, if each of its hom-spaces is a $0$-type).

\begin{definition} \label{2truncdef}
A $2$-quasi-category $X$ is said to be \emph{$2$-truncated} if, for each pair of objects $x,y \in X$, the hom-quasi-category $\Hom_X(x,y)$ is $1$-truncated.
\end{definition}

\begin{theorem} \label{2truncchar}
Let $X$ be a $2$-quasi-category. Then the following are equivalent.
\begin{enumerate}[font=\normalfont, label=(\roman*)]
\item $X$ is $2$-truncated.
\item The unit morphism $X \lra N(\Ho X)$ of the adjunction of \textup{Theorem \ref{honadjthm}} is an equivalence of $2$-quasi-categories.
\item $X$ is equivalent to the coherent nerve of a bicategory.
\end{enumerate}
\end{theorem}
\begin{proof}
By construction, the unit morphism $X \lra N(\Ho X)$ is bijective on objects, and is given on hom-quasi-categories by the unit morphism $\Hom_X(x,y) \lra \ho(\Hom_X(x,y))$. Hence, by Theorem \ref{fundythm}, the unit morphism $X \lra N(\Ho X)$ is an equivalence of $2$-quasi-categories if and only if each hom-quasi-category $\Hom_X(x,y)$ of $X$ is $1$-truncated, i.e.\ if and only if $X$ is $2$-truncated. This proves the equivalence of (i) and (ii).

The implication (ii) $\implies$ (iii) is immediate. The implication (iii) $\implies$ (i) follows from the observations that the coherent nerve of a bicategory is a $2$-truncated $2$-quasi-category (by Theorem \ref{fibcor} and Example \ref{nbichoms}), and that any $2$-quasi-category equivalent to a $2$-truncated $2$-quasi-category is $2$-truncated.
\end{proof}

\section{Bicategories vs Rezk's weak $2$-categories} \label{vssect}
In this section, we prove (Theorem \ref{firstqequiv}) that the Quillen adjunction $\tau_b \dashv N \colon \Bicat_\mathrm{s} \lra [\Theta_2^\mathrm{op},\Set]$ of Theorem \ref{firstqadjthm} is moreover a Quillen equivalence between Lack's model structure for bicategories and the model structure for $2$-truncated $2$-quasi-categories, which we construct (Proposition \ref{modstrthm}) as the Bousfield localisation of Ara's model structure for $2$-quasi-categories with respect to the boundary inclusion $\partial\Theta_2[1;3] \lra \Theta_2[1;3]$. Moreover, we prove (Theorem \ref{prop2truncboth}) that the composite of this adjunction with an adjunction due to Ara is a Quillen equivalence between Lack's model structure for bicategories and Rezk's model structure for $(2,2)$-$\Theta$-spaces on the category of simplicial presheaves over $\Theta_2$. 

The one-dimensional analogues of these results were proved in \cite{camplan}, which was written (in part) to provide the necessary background for this section, and whose arguments we closely follow. 

\begin{recall}[local objects in model categories] \label{localrecall}
Let $f \colon A \lra B$ be a morphism in a model category $\mathcal{M}$. An object $X$ of $\mathcal{M}$ is said to be \emph{local with respect to $f$} (cf.\ \cite[Definition 3.1.4]{MR1944041}) if the induced morphism of derived hom-spaces
\begin{equation*}
\underline{\Ho \mathcal{M}}(f,X) : \underline{\Ho \mathcal{M}}(B,X) \lra \underline{\Ho \mathcal{M}}(A,X)
\end{equation*}
is an isomorphism in the homotopy category of Kan complexes.

It is a straightforward exercise to show that, for any cofibrant object $S$ of a model category $\mathcal{M}$ and any morphism
\begin{equation} \label{downtrianglepre}
\cd[@C=2em]{
& S \ar[dl]_-{a} \ar[dr]^-{b} \\
A \ar[rr]_-{f} && B
}
\end{equation}
in the category $S/\mathcal{M}$,
a fibrant object $X$ of $\mathcal{M}$ is local with respect to the morphism $f \colon A \lra B$ in $\mathcal{M}$ if and only if, for every morphism $x \colon S \lra X$ in $\mathcal{M}$, the object $x \colon S \lra X$ of $S/\mathcal{M}$ is local with respect to the morphism (\ref{downtrianglepre}) in the category $S/\mathcal{M}$ endowed with the induced model structure.
\end{recall}

Recall the suspension functor  $\Sigma \colon [\Delta^\mathrm{op},\Set] \lra [\Theta_2^\mathrm{op},\Set]$ from \S\ref{susphomadjsubsect}. 

\begin{lemma} \label{locprop}
Let $f \colon A \lra B$ be a morphism of simplicial sets. A $2$-quasi-category $X$ is local with respect to the morphism $\Sigma(f) \colon \Sigma A \lra \Sigma B$ in the model structure for $2$-quasi-categories if and only if, for each pair of objects $x,y \in X$, the hom-quasi-category $\Hom_X(x,y)$ is local with respect to the morphism $f \colon A \lra B$ in the model structure for quasi-categories.
\end{lemma}
\begin{proof}
By Recollection \ref{localrecall}, it suffices to show that, for each pair of objects $x,y$ of a $2$-quasi-category $X$, the hom-quasi-category $\Hom_X(x,y)$ is local with respect to the morphism $f \colon A\lra B$ in the model structure for quasi-categories if and only if the bi-pointed $2$-quasi-category $(x,y) \colon \partial\Theta_2[1;0] \lra X$ is local with respect to the morphism
\begin{equation} \label{downtriangle}
\cd[@C=.5em]{
& \partial\Theta_2[1;0] \ar[dl]_-{(\bot,\top)} \ar[dr]^-{(\bot,\top)} \\
\Sigma A \ar[rr]_-{\Sigma(f)} && \Sigma B
}
\end{equation}
in the model structure for bi-pointed $2$-quasi-categories. But this follows by an application of \cite[Proposition 3.1.12]{MR1944041} to the suspension-hom Quillen adjunction $\Sigma \dashv \Hom$ of Proposition \ref{susphomqadjthefirst}.
\end{proof}

Using this lemma, we can characterise the $2$-truncated $2$-quasi-categories of Definition \ref{2truncdef} as a class of local fibrant objects.

\begin{proposition} \label{2qcatloc}
A $2$-quasi-category is $2$-truncated if and only if it is local with respect to the boundary inclusion $\partial\Theta_2[1;3] \lra \Theta_2[1;3]$ in the model structure for $2$-quasi-categories.
\end{proposition}
\begin{proof}
As in the proof of Proposition \ref{susphomqadjthefirst}, the boundary inclusion $\partial\Theta_2[1;3] \lra \Theta_2[1;3]$ is the suspension of the boundary inclusion $\partial\Delta[3] \lra \Delta[3]$. Hence the result follows, by Lemma \ref{locprop}, from the $n=1$ case of \cite[Proposition 3.23]{camplan}, which states that a quasi-category is $1$-truncated if and only if it is local with respect to the boundary inclusion $\partial\Delta[3] \lra \Delta[3]$ in the model structure for quasi-categories. 
\end{proof}

\begin{proposition} \label{modstrthm}
There exists a model structure on the category $[\Theta_2^\mathrm{op},\Set]$ of $\Theta_2$-sets whose cofibrations are the monomorphisms, and whose fibrant objects are the $2$-truncated $2$-quasi-categories. This model structure is the Bousfield localisation of Ara's model structure for $2$-quasi-categories with respect to the boundary inclusion $\partial\Theta_2[1;3] \lra \Theta_2[1;3]$.
\end{proposition}
\begin{proof}
The Bousfield localisation of Ara's model structure for $2$-quasi-categories with respect to the boundary inclusion $\partial\Theta_2[1;3] \lra \Theta_2[1;3]$ exists since Ara's model structure is left proper and combinatorial. Its fibrant objects are those $2$-quasi-categories that are local with respect to the boundary inclusion $\partial\Theta_2[1;3] \lra \Theta_2[1;3]$, which are precisely the $2$-truncated $2$-quasi-categories by  Proposition \ref{2qcatloc}.
\end{proof}

We will call the model structure of Proposition \ref{modstrthm} the \emph{model structure for $2$-truncated $2$-quasi-categories}. Equipped with the results of the preceding sections, we can prove that the coherent nerve adjunction (\ref{firstlookqadj}) is a Quillen equivalence between this model structure and Lack's model structure for bicategories. 

\begin{theorem} \label{firstqequiv}
The adjunction \begin{equation*} \label{adjadjthm}
\xymatrix{
\Bicat_\mathrm{s} \ar@<-1.5ex>[rr]^-{\hdash}_-N && \ar@<-1.5ex>[ll]_-{\tau_b} [\Theta_2^\mathrm{op},\Set]
}
\end{equation*}
is a Quillen equivalence between the model structure for bicategories and the model structure for $2$-truncated $2$-quasi-categories.
\end{theorem}
\begin{proof}
Given Theorem \ref{firstqadjthm}, it remains to show (by, for instance, \cite[Theorem A.14]{camplan}) that a $2$-quasi-category is $2$-truncated if and only if it is equivalent to the coherent nerve of a bicategory. But this is precisely what was shown in Theorem \ref{2truncchar}.
\end{proof}


\subsection{$2$-quasi-categories vs Rezk's $\Theta_2$-spaces} \label{aravsrezk}
In \cite[\S11]{MR2578310}, Rezk defines an \emph{$(\infty,2)$-$\Theta$-space} -- though, following \cite{MR3350089}, we will use instead the term \emph{Rezk $\Theta_2$-space} -- to be a fibrant object of a certain model structure on the category $[(\Theta_2\times\Delta)^\mathrm{op},\Set]$ of simplicial presheaves over $\Theta_2$, which we will call the \emph{model structure for Rezk $\Theta_2$-spaces}. 

In \cite{MR3350089}, Ara proves two Quillen equivalences
\begin{equation*}
\xymatrix{
[\Theta_2^\mathrm{op},\Set] \ar@<-1.5ex>[rr]^-{\hdash}_-{t^!} && \ar@<-1.5ex>[ll]_-{t_!} [(\Theta_2\times\Delta)^\mathrm{op},\Set]
}
\qquad
\quad
\xymatrix{
[(\Theta_2\times\Delta)^\mathrm{op},\Set] \ar@<-1.5ex>[rr]^-{\hdash}_-{i^*} && \ar@<-1.5ex>[ll]_-{p^*} [\Theta_2^\mathrm{op},\Set]
}
\end{equation*}
between the model structure for $2$-quasi-categories and the model structure for Rezk  $\Theta_2$-spaces (note that Ara uses the notation $Real_{N_2\widetilde{\Delta_{\bullet}}} \dashv Sing_{N_2\widetilde{\Delta_{\bullet}}}$ for the adjunction we have denoted by $t_! \dashv t^!$). 
The functors $p^*$ and $i^*$ are defined by restriction along (the opposites of) the projection $p \colon \Theta_2 \times \Delta \lra \Theta_2$, and the functor $i \colon \Theta_2 \lra \Theta_2 \times \Delta$ that sends an object $[n;\bm{m}] \in \Theta_2$ to the pair $([n;\bm{m}],[0])$. The adjunction $t_! \dashv t^!$ is induced by Kan's construction (Recollection \ref{kanrecall}) from the functor $t \colon \Theta_2 \times \Delta \lra [\Theta_2^\mathrm{op},\Set]$ that sends a pair $([n;\bm{m}],[p])$ to the product $\Theta_2[n;\bm{m}] \times K([p])$, 
where $K$ is as in Example \ref{kexample}. Note that there is a natural isomorphism $i^*(t^!(X)) \cong X$ for $X \in [\Theta_2^\mathrm{op},\Set]$.

\subsection{Rezk's weak $2$-categories}
Let us say (in homage to the title of the paper \cite{MR2578310}) that a Rezk $\Theta_2$-space $Z$ is a \emph{weak $2$-category} if, for each pair of objects (i.e.\ $([0],[0])$-elements) $x,y \in Z$, the hom complete Segal space $M_Z(x,y)$ (see \cite[\S4.10]{MR2578310} and \cite[Proposition 8.3]{MR2578310})
 is $1$-truncated (i.e.\ each hom-space of $M_Z(x,y)$ is a $0$-type).  In \cite[\S11]{MR2578310}, Rezk shows that the weak $2$-categories (which he calls $(2,2)$-$\Theta$-spaces) are the fibrant objects of a Bousfield localisation of the model structure for Rezk $\Theta_2$-spaces. 
 We will call this Bousfield localisation \emph{Rezk's model structure for weak $2$-categories}.

\begin{proposition} \label{prop2trunc}
\begin{enumerate}[leftmargin=*, font=\normalfont]
\item A Rezk $\Theta_2$-space $Z$ is a weak $2$-category if and only if its underlying $2$-quasi-category $i^\ast (Z)$ is $2$-truncated.
\item A $2$-quasi-category $X$ is $2$-truncated if and only if the Rezk $\Theta_2$-space $t^!(X)$ is a weak $2$-category.
\end{enumerate}
\end{proposition}
\begin{proof}
(1) Let $Z$ be a Rezk $\Theta_2$-space. It is immediate from the definitions that, for each pair of objects $x,y\in Z$, the underlying quasi-category (i.e.\ the $0$th row) of the hom complete Segal space $M_Z(x,y)$ is the hom-quasi-category $\Hom_{i^*(Z)}(x,y)$. The result then follows from the fact (see \cite[Proposition 5.8]{camplan}) that a complete Segal space is $1$-truncated if and only if its underlying quasi-category is $1$-truncated.

(2) As noted above, for each $2$-quasi-category $X$, there is an isomorphism $X \cong i^\ast(t^!(X))$. Hence this result follows from (1).
\end{proof}

\begin{proposition} \label{prop2truncboth}
The adjunctions
\begin{equation*}
\xymatrix{
[\Theta_2^\mathrm{op},\Set] \ar@<-1.5ex>[rr]^-{\hdash}_-{t^!} && \ar@<-1.5ex>[ll]_-{t_!} [(\Theta_2\times\Delta)^\mathrm{op},\Set]
}
\qquad
\quad
\xymatrix{
[(\Theta_2\times\Delta)^\mathrm{op},\Set] \ar@<-1.5ex>[rr]^-{\hdash}_-{i^*} && \ar@<-1.5ex>[ll]_-{p^*} [\Theta_2^\mathrm{op},\Set]
}
\end{equation*}
are Quillen equivalences between the model structure for $2$-truncated $2$-quasi-categories and Rezk's model structure for weak $2$-categories.
\end{proposition}
\begin{proof}
By \cite[Theorem 8.4]{MR3350089} and \cite[Corollary 8.8]{MR3350089}, these adjunctions are Quillen equivalences between the model structures for $2$-quasi-categories and Rezk $\Theta_2$-spaces, of which the model structures in the statement are Bousfield localisations. The result then follows from Proposition \ref{prop2trunc} (by, for instance, \cite[Theorem A.15]{camplan}).  
\end{proof}

We may thus conclude that Rezk's weak $2$-categories are (Quillen) equivalent to the original weak $2$-categories, namely B\'enabou's bicategories.

\begin{theorem} \label{lackrezkequiv}
The composite adjunction
\begin{equation*} 
\xymatrix{
\Bicat_\mathrm{s} \ar@<-1.5ex>[rr]^-{\hdash}_-N && \ar@<-1.5ex>[ll]_-{\tau_b} [\Theta_2^\mathrm{op},\Set] \ar@<-1.5ex>[rr]^-{\hdash}_-{t^!} && \ar@<-1.5ex>[ll]_-{t_!} [(\Theta_2\times\Delta)^\mathrm{op},\Set]
}
\end{equation*}
 is a Quillen equivalence between Lack's model structure for bicategories and Rezk's model structure for weak $2$-categories.
\end{theorem}
\begin{proof}
This adjunction is the composite of the Quillen equivalence of Theorem \ref{firstqequiv} and one of the Quillen equivalences of Proposition \ref{prop2truncboth}, and is therefore a Quillen equivalence.
\end{proof}

\section{A triequivalence} \label{triequivsect} 
In this section, 
we prove (Theorem \ref{adjtriequivthm}) that the adjunction $\Ho \dashv N \colon \Bicat \lra \twoqcat_{2\text{-}\mathrm{tr}}$ (cf.\ Theorem \ref{honadjthm}) between the categories of bicategories and $2$-truncated $2$-quasi-categories underlies a triequivalence of categories enriched over the cartesian closed category $\Bicat$. 
We begin with a brief recollection of this (apparently little-known) cartesian closed structure. 

\subsection{The cartesian closed category of bicategories}
The category $\Bicat$ of bicategories and normal pseudofunctors is cartesian closed (see \cite{lackiconstalk} and \cite{algcof}). For each pair of bicategories $A$ and $B$, the objects of the internal hom bicategory $\underline{\Bicat}(A,B) = B^A$ are the normal pseudofunctors from $A$ to $B$, and its morphisms and $2$-cells are the ``enhanced pseudonatural transformations'' and modifications between them. Moreover, there is a biequivalence $\underline{\Bicat}(A,B) \lra \bHom(A,B)$ from the internal hom bicategory to the standard hom bicategory $\bHom(A,B)$  (see for instance \cite[\S1.3]{MR574662}), whose morphisms are the pseudonatural transformations in the usual sense (also called ``strong transformations'').

Note that the category $\twoqcat$ of $2$-quasi-categories is also cartesian closed. Indeed, \cite[Corollary 8.5]{MR3350089} implies that $\twoqcat$ is an exponential ideal in the category $[\Theta_2^\mathrm{op},\Set]$ of $\Theta_2$-sets.

\subsection{A $\Bicat$-enriched adjunction} \label{bicatenrsub}
Let $\underline{\Bicat}$ denote the canonical self-enrichment (see for instance \cite[\S1.6]{MR2177301}) of the cartesian closed category $\Bicat$, and let $\underline{\smash{\twoqcat}}$ denote the $\Bicat$-enriched category obtained by change of base (see \cite[Proposition II.6.3]{MR0225841}) of the canonical self-enrichment of the cartesian closed category $\twoqcat$ along the finite-product-preserving functor $\Ho \colon \twoqcat \lra \Bicat$ (see Lemma \ref{finprodlemma}); thus the hom-bicategories of $\underline{\smash{\twoqcat}}$ are given by the homotopy bicategories $\underline{\smash{\twoqcat}}(X,Y) = \Ho(Y^X)$.

\begin{proposition} 
The adjunction of  \textup{Theorem \ref{honadjthm}} underlies a $\Bicat$-enriched adjunction
\begin{equation} \label{honadjobsenr}
\xymatrix{
\underline{\Bicat} \ar@<-1.5ex>[rr]^-{\hdash}_-N && \ar@<-1.5ex>[ll]_-{\Ho} \underline{\smash{\twoqcat}}
}
\end{equation}
between the $\Bicat$-enriched categories of bicategories and $2$-quasi-categories.
\end{proposition}
\begin{proof}
Since the left adjoint of the ($\Set$-enriched) adjunction $\Ho \dashv N$ preserves finite products by Lemma \ref{finprodlemma}, it follows from \cite[\S5]{MR0255632} that this adjunction underlies a $\twoqcat$-enriched adjunction between the change of base of the $\Bicat$-category $\underline{\Bicat}$ along the functor $N \colon \Bicat \lra \twoqcat$ and the canonical self-enrichment of the cartesian closed category $\twoqcat$. The $\Bicat$-enriched adjunction of the statement is then obtained as the change of base of this $\twoqcat$-enriched adjunction along the functor $\Ho \colon \twoqcat \lra \Bicat$ (since $\Ho N \cong \mathrm{id}$).
\end{proof}

Let $\underline{\smash{\twoqcat}}_{2\text{-}\mathrm{tr}}$ denote the full $\Bicat$-enriched subcategory of $\underline{\smash{\twoqcat}}$ consisting of the $2$-truncated $2$-quasi-categories. Since the coherent nerve of a bicategory is a $2$-truncated $2$-quasi-category, the $\Bicat$-enriched adjunction (\ref{honadjobsenr})  restricts to one between $\underline{\Bicat}$ and $\underline{\smash{\twoqcat}}_{2\text{-}\mathrm{tr}}$. We will show that this latter $\Bicat$-enriched adjunction is moreover an adjoint triequivalence; but first, we must recall a few standard definitions and results from three-dimensional category theory. (Note that $\Bicat$-enriched categories may be thought of as a particular kind of strict tricategory \cite{MR1261589}, and so the general definitions of tricategory theory specialise to $\Bicat$-enriched categories.)

\begin{definition}[biequivalence in a $\Bicat$-enriched category] \label{biequivdef}
A morphism $f \colon X \lra Y$ in a $\Bicat$-enriched category $\mathcal{S}$ is said to be a \emph{biequivalence} if there exists a morphism $g \colon Y \lra X$ in $\mathcal{S}$ and equivalences $1_X \simeq gf$ and $1_Y \simeq fg$ in the hom-bicategories $\mathcal{S}(X,X)$ and $\mathcal{S}(Y,Y)$ respectively. 
\end{definition}

\begin{example}[biequivalences are biequivalences]
A standard result of bicategory theory implies that a normal pseudofunctor is a biequivalence in the $\Bicat$-category of bicategories if and only if it is a biequivalence in the sense of Definition \ref{lackdefs}. 
\end{example}

Using this definition, and the $\Bicat$-enrichment of the category of $2$-quasi-categories (see \S\ref{bicatenrsub}), we can give yet another characterisation of equivalences of $2$-quasi-categories. 

\begin{proposition} \label{biequivprop}
A morphism of $2$-quasi-categories is a biequivalence in the $\Bicat$-enriched category of $2$-quasi-categories  if and only if it is an equivalence of $2$-quasi-categories.
\end{proposition}
\begin{proof}
It suffices to show that, for a pair of morphisms of $2$-quasi-categories $f \colon X \lra Y$ and $g \colon Y \lra X$, there exist isomorphisms $1_X \cong gf$ and $1_Y \cong fg$ in the $2$-quasi-categories $X^X$ and $Y^Y$ if and only if there exist equivalences $1_X \simeq gf$ and $1_Y \simeq fg$ in the homotopy bicategories $\Ho(X^X)$ and $\Ho(Y^Y)$. The result thus follows from Proposition \ref{isoprop}.
\end{proof}

\begin{definition}[triequivalence]
A $\Bicat$-enriched functor $F \colon \mathcal{S} \lra \mathcal{T}$ is said to be a \emph{triequivalence} if it is:
\begin{enumerate}[font=\normalfont, label=(\roman*)]
\item \emph{triessentially surjective on objects}, i.e.\ for every object $T \in \mathcal{T}$, there exists an object $S \in \mathcal{S}$ and a biequivalence $FS \sim T$ in $\mathcal{T}$, and
\item \emph{a biequivalence on hom-bicategories}, i.e.\ for each pair of objects $S,T \in \mathcal{S}$, the normal pseudofunctor $F_{S,T} \colon \mathcal{S}(S,T) \lra \mathcal{T}(FS,FT)$ is a biequivalence of bicategories.
\end{enumerate}
\end{definition}

\begin{definition}[adjoint triequivalence] \label{adjtriequiv}
Let us say that a $\Bicat$-enriched adjunction $F \dashv G \colon \mathcal{T} \lra \mathcal{S}$ is an \emph{adjoint triequivalence} if each component of its unit is a biequivalence in $\mathcal{S}$ and each component of its counit is a biequivalence in $\mathcal{T}$.
\end{definition}

\begin{lemma} \label{triequivlem}
If a $\Bicat$-enriched adjunction $F \dashv G \colon \mathcal{T} \lra \mathcal{S}$ is an adjoint triequivalence, then the $\Bicat$-enriched functors $F \colon \mathcal{S} \lra \mathcal{T}$ and  $G \colon \mathcal{T} \lra \mathcal{S}$ are both triequivalences.
\end{lemma}
\begin{proof}
If $F \dashv G$ is an adjoint triequivalence, then each component of the unit $\eta_X \colon X \lra GFX$ is a biequivalence in $\mathcal{S}$, and each component of the counit $\varepsilon_A \colon FGA \lra A$ is a biequivalence in $\mathcal{T}$. Thus the $\Bicat$-enriched functors $G$ and $F$ are both triessentially surjective on objects. Moreover, they are both biequivalences on hom-bicategories, since, by standard properties of enriched adjunctions, they are given on hom-bicategories by the composite normal pseudofunctors
\begin{equation*}
\cd{
\mathcal{T}(A,B) \ar[rr]^-{\mathcal{T}(\varepsilon_A,B)}_-{\sim} && \mathcal{T}(FGA,B) \ar[r]^-{\theta}_-{\cong} & \mathcal{S}(GA,GB)
}
\end{equation*}
and
\begin{equation*}
\cd{\mathcal{S}(X,Y) \ar[rr]^-{\mathcal{S}(X,\eta_Y)}_-{\sim} && \mathcal{S}(X,GFY) \ar[r]^-{\theta^{-1}}_-{\cong} & \mathcal{T}(FX,FY)
}
\end{equation*}
respectively,
where $\theta$ denotes the hom-bicategory isomorphism defining the $\Bicat$-enriched adjunction $F \dashv G$. 
\end{proof}

\begin{remark}
One can show by a Yoneda-style argument that if either the left or the right adjoint of a $\Bicat$-enriched adjunction $F\dashv G$ is a triequivalence, then the $\Bicat$-enriched adjunction $F \dashv G$ is an adjoint triequivalence in the sense of Definition \ref{adjtriequiv}.
\end{remark}

We may thus conclude that the $\Bicat$-enriched categories of bicategories and $2$-truncated $2$-quasi-categories are triequivalent.

\begin{theorem}\label{adjtriequivthm}
The $\Bicat$-enriched adjunction \textup{(\ref{honadjobsenr})} restricts to an adjoint triequivalence 
\begin{equation*}
\xymatrix{
\underline{\Bicat} \ar@<-1.5ex>[rr]^-{\hdash}_-N && \ar@<-1.5ex>[ll]_-{\Ho} \underline{\smash{\twoqcat}}_{2\text{-}\mathrm{tr}}
}
\end{equation*}
between the  $\Bicat$-enriched categories of bicategories and $2$-truncated $2$-quasi-categories. Hence the $\Bicat$-enriched functors $N \colon \underline{\Bicat} \lra \underline{\smash{\twoqcat}}_{2\text{-}\mathrm{tr}}$ and $\Ho \colon \underline{\smash{\twoqcat}}_{2\text{-}\mathrm{tr}} \lra \underline{\Bicat}$ are both triequivalences of $\Bicat$-enriched categories. %
\end{theorem}
\begin{proof}
Since the coherent nerve of a bicategory is a $2$-truncated $2$-quasi-category, the $\Bicat$-enriched adjunction \textup{(\ref{honadjobsenr})} restricts to one between $\underline{\Bicat}$ and $\underline{\smash{\twoqcat}}_{2\text{-}\mathrm{tr}}$.  
Each component of the counit of this restricted $\Bicat$-enriched adjunction (\ref{honadjobsenr}) is an isomorphism by Theorem \ref{ffthm}, and each component of the unit is a biequivalence by Theorem \ref{2truncchar} and Proposition \ref{biequivprop}. 
Hence the restricted $\Bicat$-enriched adjunction is an adjoint triequivalence. The second statement then follows from Lemma \ref{triequivlem}.
\end{proof}

\section{Ara's conjecture} \label{araconsec}
The primary goal of this section is to prove (Theorem \ref{fibreplthm}) that the coherent nerve $NA$ of a $2$-category $A$ is a fibrant replacement of its strict nerve $N_sA$ in the model structure for $2$-quasi-categories. To prove this result, we will first prove a Quillen equivalence (Theorem \ref{qe10}) between the Hirschowitz--Simpson--Pellissier model structure for Joyal-enriched Segal categories (see \S\ref{hspmodstr}) and Ara's model structure for $2$-quasi-categories, which is of interest in its own right. Using this Quillen equivalence, we will also prove Ara's conjecture (Theorem \ref{araconjthm}) that a $2$-functor is a biequivalence if and only if it is sent by the strict nerve functor $N_s \colon \twocat \lra [\Theta_2^\mathrm{op},\Set]$ to a weak equivalence in the model structure for $2$-quasi-categories. 

We begin by describing the Hirschowitz--Simpson--Pellissier model structure for Joyal-enriched Segal categories. This model structure is an instance of a class of model structures for enriched Segal categories constructed in Simpson's book \cite{MR2883823}, generalising earlier work of Hirschowitz and Simpson on Segal $n$-categories  \cite{hirschsimp} and Pellissier's thesis \cite{pellthesis}. The advantage of this model structure is the following simple description of the weak equivalences between a class of objects more general than the fibrant objects, namely the Joyal-enriched Segal categories. 

\subsection{Equivalences of Joyal-enriched Segal categories}
Let $\tau_0 \colon [\Delta^\mathrm{op},\Set] \lra \Set$ denote the functor that sends a simplicial set to the set of isomorphism classes of objects of its homotopy category. Since this functor preserves pullbacks over discrete objects, and sends weak categorical equivalences to bijections, for any Joyal-enriched Segal category $X$ (see Definition \ref{joysegdef}), the simplicial set
\begin{equation*}
\cd{
\Delta^\mathrm{op} \ar[r]^-{X} & [\Delta^\mathrm{op},\Set] \ar[r]^-{\tau_0} & \Set
}
\end{equation*}
 is isomorphic to the nerve of a category $h(X)$, which is called the \emph{homotopy category} of $X$. 
 
\begin{definition} \label{joysegequivdef}
Let us say that a morphism $f \colon X \lra Y$ of Joyal-enriched Segal categories is an \emph{equivalence \textup{(}of Joyal-enriched Segal categories\textup{)}} if:
\begin{enumerate}[(i)]
\item the functor $h(f) \colon h(X) \lra h(Y)$ is essentially surjective on objects, and
\item  for each pair of objects $x,y \in X$, the morphism $f \colon X(x,y) \lra Y(fx,fy)$ is a weak categorical equivalence.
\end{enumerate}
\end{definition}

\begin{example}[levelwise nerves of $2$-categories] \label{levelwiseex}
Let $N_l \colon \twocat \lra [(\Delta\times\Delta)^\mathrm{op},\Set]$ denote the functor that sends a $2$-category $A$ to its \emph{levelwise nerve} $N_lA$, i.e.\ to the composite
\begin{equation*}
\cd[@C=2.5em]{
\Delta^\mathrm{op} \ar[r]^-{N_{st}A} & \Cat \ar[r]^-N & [\Delta^\mathrm{op},\Set]
}
\end{equation*}
of its standard nerve (Recollection \ref{stnrecall}) and the simplicial nerve functor. It is immediate from the definitions that the levelwise nerve of a $2$-category is a Joyal-enriched Segal category, and  
that a $2$-functor $F \colon A \lra B$ is a biequivalence if and only if the morphism $N_l(F) \colon N_lA \lra N_lB$ is an equivalence of Joyal-enriched Segal categories.
\end{example}

\subsection{The model structure for Joyal-enriched Segal categories}  \label{hspmodstr}
Let $\PCat$ denote the full subcategory of the category of bisimplicial sets consisting of the \emph{precategories}, i.e.\ the bisimplicial sets $X$ whose first column $X_0$ is discrete.  By \cite[Theorem 19.2.1]{MR2883823} applied to Joyal's model structure for quasi-categories, the category $\PCat$ admits a model structure in which the cofibrations are the monomorphisms, and the fibrant objects are the Joyal-enriched Segal categories that are Reedy fibrant as simplicial objects in the model structure for quasi-categories. Importantly, a morphism of Joyal-enriched Segal categories is a weak equivalence in this model structure if and only if it is an equivalence of Joyal-enriched Segal categories (in the sense of Definition \ref{joysegequivdef}). 

We will call this model structure the \emph{\textup{(}Hirschowitz--Simpson--Pellissier\textup{)} model structure for Joyal-enriched Segal categories}. 
Note that any Reedy fibrant replacement (with respect to the model structure for quasi-categories) of a Joyal-enriched Segal category $X$ is moreover a fibrant replacement of $X$ in the model structure for Joyal-enriched Segal categories.

\medskip

Using the results of \S\S\ref{hobicatsec}--\ref{fundythmsec}, we can prove 
a Quillen equivalence between the model structures for Joyal-enriched Segal categories and $2$-quasi-categories. 
In this proof, we will use the following variant of Proposition \ref{welemma}, which involves the horizontal spine inclusions (\S\ref{horspsubsect}).

\begin{proposition} \label{welemma2}
Let $\mathcal{M}$ be a model category and let  $F \colon [\Theta_2^\mathrm{op},\Set] \lra \mathcal{M}$ be a cocontinuous functor that sends monomorphisms to cofibrations. Then $F$ sends the weak equivalences in the model structure for $2$-quasi-categories to weak equivalences in $\mathcal{M}$ if and only if it sends  the following morphisms to weak equivalences in $\mathcal{M}$:
\begin{enumerate}[font=\normalfont,  label=(\roman*)]
\item for each $[n;\bm{m}] \in \Theta_2$, the projection $\mathrm{pr}_2 \colon J \times \Theta_2[n;\bm{m}] \lra \Theta_2[n;\bm{m}]$,
\item for each $[n;\bm{m}] \in \Theta_2$, the horizontal spine inclusion $Sp[n;\bm{m}] \lra \Theta_2[n;\bm{m}]$,
\item for each $[m] \in \Delta$, the spine inclusion $I[1;m] \lra \Theta_2[1;m]$, and
\item the morphism  $j_2 \colon J_2 \lra \Theta_2[1;0]$.
\end{enumerate}
\end{proposition}
\begin{proof}
It is evident from the proof of Lemma \ref{horsplemma} that the functor $F$ sends all spine inclusions $I[n;\bm{m}] \lra \Theta_2[n;\bm{m}]$ to weak equivalences if and only if it sends all ``vertical'' spine inclusions $I[1;m] \lra \Theta_2[1;m]$ and all horizontal spine inclusions $Sp[n;\bm{m}] \lra \Theta_2[n;\bm{m}]$ to weak equivalences. The result then follows from Proposition \ref{welemma}.
\end{proof}

\begin{observation}[preservation of underlying enriched graphs] \label{graphpresobs}
Let $d^* \colon [\Theta_2^\mathrm{op},\Set] \lra \PCat$ denote the functor that sends a $\Theta_2$-set $U$ to its underlying bisimplicial set $d^*(U)$ (see \S\ref{undbisimp}), which is a precategory. It is immediate from the definitions that this functor preserves the underlying simplicially enriched graph of a $\Theta_2$-set $U$; that is, the $\Theta_2$-set $U$ and the precategory $d^*(U)$ have the same set of objects $U_0$, and for each pair of objects $x,y \in U_0$, the hom-simplicial-sets $\Hom_U(x,y)$ and $(d^*U)(x,y)$ coincide. 

The right adjoint functor $d_* \colon \PCat \lra [\Theta_2^\mathrm{op},\Set]$ (defined for precategories as it is for bisimplicial sets) also preserves underlying simplicially enriched graphs. This follows by adjointness from the observations that the objects of a precategory $X$ are in natural bijection with morphisms from the terminal precategory, i.e.\ $d^*(\Theta_2[0])$, to $X$, and that, for each pair of objects $x,y \in X$, the $m$-simplices of the hom-simplicial-set $X(x,y)$ are in natural bijection with the endpoint-preserving morphisms from the standard nerve of the two-object suspension of $\Delta[m]$, i.e.\ the precategory $d^*(\Theta_2[1;m])$, to $X$. 
\end{observation}

\begin{theorem} \label{qe10}
The adjunction
\begin{equation*}
\xymatrix{
\PCat \ar@<-1.5ex>[rr]^-{\hdash}_-{d_*} && \ar@<-1.5ex>[ll]_-{d^*} [\Theta_2^\mathrm{op},\Set]
}
\end{equation*}
is a Quillen equivalence between the Hirschowitz--Simpson--Pellissier model structure for Joyal-enriched Segal categories and Ara's model structure for $2$-quasi-categories.
\end{theorem}
\begin{proof}
The left adjoint $d^*$ evidently preserves monomorphisms, i.e.\ cofibrations. To prove that the adjunction is a Quillen adjunction, it therefore remains to verify the hypotheses of Proposition \ref{welemma2}. Since the functor $d^*$ sends the strict nerve of a $2$-category to its levelwise nerve, it follows from Example \ref{levelwiseex} that the functor $d^*$ sends the morphisms (i) and (iv) in Proposition \ref{welemma2} to equivalences of Joyal-enriched Segal categories. Also, for each $[m] \in \Delta$, the spine inclusion $I[1;m] \lra \Theta_2[1;m]$ is sent by $d^*$ to a bijective-on-objects equivalence of Joyal-enriched Segal categories (which is given on the only non-trivial hom-simplicial-set by the spine inclusion $I[m] \lra \Delta[m]$, which is a weak categorical equivalence). Furthermore, by \cite[Theorem 16.1.2]{MR2883823}, the horizontal spine inclusions are sent by $d^*$ to weak equivalences in the model structure for Joyal-enriched Segal categories. This proves that the adjunction $d^* \dashv d_*$ is a Quillen adjunction.

To prove that the adjunction is moreover a Quillen equivalence, it remains to show that its derived unit and derived counit are isomorphisms. Let $X$ be a fibrant object in the model structure for Joyal-enriched Segal categories.  Since $d^* \dashv d_*$ is a Quillen adjunction, the $\Theta_2$-set $d_*(X)$ is a $2$-quasi-category, whence by Theorem \ref{segalthm}, the precategory $d^*(d_*(X))$ is a Joyal-enriched Segal category.  Hence the counit morphism $d^*(d_*(X)) \lra X$ is a morphism of Joyal-enriched Segal categories, and is moreover an equivalence of such, since it is bijective on objects and an isomorphism on hom-simplicial-sets by Observation \ref{graphpresobs}, and hence is a weak equivalence. Thus the derived counit of the Quillen adjunction $d^* \dashv d_*$ is an isomorphism.

Now, let $X$ be a $2$-quasi-category, and let $r \colon d^*(X) \lra Y$ be a fibrant replacement of $d^*(X)$ in the model structure for Joyal-enriched Segal categories, which we may suppose to be bijective on objects. Since $d^*(X)$ is a Joyal-enriched Segal category by Theorem \ref{segalthm}, this morphism $r$ is an equivalence of Joyal-enriched Segal categories. It then follows by Observation \ref{graphpresobs} that the composite morphism
\begin{equation*}
\cd[@C=2.5em]{
X\ar[r]^-{\eta_X} & d_*(d^*(X)) \ar[r]^-{d_*(r)} & d_*(Y)
}
\end{equation*}
is a bijective-on-objects and fully faithful morphism of $2$-quasi-categories, which is therefore an equivalence of $2$-quasi-categories by Theorem \ref{fundythm}. Hence the derived unit of this Quillen adjunction $d^* \dashv d_*$ is an isomorphism. This completes the proof that the adjunction $d^* \dashv d_*$ is a Quillen equivalence.
\end{proof}

\begin{corollary} \label{createcor}
A morphism of $\Theta_2$-sets is a weak equivalence in the model structure for $2$-quasi-categories if and only if it is sent by the functor $d^* \colon [\Theta_2^\mathrm{op},\Set] \lra \PCat$ to a weak equivalence in the model structure for Joyal-enriched Segal categories.
\end{corollary}
\begin{proof}
Since every $\Theta_2$-set is cofibrant in the model structure for $2$-quasi-categories, this is a consequence of Theorem \ref{qe10}.
\end{proof}

\begin{remark}
The Quillen equivalence of Theorem \ref{qe10} can be thought of as an ``unenriched'' analogue of the simplicially enriched Quillen equivalence of \cite[Corollary 7.1]{brcomp} between complete Segal objects in complete Segal spaces and Rezk $\Theta_2$-spaces. However, we have been unable to deduce the results of this section from the results of \cite{brcomp}; our proofs use not only the Quillen equivalence of Theorem \ref{qe10}, but also properties of the model structures for enriched Segal categories (in particular, the characterisation of weak equivalences between not necessarily fibrant enriched Segal categories), whose analogues for complete Segal objects are not present in \cite{brcomp}.
\end{remark}

%

In \S\ref{nervefunsec}, we motivated the coherent nerve construction as a solution for the problem that the strict nerve of a $2$-category is not in general fibrant in the model structure for $2$-quasi-categories. Using Corollary \ref{createcor}, we can prove that the coherent nerve of a $2$-category is in fact a fibrant replacement of its strict nerve in this model structure. Note that, since every $2$-functor is a normal pseudofunctor, there is a canonical inclusion $N_sA \lra NA$ for every $2$-category $A$.

\begin{theorem} \label{fibreplthm}
Let $A$ be a $2$-category.  Then the inclusion $N_sA \lra NA$ of the strict nerve of $A$ into the coherent nerve of $A$ is a weak equivalence in the model structure for $2$-quasi-categories. 
\end{theorem}
\begin{proof}
The induced morphism of Joyal-enriched Segal categories $d^*(N_sA) \lra d^*(NA)$ is bijective on objects and an isomorphism on hom-quasi-categories, and is therefore a weak equivalence in the model structure for Joyal-enriched Segal categories. The result thus follows from Corollary \ref{createcor}.
\end{proof}


To conclude this section, we make another application of Corollary \ref{createcor} to prove Ara's conjecture (see \cite[\S7]{MR3350089}) that the strict nerve functor $N_s \colon \twocat \lra [\Theta_2^\mathrm{op},\Set]$ preserves and reflects weak equivalences. 

\begin{theorem}[Ara's conjecture] \label{araconjthm}
A $2$-functor is a biequivalence if and only if it is sent by the strict nerve functor $N_s \colon \twocat \lra [\Theta_2^\mathrm{op},\Set]$ to a weak equivalence in the model structure for $2$-quasi-categories.
\end{theorem}
\begin{proof}
It is immediate from the definitions that a $2$-functor is a biequivalence if and only if it is sent by the levelwise nerve functor $N_l \colon \twocat \lra \PCat$ to an equivalence of Joyal-enriched Segal categories. The result therefore follows from Corollary \ref{createcor} and the observation (cf.\ Recollection \ref{stnrecall}) that the levelwise nerve functor is naturally isomorphic to the composite
\begin{equation*}
\cd{
\twocat \ar[r]^-{N_s} & [\Theta_2^\mathrm{op},\Set] \ar[r]^-{d^*} & \PCat
}
\end{equation*}
of the strict nerve functor and the underlying bisimplicial set functor. 

Alternatively, this result can be deduced directly from Theorems \ref{firstqadjthm} and \ref{fibreplthm}.
\end{proof}

\section{From quasi-categories to $2$-quasi-categories} \label{sec1to2}
Recall (Observation \ref{nervecatrecall}) that the fully faithful functor $\pi^* \colon [\Delta^\mathrm{op},\Set] \lra [\Theta_2^\mathrm{op},\Set]$ sends the simplicial nerve of a category, which is a quasi-category, to its $2$-cellular nerve, which is a $2$-quasi-category. However, it is far from the case that this functor sends all quasi-categories to $2$-quasi-categories; in fact, as we prove in Proposition \ref{onlycats}, the nerves of categories turn out to be the only examples. 

The goal of this final section is to construct, for each quasi-category $X$, an explicit fibrant replacement of $\pi^*(X)$ in Ara's model structure for $2$-quasi-categories (see Theorem \ref{finalfinalthm}). (We note that the proof of this result will use, among other things, the results of \S\ref{mainthmsec}.) Along the way, we prove two Quillen equivalences (one in each direction; see Theorem \ref{finalqe1} and Corollary \ref{finalqe2}) between Joyal's model structure for quasi-categories and a model structure for ``locally Kan'' $2$-quasi-categories (Definition \ref{lockandef}), which we construct  as a Bousfield localisation of Ara's model structure for $2$-quasi-categories (Proposition \ref{lockanmodstr}).

\begin{proposition} \label{onlycats}
Let $X$ be a simplicial set. Then $\pi^{\ast} (X)$ is a $2$-quasi-category if and only if $X$ is isomorphic to the simplicial nerve of a category.
\end{proposition}
\begin{proof}
If $X$ is isomorphic to the simplicial nerve of a category $C$, then $\pi^*(X)$ is isomorphic to the $2$-cellular nerve of $C$, and so is a $2$-quasi-category by \cite[Proposition 7.10]{MR3350089} (or Corollary \ref{fibcor}). 

Conversely, suppose that $\pi^*(X)$ is a $2$-quasi-category. To prove that $X$ is isomorphic to the nerve of a category, it suffices to prove that the function
\begin{equation} \label{ineedalabel}
[\Delta^\mathrm{op},\Set](i_n,X) : [\Delta^\mathrm{op},\Set](\Delta[n],X) \lra [\Delta^\mathrm{op},\Set](I[n],X)
\end{equation}
is a bijection for each $n \geq 2$. Let $\pi_! \colon [\Theta_2^\mathrm{op},\Set] \lra [\Delta^\mathrm{op},\Set]$ denote the left adjoint of the functor $\pi^*$.

Let $n\geq 2$ be an integer.
The $2$-quasi-category  $\pi^*(X)$ has the right lifting property with respect to the spine inclusion $I[n;0,\ldots,0] \lra \Theta_2[n;0,\ldots,0]$, since the latter is a trivial cofibration in the model structure for $2$-quasi-categories. It follows, by adjointness, that $X$ has the right lifting property with respect to the induced morphism $\pi_!(I[n;0,\ldots,0]) \lra \pi_!(\Theta_2[n;0,\ldots,0])$, which is the spine inclusion $I[n] \lra \Delta[n]$. This proves that the function (\ref{ineedalabel}) is surjective. 

Now, observe that the function (\ref{ineedalabel}) is injective if and only if $X$ has the right lifting property with respect to the pushout-corner map of the commutative square on the left below,
\begin{equation*}
\cd{
I[n] + I[n] \ar[r] \ar[d] & I[n] \ar[d]\\
\Delta[n] + \Delta[n] \ar[r] & \Delta[n]
}
\qquad
\qquad
\cd{
I[n;0,\ldots,0] + I[n;0,\ldots,0] \ar[r] \ar[d] & I[n;1,\ldots,1] \ar[d]  \\
\Theta_2[n;0,\ldots,0] + \Theta_2[n;0,\ldots,0] \ar[r] & \Theta_2[n;1,\ldots,1]
}
\end{equation*}
which is the image under the left adjoint functor $\pi_!$ of the pushout-corner map of the commutative square on the right above (whose bottom morphism has the components $(\mathrm{id};\delta^1,\ldots,\delta^1)$ and $(\mathrm{id};\delta^0,\ldots,\delta^0)$). But this latter pushout-corner map is a trivial cofibration in the model structure for $2$-quasi-categories:\ it is a monomorphism by the final paragraph of the proof of Theorem \ref{segalthm}, and it is a weak equivalence by the two-out-of-three property (since the vertical morphisms in the square are trivial cofibrations). Hence the  $2$-quasi-category $\pi^*(X)$ has the right lifting property with respect to the pushout-corner map of the square on the right above, whence by adjointness, the simplicial set $X$ has the right lifting property with respect to the pushout-corner map of the square on the left above.  This proves that the function (\ref{ineedalabel}) is injective, and thus completes the proof. 
\end{proof}

Nevertheless, since quasi-categories are a model for $(\infty,1)$-categories and $2$-quasi-categories are a model for $(\infty,2)$-categories, there ought to be some construction which assigns to each quasi-category its corresponding ``locally $\infty$-groupoidal'' $2$-quasi-category. And since we have said that the right Quillen functor $\tau^* \colon [\Theta_2^\mathrm{op},\Set] \lra [\Delta^\mathrm{op},\Set]$ sends a $2$-quasi-category to its underlying quasi-category, it seems reasonable to expect that the functor induced by the weak-equivalence-preserving functor $\pi^*$ (Proposition \ref{undqadj})
\begin{equation} \label{locpi}
\Ho(\pi^*) \colon \Ho(\qcat) \lra \Ho(\twoqcat)
\end{equation}
between the homotopy categories of (the model structures for) quasi-categories and $2$-quasi-categories, which sends a quasi-category $X$ to a fibrant replacement of $\pi^*(X)$ in the model structure for $2$-quasi-categories, is such a construction.  Indeed, we will show in Corollary \ref{finalcor} that the functor (\ref{locpi}) is fully faithful, and that its essential image consists of the ``locally Kan'' $2$-quasi-categories, in the sense of the following definition.

\begin{definition} \label{lockandef}
A $2$-quasi-category $X$ is said to be \emph{locally Kan} if, for each pair of objects $x,y \in X$, the hom-quasi-category $\Hom_X(x,y)$ is a Kan complex.
\end{definition}

The following proposition contains a few equivalent characterisations of the locally Kan $2$-quasi-categories, one of which involves the left adjoint functor $\tau_2 \colon [\Theta_2^\mathrm{op},\Set] \lra \twocat$ (see Observation \ref{coobs}), which we say sends a $\Theta_2$-set to its \emph{$2$-category truncation}. We say that a bicategory $A$ is \emph{locally groupoidal} if each of its hom-categories $A(a,b)$ is a groupoid.

\begin{proposition} \label{lockanequivs}
Let $X$ be a $2$-quasi-category. Then the following are equivalent.
\begin{enumerate}[font=\normalfont, label=(\roman*)]
\item $X$ is locally Kan.
\item The homotopy bicategory $\Ho(X)$ of $X$ is locally groupoidal.
\item The $2$-category truncation $\tau_2(X)$ of $X$ is locally groupoidal.
\item $X$ is local with respect to the morphism $(\mathrm{id};\sigma^0) \colon \Theta_2[1;1] \lra \Theta_2[1;0]$ in the model structure for $2$-quasi-categories.
\end{enumerate}
\end{proposition}
\begin{proof}
By construction, the hom-categories of the homotopy bicategory of $X$ are the homotopy categories of the hom-quasi-categories of $X$. Hence the equivalence of (i) and (ii) follows from Joyal's result  that a quasi-category is a Kan complex if and only if its homotopy category is a groupoid \cite[Corollary 1.4]{MR1935979}. 

By comparing their universal properties, we see that the  $2$-category truncation $\tau_2(X)$ of $X$ is isomorphic to the normal strictification (see \S\ref{nstrecall}) of the homotopy bicategory $\Ho(X)$ of $X$. Hence $\tau_2$ and $\Ho(X)$ are biequivalent, which implies the equivalence of (ii) and (iii).

The equivalence of (i) and (iv) follows by Proposition \ref{2qcatloc} from the fact (see, for instance, \cite[Proposition 3.30]{camplan}) that a quasi-category is a Kan complex if and only if it is local with respect to the morphism $\sigma^0 \colon \Delta[1] \lra \Delta[0]$ in the model structure for quasi-categories. 
\end{proof}

\begin{proposition} \label{lockanmodstr}
There exists a model structure on the category $[\Theta_2^\mathrm{op},\Set]$ of $\Theta_2$-sets whose cofibrations are the monomorphisms, and whose fibrant objects are the locally Kan $2$-quasi-categories. This model structure is the Bousfield localisation of Ara's model structure for $2$-quasi-categories with respect to the morphism $(\mathrm{id};\sigma^0) \colon \Theta_2[1;1] \lra \Theta_2[1;0]$. 
\end{proposition}
\begin{proof}
As in the proof of Proposition \ref{modstrthm}, this follows from the equivalence (i) $\Leftrightarrow$ (iv) of Proposition \ref{lockanequivs}.
\end{proof}

We will call the model structure of Proposition \ref{lockanmodstr} the \emph{model structure for locally Kan $2$-quasi-categories}.

\begin{definition}
A $\Theta_2$-set $U$ is said to be \emph{locally groupoidal} if some (and hence every) fibrant replacement of $U$ in the model structure for $2$-quasi-categories is  locally Kan. In particular, a $2$-quasi-category is locally groupoidal if and only if it is locally Kan.
\end{definition}


\begin{lemma} \label{localwes}
A morphism of locally groupoidal $\Theta_2$-sets is a weak equivalence in the model structure for $2$-quasi-categories if and only if it is a weak equivalence in the model structure for locally Kan $2$-quasi-categories.
\end{lemma}
\begin{proof}
Since the model structure for locally Kan $2$-quasi-categories is a Bousfield localisation of the model structure for $2$-quasi-categories, the statement holds for any morphism of locally Kan $2$-quasi-categories. The statement then extends to morphisms of arbitrary locally groupoidal $\Theta_2$-sets by the two-out-of-three property.
\end{proof}

To apply this lemma, we will use the following recognition principle for locally groupoidal $\Theta_2$-sets.

\begin{proposition} \label{locgrpprop}
A $\Theta_2$-set $U$ is locally groupoidal if and only if its $2$-category truncation $\tau_2(U)$ is locally groupoidal.
\end{proposition}
\begin{proof}
Let $U$ be a $\Theta_2$-set, and let $r \colon U \lra X$ be a fibrant replacement of $U$ in the model structure for $2$-quasi-categories. Since the functor $\tau_2 \colon [\Theta_2^\mathrm{op},\Set] \lra \twocat$ is left Quillen (see Remark \ref{part1endremark}), the $2$-functor $\tau_2(r) \colon \tau_2(U) \lra \tau_2(X)$ is a biequivalence. Hence the $2$-category $\tau_2(U)$ is locally groupoidal if and only if $\tau_2(X)$ is locally groupoidal, which, by Proposition \ref{lockanequivs}, is so if and only if the $2$-quasi-category $X$ is locally Kan. 
\end{proof}

\begin{proposition} \label{locgrpprop2}
For any simplicial set  $S$, the $\Theta_2$-set $\pi^*(S)$ is locally groupoidal.
\end{proposition}
\begin{proof}
By cocontinuity of left adjoints and the Yoneda lemma, 
the $2$-category $\tau_2(\pi^*(S))$ is a colimit $$\tau_2(\pi^*(S)) \cong \int^{[n]\in\Delta} S_n \times \st([n])$$ of locally groupoidal $2$-categories, and is therefore locally groupoidal. Hence the $\Theta_2$-set $\pi^*(S)$ is locally groupoidal by Proposition \ref{locgrpprop}. 
\end{proof}

In the following paragraphs, we will use Segal categories to obtain an explicit construction of a fibrant replacement of the $\Theta_2$-set $\pi^*(X)$ for each quasi-category $X$, and to prove that the model structures for quasi-categories and locally Kan $2$-quasi-categories are Quillen equivalent.

\subsection{The model structure for Segal categories}
Recall that a precategory $X$ is said to be a \emph{\textup{(}Kan-enriched\textup{)} Segal category} if, for each $n \geq 2$, the Segal map $X_n \lra X_1\times_{X_0} \cdots \times_{X_0} X_1$ is a weak homotopy equivalence. By \cite{hirschsimp,pellthesis,MR2341955} (or again by \cite{MR2883823}), the category $\PCat$ of precategories admits a model structure whose cofibrations are the monomorphisms, and whose fibrant objects are the Segal categories that are Reedy fibrant with respect to the model structure for Kan complexes.  
We will call this model structure the \emph{\textup{(}Hirschowitz--Simpson--Pellissier\textup{)} model structure for Segal categories}. The following lemma implies that this model structure is a Bousfield localisation of the model structure for Joyal-enriched Segal categories.

\begin{lemma} \label{lockanlemma}
A precategory $X$ is a fibrant object in the model structure for Segal categories if and only if it is fibrant in the model structure for Joyal-enriched Segal categories and, for each pair of objects $x,y \in X$, the hom-quasi-category $X(x,y)$ is a Kan complex.
\end{lemma}
\begin{proof}
This follows from the fact that the model structure for Kan complexes is a Bousfield localisation of the model structure for quasi-categories, in particular that a 
morphism of Kan complexes is a Kan fibration (resp.\ weak homotopy equivalence) if and only if it is an isofibration (resp.\ weak categorical equivalence).
\end{proof}

\begin{theorem} \label{pcat2qcat}
The adjunction
\begin{equation*}
\xymatrix{
\PCat \ar@<-1.5ex>[rr]^-{\hdash}_-{d_*} && \ar@<-1.5ex>[ll]_-{d^*} [\Theta_2^\mathrm{op},\Set]
}
\end{equation*}
is a Quillen equivalence between the Hirschowitz--Simpson--Pellissier model structure for Segal categories and the model structure for locally Kan $2$-quasi-categories.
\end{theorem}
\begin{proof}
Given Theorem \ref{qe10}, it suffices (by, for instance, \cite[Theorem A.15]{camplan}) to prove that a fibrant object $X$ in the model structure for Joyal-enriched Segal categories is fibrant in the model structure for (Kan-enriched) Segal categories  if and only the $2$-quasi-category $d_*(X)$ is locally Kan. Since the hom-quasi-categories of $d_*(X)$ coincide with the hom-quasi-categories of $X$ (by Observation \ref{graphpresobs}), this is precisely what was shown in Lemma \ref{lockanlemma}.
\end{proof}

We may now compose this Quillen equivalence with a Quillen equivalence due to Joyal and Tierney to obtain a Quillen equivalence between quasi-categories and locally Kan $2$-quasi-categories. Let $D \colon \Delta \lra \Theta_2$ denote the functor given by $D([n]) = [n;n,\ldots,n]$. This functor induces an adjunction $D^* \dashv D_*$ between the categories of simplicial sets and $\Theta_2$-sets by restriction and right Kan extension.

\begin{theorem} \label{finalqe1}
The adjunction
\begin{equation*}
\xymatrix{
[\Delta^\mathrm{op},\Set] \ar@<-1.5ex>[rr]^-{\hdash}_-{D_*} && \ar@<-1.5ex>[ll]_-{D^*} [\Theta_2^\mathrm{op},\Set]
}
\end{equation*}
is a Quillen equivalence between Joyal's model structure for quasi-categories and the model structure for locally Kan $2$-quasi-categories.
\end{theorem}
\begin{proof}
This follows from the observation that the adjunction $D^* \dashv D_*$ is equal to the composite
\begin{equation*}
\xymatrix{
[\Delta^\mathrm{op},\Set] \ar@<-1.5ex>[rr]^-{\hdash}_-{d_*} && \ar@<-1.5ex>[ll]_-{d^*} \PCat \ar@<-1.5ex>[rr]^-{\hdash}_-{d_*} && \ar@<-1.5ex>[ll]_-{d^*}  [\Theta_2^\mathrm{op},\Set]
}
\end{equation*}
of the Quillen equivalence of \cite[Theorem 5.7]{MR2342834} (whose left adjoint sends a precategory to its diagonal) between the model structures for quasi-categories and (Kan-enriched) Segal categories, and the Quillen equivalence of Theorem \ref{pcat2qcat}.
\end{proof}

\begin{observation}
Let $X$ be a quasi-category. Then the $2$-quasi-category $D_*(X)$ has the same objects as $X$, and its hom-quasi-categories are precisely the hom-spaces of $X$ (as defined in \S\ref{1dimfundy}).
\end{observation}

From Theorem \ref{finalqe1}, we may deduce a Quillen equivalence between quasi-categories and locally Kan $2$-quasi-categories in the opposite direction.

\begin{corollary} \label{finalqe2}
The adjunction 
\begin{equation*}
\xymatrix{
[\Theta_2^\mathrm{op},\Set] \ar@<-1.5ex>[rr]^-{\hdash}_-{\tau^*} && \ar@<-1.5ex>[ll]_-{\pi^*} [\Delta^\mathrm{op},\Set]
}
\end{equation*}
is a Quillen equivalence between the model structure for locally Kan $2$-quasi-categories and Joyal's model structure for quasi-categories.
\end{corollary}
\begin{proof}
This adjunction is a Quillen adjunction between these model structures by Theorem \ref{undqadj}. Since $\pi\circ D = \mathrm{id}$, the composite
\begin{equation*}
\xymatrix{
[\Delta^\mathrm{op},\Set] \ar@<-1.5ex>[rr]^-{\hdash}_-{D_*} && \ar@<-1.5ex>[ll]_-{D^*} [\Theta_2^\mathrm{op},\Set] \ar@<-1.5ex>[rr]^-{\hdash}_-{\tau^*} && \ar@<-1.5ex>[ll]_-{\pi^*}  [\Delta^\mathrm{op},\Set]
}
\end{equation*}
of this adjunction with the Quillen equivalence of Theorem \ref{finalqe1} is the identity adjunction. Hence the result follows from the two-out-of-three property for Quillen equivalences.
\end{proof}

We now prove the main theorem of this section. Since, for each simplicial set $X$, the underlying simplicial set of the $\Theta_2$-set $D_*(X)$ is none other than $X$ itself, there is a canonical inclusion of $\Theta_2$-sets $\pi^*(X) \lra D_*(X)$.

\begin{theorem} \label{finalfinalthm}
Let $X$ be a quasi-category. Then the inclusion $\pi^*(X) \lra D_*(X)$ exhibits the $2$-quasi-category $D_*(X)$ as a fibrant replacement of $\pi^*(X)$ in the model structure for $2$-quasi-categories.
\end{theorem}
\begin{proof}
Since $X = D^*(\pi^*(X))$ is a quasi-category, the inclusion $\pi^*(X) \lra D_*(X)$ is a component of the derived unit of the Quillen equivalence $D^* \dashv D_*$ of Theorem \ref{finalqe1}, and is therefore a weak equivalence in the model structure for locally Kan $2$-quasi-categories. But $\pi^*(X)$ and $D_*(X)$ are locally groupoidal, the one by Proposition \ref{locgrpprop2}, and the other by Theorem \ref{finalqe1}. 
Hence the inclusion $\pi^*(X) \lra D_*(X)$ is a weak equivalence in the model structure for $2$-quasi-categories by Proposition \ref{localwes}. Since $D_*(X)$ is a $2$-quasi-category by Theorem \ref{finalqe1}, this completes the proof.
\end{proof}

We conclude this section by applying Theorem \ref{finalfinalthm} to prove that the functor (\ref{locpi}) defines a full embedding of the homotopy category of quasi-categories into the homotopy category of $2$-quasi-categories, with essential image the locally Kan $2$-quasi-categories.

\begin{corollary} \label{finalcor}
The functor $\Ho(\pi^*) \colon \Ho(\qcat) \lra \Ho(\twoqcat)$ is fully faithful, and its essential image consists of the locally Kan $2$-quasi-categories. 
\end{corollary}
\begin{proof}
By Theorem \ref{finalqe1}, the right derived functor $\mathbb{R}(D_*) \colon \Ho(\qcat) \lra \Ho(\twoqcat)$ of the right Quillen functor $D_* \colon [\Delta^\mathrm{op},\Set] \lra [\Theta_2^\mathrm{op},\Set]$ (with respect to the model structures for quasi-categories and $2$-quasi-categories) is fully faithful, and its essential image consists of the locally Kan $2$-quasi-categories. But 
Theorem \ref{finalfinalthm} implies that the functor $\Ho(\pi^*) \colon \Ho(\qcat) \lra \Ho(\twoqcat)$ is naturally isomorphic to this right derived functor, which proves the result.
\end{proof}

\end{document}